\documentclass[11pt,reqno]{amsart}

\usepackage{enumerate}
\usepackage{amsmath, amssymb, amsthm}
\usepackage{mathrsfs}
\usepackage{esint}
\usepackage{xcolor}
\usepackage{mathtools}
\usepackage{hyperref}
\usepackage{bm}

\DeclareMathOperator\supp{supp}

\newtheorem{thm}{Theorem}[section]
\newtheorem{corollary}[thm]{Corollary}
\newtheorem{lem}[thm]{Lemma}
\newtheorem{prop}[thm]{Proposition}
\theoremstyle{definition}
\newtheorem{defn}[thm]{Definition}

\theoremstyle{remark}
\newtheorem{rem}[thm]{Remark}

\newcommand\bC{\mathbb{C}}

\newcommand\bN{\mathbb{N}}

\newcommand\bR{\mathbb{R}}

\newcommand\bZ{\mathbb{Z}}

\newcommand\fR{\mathbb{R}}

\newcommand\fH{\mathbf{H}}

\newcommand\cD{\mathcal{D}}
\newcommand\cF{\mathcal{F}}

\newcommand\cM{\mathcal{M}}

\newcommand\cS{\mathcal{S}}
\newcommand\cT{\mathcal{T}}
\newcommand\cU{\mathcal{U}}

\DeclareMathOperator*{\esssup}{ess\,sup}
\DeclareMathOperator*{\essinf}{ess\,inf}

\newcommand{\mysection}[1]{\section{#1}
\setcounter{equation}{0}}

\begin{document}

\title[An existence and uniqueness result to evolution equations with P-D operators]
{An existence and uniqueness result to evolution equations with sign-changing pseudo-differential operators and its applications to logarithmic Laplacian operators and second-order differential operators without ellipticity}

\author[J.-H. Choi]{Jae-Hwan Choi}
\address[J.-H. Choi]{School of Mathematics, Korea Institute for Advanced Study, 85 Hoegiro Dongdaemun-gu, Seoul 02455, Republic of Korea}
\email{jhchoi@kias.re.kr}

\author[I. Kim]{Ildoo Kim}
\address[I. Kim]{Department of Mathematics, Korea University, 145 Anam-ro, Seongbuk-gu, Seoul, 02841, Republic of Korea}
\email{waldoo@korea.ac.kr}

\thanks{J.-H. Choi has been supported by a KIAS Individual Grant (MG102701) at Korea Institute for Advanced Study.
I. Kim has been supported by the National Research Foundation of Korea(NRF) grant funded by the Korea government(MSIT) (No.2020R1A2C1A01003959)}

\subjclass[2020]{35S05, 35S10, 35A01, 35D30, 47G30}

\keywords{Pseudo-differential operators with sign-changing symbols, Fourier transform beyond tempered distributions, logarithmic Laplacian, Cauchy problem without ellipticity, Weighted Bessel potential spaces}

\maketitle
\begin{abstract}
We broaden the domain of the Fourier transform to contain all distributions without using the Paley-Wiener theorem and devise a new weak formulation built upon this extension.
This formulation is applicable to evolution equations involving pseudo-differential operators, even when the signs of their symbols may vary over time.
Notably, our main operator includes the logarithmic Laplacian operator $\log (-\Delta)$ and a second-order differential operator whose leading coefficients are not positive semi-definite.
\end{abstract}

\mysection{Introduction}
\label{24.03.29.12.56}

Pseudo-differential operators emerge as a mathematical generalization of differential operators. 
An illustrative form of a pseudo-differential operator is expressed as
$$
P(x,-i\nabla)u(x) = \cF^{-1}\left[ \sigma(x,\xi)  \cF[u](\xi)\right](x),
$$
where $\cF$ and $\cF^{-1}$ denote the Fourier and the inverse Fourier transforms on $\bR^d$, respectively. 
In general, pseudo-differential operators lack local properties, which make conventional perturbation methods from partial differential equation theories inapplicable. Consequently, studying the well-posedness of the equation
\begin{align}
						\label{elliptic symbol equation}
P(x,-i\nabla)u(x)=f(x)
\end{align}
within a specific framework becomes highly challenging without imposing very strong assumptions on both the symbol $\sigma(x,\xi)$ and the data $f$.
However, if one considers pseudo-differential operators which are independent of the space variable $x$, 
there might be intriguing theories for addressing a general symbol $\sigma(\xi)$ and data $f$ within a weak formulation.
Specifically, when considering a polynomial symbol $\sigma(\xi)$, there exist exceptionally elegant theories regarding the existence and uniqueness of solutions to (elliptic) equations like \eqref{elliptic symbol equation} in the weak sense. Remarkably, these theories hold even in the absence of strong mathematical conditions on $\sigma$ and $f$.
In particular, a solution $u$ exists in the sense of distributions if the symbol $\sigma$ is a polynomial, $\sigma(\xi) \not\equiv 0$,  and $f$ is a distribution with a compact support in $\bR^d$.
Even more surprisingly, there always exists a distribution solution $u$ to \eqref{elliptic symbol equation} for any distribution $f$ if the symbol $\sigma$ is a non-zero polynomial, which is shown based on the Paley-Wiener theorem and duality arguments.
This existence result holds even though a nonzero polynomial symbol $\sigma(\xi)$ does not satisfy an ellipticity condition, since the set of all zeros of a non-zero polynomial has a lower dimension.
However, the uniqueness of a weak solution $u$ is not guaranteed if there is no ellipticity condition.
Thus ellipticity is necessary and sufficient to guarantee both existence and uniqueness of a weak solution to \eqref{elliptic symbol equation} even for a polynomial symbol which is independent of $x$.
Moreover, this well-posedness theory can be generalized to accommodate non-polynomial symbols $\sigma$ and tempered distributions $f$ by imposing supplementary conditions involving growth, regularity, and ellipticity on the symbol $\sigma$. 
For more in-depth information, we recommend that the reader consult the books \cite{Hormander 1990,Hormander 1990-2,Komech 1994,Krylov 1996,Krylov 2008}.
It is noteworthy that for a polynomial symbol $\sigma(\xi)$, the sufficient condition to ensure the weak existence of \eqref{elliptic symbol equation} with $f$ having a distribution with compact support is simply that $\sigma(\xi)$ is not identically zero. 
This condition is feasible through a translation of the polynomial $\sigma(\xi)$ and using properties of convolution and analytic functions (see, for instance, \cite[Theorem 7.3.10]{Hormander 1990} and \cite[Theorem 10.3.1]{Hormander 1990-2}).
However, this approach may not be applicable for non-polynomial symbols, particularly in scenarios where the symbols lack regularity. In such scenarios, an ellipticity condition on $\sigma$ becomes crucial, as it prompts the consideration of the reciprocal $\frac{1}{\sigma}$ to derive a fundamental solution (Green's function).

On the other hand, when we consider the evolutionary (or parabolic) counterparts of these theories, the reciprocal $\frac{1}{\sigma}$ could not emerge as a factor in acquiring a fundamental solution if the Fourier transforms are taken solely with respect to the space variable.
Moreover, this approach to take the Fourier transforms only with respect to the spatial variable facilitates examinations of time-dependent symbols.
Certainly, at least in a formal sense, a fundamental solution to the following simple second-order evolution equation
\begin{align*}
\partial_t u(t,x) = a^{ij}(t) u_{x^ix^j}(t,x), \quad u(0)=\delta_0
\end{align*}
is given by 
\begin{align}
								\label{second funda}
\cF_{\xi}^{-1} \left[   \exp\left( -\int_0^t a^{ij}(s) \xi^i\xi^j \mathrm{d}s \right)    \right],
\end{align}
where $\delta_0$ is the Dirac delta measure centered at the origin and $\cF_{\xi}^{-1}$ denotes the inverse Fourier transform with respect to the variable $\xi$. 
Similarly, we could claim that a solution $u$ to 
\begin{align}
										\notag
&\partial_tu(t,x)=\cF^{-1}\left[ \sigma(t,\xi)  \cF[u(t,\cdot)](\xi) \right]\\
									\label{simple pseudo}
&u(0,x)=\delta_0
\end{align}
is given by 
\begin{align*}
u(t,x)=\cF_{\xi}^{-1} \left[   \exp\left( \int_0^t \sigma(s,\xi) \mathrm{d}s \right)    \right](x).
\end{align*}
This leads us to ponder whether ellipticity can be eliminated in a weak formulation of evolution equations for general symbols.
We found that providing a positive answer to this question is indeed valid when investigating evolution equations through a novel weak formulation.
This paper is an outcome of developing a suitable weak formulation to eliminate all elliptic conditions from main operators in evolution equations even though symbols are not regular (thus definitely are not polynomials). 
This new weak formulation relies entirely on an extension of the Fourier and inverse Fourier transforms. 
Traditionally, the class of all tempered distributions has been recognized as an extensive class suitable for these transforms since both transforms become automorphisms on this class. 
However, there exists a remarkable method to extend the domain of the Fourier transform to include all distributions.
This is possible due to the Paley-Wiener theorem. 
More precisely, the Fourier transform of a distribution on $\bR^d$ becomes an element of the dual space of the subspace of entire functions $F$ so that for any multi-index $\alpha$ and $N>0$,
\begin{align}
								\label{Paley class}
|D^\alpha F(z)| \leq C (1+|z|)^{-N} \mathrm{e}^{B|\Im(z)|} \quad \forall z \in \bC^d,
\end{align}
where $C$ and $A$ are positive constants depending on $\alpha$ and $N$, and $\Im(z)$ denotes the imaginary part of the complex vector $z$
(\textit{cf}. \cite[Proposition 4.1]{Komech 1994}).
Nonetheless, this extension of the Fourier transform is not appropriate to deal with initial and inhomogeneous data on the Euclidean spaces
without having analytic continuations.
Thus we propose a new approach which enables the Fourier and inverse Fourier transforms to apply for all distributions without considering analytic continuations.
We will provide a straightforward explanation of these extensions later in the introduction, following the presentation of the main equations. Alternatively, you can examine the comprehensive formulations with all the specifics in Definitions \ref{defn fourier distribution}, \ref{fourier defn 1}, and \ref{fourier defn 2} in the upcoming section.

Now, we succinctly present our equations, which naturally arise as a generalization of \eqref{simple pseudo} :
\begin{equation}
\begin{cases}
								\label{ab eqn}
\partial_tu(t,x)=\psi(t,-i\nabla)u(t,x)+f(t,x),\quad &(t,x)\in(0,T)\times\mathbb{R}^d,\\
u(0,x)=u_0,\quad & x\in\mathbb{R}^d.
\end{cases}
\end{equation}
Here the main operator $\psi(t,-i\nabla)$ is given by
\begin{align}
								\label{main operator}
\psi(t,-i\nabla)u(t,x):=\cF^{-1}\left[\psi(t,\cdot)\cF[u(t,\cdot)] \right](x)
\end{align}
and called a \textit{time-measurable pseudo-differential operator},
where $\psi(t,\xi)$ represents a complex-valued (Borel) measurable function defined on $(0,T) \times \mathbb{R}^d$ and is commonly referred to as a symbol. 
To maintain consistency with the notations used in our previous findings, we substitute $\psi(t,\xi)$ for $\sigma(t,\xi)$. 
This adjustment is made since $\sigma$ conventionally represents the coefficients associated with random noise terms in the realm of stochastic partial differential equations (\textit{cf}. \cite{Krylov 1999}).

Numerous attempts have been undertaken to ease the conditions on symbols while still ensuring the well-posedness of partial differential equations that involve these time-measurable pseudo-differential operators.
It is reasonable to anticipate that the symbol may not require a regularity condition concerning the time variable $t$, as evidenced by many results pertaining to second-order parabolic equations and diffusion processes (\emph{cf}. \cite{Krylov 2008, Stroock Varadhan 1997}). However, a regularity condition on $\psi(t, \xi)$ with respect to $\xi$ appears to be essential for defining the operator $\psi(t, -i\nabla)$ in a strong space, such as an $L_p$-space or a H\"older space, as indicated by multiplier theories in Fourier analysis (\emph{cf}. \cite{Grafakos 2014, Grafakos 2014-2}).
One of the most famous examples is Mihlin's multiplier and its evolutionary counterpart is given by
\begin{align}
							\label{regular symbol}
|D_{\xi}^\alpha \psi(t,\xi)| \lesssim |\xi|^{\nu - |\alpha|} \quad \forall (t,\xi) \in (0,T) \times \bR_0^d
\end{align}
for all multi-index $\alpha$ so that $|\alpha| \leq \lfloor \frac{d}{2}  \rfloor  + 1$, where
$\bR_0^d = \bR^d \setminus \{0\}$, $\nu$ is a positive constant representing an order of the operator, and $\lfloor \frac{d}{2} \rfloor$ is the largest integer which is less than or equal to $\frac{d}{2}$.

 Furthermore, a regularity condition on $\psi(t, \xi)$ alone is insufficient to ensure a strong solution to \eqref{ab eqn}, as observed in the case of a fundamental solution for the second-order equation in \eqref{second funda}. In essence, the term $\exp\left( - \int_0^ta^{ij}(s)\xi^i\xi^j \mathrm{d}s\right)$ in \eqref{second funda} may exhibit exponential growth if there is no elliptic condition on the coefficients $a^{ij}$. 
Thus, we introduce the following strong elliptic condition: there exists a positive constant $\nu>0$ such that
\begin{align}
							\label{strong elliptic}
-\Re[\psi(t,\xi)] \geq \nu |\xi|^\gamma \quad \forall \xi \in \mathbb{R}^d,
\end{align}
where $\Re[\psi(t,\xi)]$ represents the real part of the complex number $\psi(t,\xi)$.

Drawing upon the criteria outlined in \eqref{regular symbol} and \eqref{strong elliptic}, extensive research conducted by the authors and collaborators have consistently shown robust well-posedness for the equation \eqref{ab eqn}.
Examples of such research can be found in \cite{CJH KID 2023,CJH LJB KID 2023,CJH 2024,KID SBL KHK 2015, KID SBL KHK 2016,KID KHK 2016,KID 2018}.
These operators exhibit a close connection to non-local operators and serve as generators of stochastic processes. 
Additionally, for a discussion on the strong well-posedness of these equations with non-local operators and generators of stochastic processes in $L_p$-spaces, we direct readers to the references \cite{CJH KJH PDH 2024,CJH KID 2023-1,HD YL 2023,STC2024,KJH PDH 2023,KJH PDH 2023-1,KID KKH KPK 2019,KKH PDH RJH 2021,RM CP 2017,RM CP 2019,XZ 2013}.

Historically, our attention has been primarily directed towards investigating the strong well-posedness in $L_p$ or H\"older spaces. 
This emphasis was rooted in a longstanding belief that the uniform ellipticity, as expressed in \eqref{strong elliptic}, was indispensable and could not be eliminated, even when considering weak solutions to equations like \eqref{ab eqn}.
The reasoning seems clear and uncomplicated.
If we examine the quantity $\exp[\psi(t, \xi)]$ without an ellipticity, then it becomes apparent that this function could have exponential growth in general, hindering it from being a tempered distribution on $\mathbb{R}^d$. 
Thus, the inverse Fourier transform of the function $\exp(\psi(t, \xi))$ loses its meaningfulness, particularly when considered as a fundamental solution to \eqref{ab eqn}, even in a weak sense if we do not consider extensions of equations to $\bC^d$ after taking the Fourier transform.
Moreover, it is impossible to find a nice analytic continuation of $\psi(t,\xi)$ if there is no regularity condition on $\psi(t,\xi)$ at all.
Hence, it becomes crucial to ensure that the real part of the symbol $\psi(t,\xi)$ remains non-positive in order to guarantee the existence of a fundamental solution even as a tempered distribution-valued function.
In other words, the strong elliptic condition outlined in \eqref{strong elliptic} (or non-positivity of the symbol) seems to be a suitable requirement, even in the context of a weak formulation.
Therefore, initially we were inclined to believe that a weak well-posedness theory for \eqref{ab eqn} with a strong elliptic condition \eqref{strong elliptic} might not be appealing for research, even if regularity conditions such as \eqref{regular symbol} could be eliminated in the weak formulation.
However, we recently discovered a method to extend the domain of the Fourier and the inverse Fourier transforms to all distributions on $\mathbb{R}^d$ without using the Paley-Wiener theorem. 
In particular, it becomes possible to assign a certain meaning to the term $\cF^{-1}\left[\exp(\psi(t,\xi)) \right]$ instead of an element in the dual space of entire functions $F$ on $\bC^{d}$ so that \eqref{Paley class} holds, even when the sign of the real part of the symbol $\psi(t,\xi)$ is strictly positive. 
With these extensions, we ultimately achieve a satisfactory weak well-posedness result for \eqref{ab eqn} without imposing any ellipticity and regularity conditions. A certain local integrability on the symbol $\psi(t,\xi)$ is sufficient to establish weak well-posedness of \eqref{ab eqn}. 
We emphasize that not only the existence, but also the uniqueness of a weak solution is obtained even though a weak solution to equation \eqref{elliptic symbol equation} is not unique even for a polynomial symbol $\sigma(\xi)$ without ellipticity.
For further details, see Theorem \ref{weak solution thm} below.

Next, we offer a brief overview of a fundamental idea behind our weak formulation. 
It is essential to recognize that trying to create a weak formulation across the entire domain $(0,T) \times \mathbb{R}^d$ is unfeasible due to the fact that a symbol $\psi(t,\xi)$ varies with respect to the time.
Additionally, $(0,T)$ could be a finite interval where Schwartz's functions do not work effectively.
As a result, our testing actions for the equations are limited to only the spatial variables.
The core aspect of our approach lies in utilizing a different class of test functions.
Specifically, we use $\cF^{-1}\cD(\bR^d)$ as a class of test functions instead of $\cD(\bR^d)$.
Here, $\cF^{-1}\cD(\bR^d)$ represents a subset of Schwartz functions whose Fourier transforms are belonging to the class $\cD(\bR^d)$ and $\cD(\bR^d)$ is the class of all functions that are infinitely differentiable and have compact supports.
We establish the Fourier transform and its inverse for a distribution $v$ defined on $\bR^d$ as linear functionals acting on $\cF^{-1}\cD(\bR^d)$. This is achieved by applying straightforward operations on $\cF^{-1}\cD(\bR^d)$, which are motivated by Plancherel's theorem. 
In other words, we define
\begin{align}
							\label{new inverse Fourier}
\langle \cF^{-1}[v] , \varphi \rangle = \langle v , \cF[\varphi] \rangle \quad \forall \varphi \in \cF^{-1}\cD(\bR^d).
\end{align}
and
\begin{align*}
\langle \cF[v] , \varphi \rangle = \langle v , \cF^{-1}[\varphi] \rangle \quad \forall \varphi \in \cF^{-1}\cD(\bR^d).
\end{align*}
Specifically, we have the option to consider the Fourier transform and its inverse for any function which is locally integrable.

Now, suppose that we have a solution $u$ to equation \eqref{ab eqn} such that the Fourier transform of $u$ has a realization on $\bR^d$. 
Furthermore, we also assume that the product of $\psi(t,\xi)$ and the Fourier transform of $u$ with respect to $\xi$ is locally integrable. 
Then, the function referred to in \eqref{main operator} becomes meaningful due to the revised definition of the inverse Fourier transform presented in \eqref{new inverse Fourier}.

To conclude, we can define $u$ as a weak solution to \eqref{ab eqn} if, for every $\varphi$ belonging to the space $\cF^{-1}\cD(\bR^d)$, the following equation holds:
\begin{align*}
\left\langle u(t,\cdot),\varphi \right\rangle 
= \langle u_0, \varphi \rangle +  \int_0^t \langle \psi(s,-i\nabla)u(s,\cdot) , \varphi \rangle \mathrm{d}s 
+ \int_0^t \left\langle f(s,\cdot),\varphi\right\rangle \mathrm{d}s
\quad (a.e.)~t\in (0,T).
\end{align*}
This straightforward concept enables us to naturally manage the exponential growth that may arise when the real part of the symbol $\psi(t,\xi)$ is sign-changing.

It is worth noting that we do not rely on any arguments related to weak pre-compactness to derive our main theorems. Instead, our results are solely grounded in elementary techniques, such as H\"older and Minkowski inequalities, combined with well-known properties of the Fourier and inverse Fourier transforms.

This paper comprises seven sections.
Section \ref{24.03.29.12.56} serves as an introductory section, presenting an overview of the results. 
Our primary theorems are introduced in Section \ref{main section}.
In Section \ref{fourier section}, we revisit the classical Bessel potential spaces and extend them to incorporate weighted variations.
Section \ref{24.03.29.12.58} is dedicated to establishing the uniqueness of a solution $u$ for \eqref{ab eqn} under a slightly relaxed condition.
The proof of the main theorem is provided in Section \ref{pf main thm 1}, along with noteworthy corollaries for practical applications.
Section \ref{log thm pf} focuses on handling logarithmic operators as a specific application of our theorem, while Section \ref{second thm pf} addresses second-order differential operators without ellipticity.

\vspace{2mm}
We finish this section with the notations used in the article.

\begin{itemize}
\item 
Let $\bN$, $\bZ$, $\bR$, $\bC$ denote the natural number system, the integer number system, the real number system, and the complex number system, respectively. For $d\in\bN$, $\bR^d$ denotes the $d$-dimensional Euclidean space. 
\item 
For $i=1,...,d$, a multi-index $\alpha=(\alpha_{1},...,\alpha_{d})$ with
$\alpha_{i}\in\{0,1,2,...\}$, and a function $g$, we set
$$
\frac{\partial g}{\partial x^{i}}=D_{x^i}g,\quad
D^{\alpha}g=D_{x^1}^{\alpha_{1}}\cdot...\cdot D^{\alpha_{d}}_{x^d}g,\quad |\alpha|:=\sum_{i=1}^d\alpha_i.
$$
For $\alpha_i =0$, we define $D^{\alpha_i}_{x^i} g = g$. 
If $g=g(t,x):\bR\times\bR^d\to\bR^d$, then we denote
$$
 D_x^{\alpha}g=D_{x^1}^{\alpha_{1}}\cdot...\cdot D^{\alpha_{d}}_{x^d}g,
$$
where $\alpha=(\alpha_1,\cdots,\alpha_d)$.

\item We employ the identical symbol ``$\sup$'' to represent both the supremum and the essential supremum
by slightly abusing the notation $\sup$. On occasion, ``$\esssup$'' may be utilized to emphasize its role as the essential supremum.

\item 
We use $\cD(\bR^d)$ to denote the space of all infinitely differentiable functions with compact supports.
$\cS(\bR^d)$ represents the Schwartz space on $\bR^d$ and the topology on $\cS(\bR^d)$ is generated by the Schwartz semi-norms $\sup_{x \in \bR^d} |x^\alpha (D^\beta f)(x)|$ for all multi-indices $\alpha$ and $\beta$.  
$\cS'(\bR^d)$ is used to denote the dual space of $\cS(\bR^d)$, \textit{i.e.} $\cS'(\bR^d)$ is the space of all tempered distributions on $\bR^d$. 
Additionally, we assume that $\cS'(\bR^d)$ is a topological space equipped with the weak*-topology if there is no special remark about the topology on $\cS'(\bR^d)$.

\item 
Let $F$ be a normed space and $(X,\mathcal{M},\mu)$ be a measure space.
\begin{itemize}
    \item $\mathcal{M}^{\mu}$ denotes the completion of $\cM$ with respect to the measure $\mu$.
    \item For $p\in[1,\infty)$, the space of all $\mathcal{M}^{\mu}$-measurable functions $f : X \to F$ with the norm 
\[
\left\Vert f\right\Vert _{L_{p}(X,\cM,\mu;F)}:=\left(\int_{X}\left\Vert f(x)\right\Vert _{F}^{p}\mu(\mathrm{d}x)\right)^{1/p}<\infty
\]
is denoted by $L_{p}(X,\cM,\mu;F)$. We also denote by $L_{\infty}(X,\cM,\mu;F)$ the space of all $\mathcal{M}^{\mu}$-measurable functions $f : X \to F$ with the norm
$$
\|f\|_{L_{\infty}(X,\cM,\mu;F)}:=\inf\left\{r\geq0 : \mu(\{x\in X:\|f(x)\|_F\geq r\})=0\right\}<\infty.
$$
We usually omit the given measure or $\sigma$-algebra in the notations of $L_p$-spaces if there is no confusion (\textit{e.g.} Lebesgue (or Borel) measure and $\sigma$-algebra).
We similarly leave out the representation of $F$ when it takes on the scalar values such as $\mathbb{R}$ or $\mathbb{C}$. 
\end{itemize}

\item Let $f(t)$ and $g(t)$ be Lebesgue measurable (or Borel measurable) functions on $\cU$.
We write
\begin{align*}
f(t)=g(t)~ (a.e.)~ t \in \cU
\end{align*}
if and only if there exists a measurable subset $\cT \subset \cU$ such that the Lebesgue measure $|\cU \setminus \cT|$ is zero and $f(t)=g(t)$ for all $t \in \cT$.
Moreover, we say that a function $f(t)$ is defined $(a.e.)$ on a Lebesgue measurable (or Borel measurable) set $\cU$ if 
 there exists a measurable subset $\cT \subset \cU$ such that the Lebesgue measure $|\cU \setminus \cT|=0$ and $f(t)$ is defined for all $t \in \cT$.
\item
For $R>0$,
$$
B_R(x):=\{y\in\bR^d:|x-y|< R\},\quad
\overline{B_R(x)}:=\{y\in\bR^d:|x-y|\leq R\}
$$

\item 
For a measurable function $f$ on $\bR^d$, we denote the $d$-dimensional Fourier transform of $f$ by 
\[
\cF[f](\xi) := \frac{1}{(2\pi)^{d/2}}\int_{\bR^{d}} \mathrm{e}^{-i\xi \cdot x} f(x) \mathrm{d}x
\]
and the $d$-dimensional inverse Fourier transform of $f$ by 
\[
\cF^{-1}[f](x) := \frac{1}{(2\pi)^{d/2}}\int_{\bR^{d}} \mathrm{e}^{ ix \cdot \xi} f(\xi) \mathrm{d}\xi.
\]
Moreover, for a function $f(t,x)$ defined on $(0,T)\times \bR^d$, we use the notation
\begin{align*}
\cF_x[f](\xi)
:=\cF[f(t,x)](\xi)
:=\cF[f(t,\cdot)](\xi)
=\frac{1}{(2\pi)^{d/2}}\int_{\bR^{d}} \mathrm{e}^{ -i\xi \cdot x} f(t,x) \mathrm{d}x
\end{align*}
and it is called the Fourier transform of $f$ with respect to the space variable.
On the other hand, for the inverse Fourier transform of $f(t,\xi)$ with respect to the space variable, we use the notation
\begin{align*}
\cF_\xi^{-1}[f](x)
:=\cF^{-1}[f(t,\xi)](x)
:=\cF^{-1}[f(t,\cdot)](x)
=\frac{1}{(2\pi)^{d/2}}\int_{\bR^{d}} \mathrm{e}^{  ix \cdot \xi} f(t,\xi) \mathrm{d}\xi.
\end{align*}
For the sake of simplicity, we often omit the subscripts denoted by $x$ and $\xi$.
It is well-known that both the Fourier transform and the inverse Fourier transform have extensions on $L_1(\bR^d)+ L_2(\bR^d)$.
By slightly abusing the notation, we use the same notation $\cF$ and $\cF^{-1}$ to denote these extensions.
We discuss further details and some properties of these extensions and corresponding inversion theorems in Section \ref{fourier section}.
\item 
We write $\alpha \lesssim \beta$ if there is a positive constant $N$ such that $ \alpha \leq N \beta$.
We use $\alpha \approx \beta$ if $\alpha \lesssim \beta$ and $\beta \lesssim \alpha$. 
In particular, we use the notation $\alpha \lesssim_{a,b,\cdots} \beta$  if the constant $N$ so that $\alpha \leq N \beta$  depends only on $a,b,\cdots$.
Moreover, if we write $N=N(a,b,\cdots)$, this means that the
constant $N$ depends only on $a,b,\cdots$. 
A generic constant $N$ may change from a location to a location, even within a line. 
The dependence of generic constants is usually specified in each statement of theorems, propositions, lemmas, and corollaries.
\item 
For $z\in\bC$, $\Re[z]$ denotes the real part of $z$, $\Im[z]$ is the imaginary part of $z$ and $\bar{z}$ is the complex conjugate of $z$.
\end{itemize}

\mysection{Settings and main results}
											\label{main section}

We fix $T \in (0,\infty]$ and $d \in \bN$ throughout the paper. 
Here $T$ and $d$ denote the terminal time of the evolution equation and the dimension of the space-variable, respectively.  
Note that $T= \infty$ is possible in our theory.
All functions are complex-valued if there is no special remark about the range of a function. 
In particular, $\psi(t,\xi)$ denotes a complex-valued measurable function defined $(a.e.)$ on $(0,T) \times \bR^d$ and denote
\begin{align}
									\label{20230522 01}
\psi(t,-i\nabla)u(t,x):=\cF^{-1}\left[\psi(t,\cdot)\cF[u(t,\cdot)]\right](x)
\end{align}
in the whole paper. 
All the Fourier and the inverse Fourier transforms in the paper are taken only with respect to the space variable and thus they are considered only on $\bR^d$ even in a weak sense.

The main goal of this paper is to show the existence and uniqueness of a weak solution $u$ to the Cauchy problem \eqref{ab eqn}. 
Note that it is not easy to recover a strong solution to \eqref{ab eqn} even in $L_p$-spaces if there are no strong mathematical conditions on the symbol and data.
A naive condition is that for each $t \in (0,T)$, $\psi(t,\cdot)\cF[u](t,\cdot) \in L_p(\bR^d)$
 with $p \in [1,2]$. Then the function $t \mapsto \psi(t,-i\nabla)u(t,\cdot)$ is well-defined as an $L_{p'}(\bR^d)$-valued function defined on $(0,T)$  due to the Riesz-Thorin Theorem, where $p'$ denotes the H\"older conjugate of $p$, \textit{i.e.}
\begin{align*}
p' =
\begin{cases}
&\frac{p}{p-1}  \quad \text{if $p \in (1,2]$} \\
&\infty \quad \text{if $p=1$}.
\end{cases}
\end{align*}
We revisit Riesz-Thorin's Theorem in Section \ref{fourier section}, providing more specific details about the constants involved.

On the other hand, 
it appears that there is no need for numerous mathematical conditions when approaching the operator described in \eqref{20230522 01} in a weak sense.
In particular, if $\psi(t,\cdot)\cF[u](t,\cdot) \in \cS'(\bR^d)$, then 
$t \mapsto \psi(t,-i\nabla)u(t,\cdot)$ is well-defined as an $\cS'(\bR^d)$-valued function defined on $(0,T)$, where $\cS'(\bR^d)$ denotes the class of all tempered distributions on $\bR^d$. 
Nevertheless, if the function $\psi(t,\cdot)\mathcal{F}[u]$ is only locally integrable, it does not qualify as a tempered distribution on $\mathbb{R}^d$. Consequently, the operator in \eqref{20230522 01} is not well-defined in this scenario, as it may have exponential growth. However, local integrability of $\psi(t,\cdot)\mathcal{F}[u]$ is sufficient to establish it as a distribution on $\mathbb{R}^d$.
Hence, if we suggest a new technique for expanding the domain of the inverse Fourier transform to encompass all distributions, 
then the operator defined in \eqref{20230522 01} could be meaningful even with only local integrability on $\psi(t,\cdot)\mathcal{F}[u]$. 
This extension is especially important when dealing with functions that exhibit exponential growth.
Additionally, it is well-known that there exists an elegant method to define the Fourier transform of a distribution on $\bR^d$ as an element in the dual space of a subspace of entire functions $F$ on $\bC^{d}$ so that \eqref{Paley class} holds.
However, this extension is not appropriate to fulfill our well-posedness theory since our symbol $\sigma(t,\xi)$ is not a polynomial type in general.
In other words, we cannot consider an analytic continuation of $\sigma(t,\xi)$ on $\bC^d$ for each $t$.
Thus, we suggest a new extension of the Fourier transform to include all distributions below in Definitions \ref{fourier defn 1} and \ref{fourier defn 2}.
Remarkably, we have not come across a suitable reference discussing this type extension of the Fourier transform beyond tempered distributions without using analytic continuations to the best of our knowledge.

Utilization of distribution theories has significantly advanced our understanding of partial differential equations on a broader scale, as noted in \cite{Gelfand 1968,Hormander 1990, Hormander 1990-2}. 
Nevertheless, a conventional weak formulation utilizing distributions appears ineffective in the absence of strong ellipticity on the symbols for our evolution equations due to weak mathematical conditions with respect to the time variable $t$ and possibility having exponential growth as highlighted in the introduction.
More specifically, if we allow the symbol $\psi(t,\xi)$ to change its sign, then our equation \eqref{ab eqn} includes second-order evolution equations without the parabolicity, which clearly shows that classical weak formulations taken for both time and space variables simultaneously, cannot be applicable since the coefficients are depending on the time, $(0,T)$ could be finite interval, and data $u_0$ and $f$ might have not compact supports in general. 
Moreover, exponential growth could emerge naturally as a result of the influence of exponential functions generated by time evolution when examining evolution equations instead of stationary equations if there is no ellipticity.
This occurrence is particularly evident when Fourier transforms are exclusively applied to the spatial variable to accommodate variations in symbols with respect to the time variable.

Thus, it is reasonable to anticipate that the standard framework of weak formulations typically tested with $\cD(\mathbb{R}^d)$ or $\cS(\mathbb{R}^d)$ is insufficient to address well-posedness of our equation \eqref{ab eqn}. 
A typical form of a solution to \eqref{ab eqn} clarifies these reasons.
Indeed, at least in a formal sense, a solution $u$ to \eqref{ab eqn} can be expressed as follows:
\begin{align}
						\label{make sense representation}
u(t,x)=  
 \cF^{-1}\left[  \exp\left(\int_0^t\psi(r,\cdot)\mathrm{d}r \right) \cF[u_0] \right](x)
+\cF^{-1}\left[ \int_0^t  \exp\left(\int_s^t\psi(r,\cdot)\mathrm{d}r \right) \cF[f(s,\cdot)]\mathrm{d}s \right](x).
\end{align}
Hence, if our symbol $\psi(t,\xi)$ shows a sign-changing behavior, it can lead to exponential growth in the frequency of our solution $u$. 
Nonetheless, it is widely recognized that tempered distributions do not offer adequate control over the exponential growth, as has been emphasized multiple times previously.
Thus, it becomes necessary to find a more extensive class than $\cS'(\bR^d)$ to ensure well-posedness of our evolution equation \eqref{ab eqn}, particularly when we permit symbols to change signs.
However, using the class of all distributions as a direct replacement is not suitable because the Fourier transforms of distributions are not universally defined if we do not consider them as elements in a dual space of some entire functions. 
Therefore, we propose a new class referred to as \emph{Fourier transforms of distributions}.
\begin{defn}[Fourier transforms of distributions]
								\label{defn fourier distribution}
We use $\cF^{-1}\cD(\bR^d)$ to denote the subclass of the Schwartz class whose Fourier transform is in $\cD(\bR^d)$, \textit{i.e.}
\begin{align*}
\cF^{-1}\cD(\bR^d) 
:= \{ \varphi \in \cS(\bR^d) : \cF[\varphi] \in \cD(\bR^d)\},
\end{align*}
where $\cS(\bR^d)$ denotes the (complex-valued) Schwartz class on $\bR^d$ and $\cD(\bR^d)$ denotes the class of all complex-valued infinitely differentiable functions with compact supports defined on $\bR^d$.
Recall that there is a topology on $\cS(\bR^d)$ generated by the Schwartz semi-norms. Since $\cF^{-1}\cD(\bR^d) \subset \cS(\bR^d)$, there is a subspace topology on $\cF^{-1}\cD(\bR^d)$.
Therefore, one can consider its dual space. 
We use the notation $\left(\cF^{-1}\cD (\bR^d) \right)'$ to denote the dual space of $\cF^{-1}\cD(\bR^d)$, \textit{i.e.} $u \in \left(\cF^{-1}\cD (\bR^d) \right)'$
if and only if $u$ is a continuous linear functional defined on $\cF^{-1}\cD(\bR^d)$.
For $u \in \left(\cF^{-1}\cD (\bR^d) \right)'$ and $\varphi \in \cF^{-1}\cD(\bR^d)$, we write
\begin{align*}
\langle u, \varphi \rangle := u(\varphi).
\end{align*}
In other words, $\langle u, \varphi \rangle$ denotes the image of $\varphi$ under $u$.
Defining a topology for $\left(\cF^{-1}\cD(\bR^d)\right)'$ is a straightforward task. Specifically, one can adopt the weak*-topology for $\left(\cF^{-1}\cD(\bR^d)\right)'$ and work with its associated Borel sets to mention measurability of functions whose range are in $\left(\cF^{-1}\cD(\bR^d)\right)'$.
\end{defn}
\begin{defn}[Fourier and inverse Fourier transforms of $u \in \left(\cF^{-1}\cD (\bR^d)\right)'$]
							\label{fourier defn 1}
 For $u \in \left(\cF^{-1}\cD (\bR^d)\right)'$, one can define the Fourier and inverse Fourier transforms of $u$ in the following way:
\begin{align*}
\langle \cF[u], \varphi \rangle
:=\langle u, \cF^{-1}[\varphi] \rangle
\end{align*}
and
\begin{align*}
\langle \cF^{-1}[u], \varphi \rangle
:=\langle u, \cF[\varphi] \rangle
\end{align*}
for all $\varphi \in \cD(\bR^d)$.
Thus for any $u \in \left(\cF^{-1}\cD(\bR^d) \right)'$, we have
$\cF[u] \in \cD'(\bR^d)$ and $ \cF^{-1}[u] \in \cD'(\bR^d)$
since it is obvious that both $\cF^{-1}[\varphi]$ and $\cF[\varphi] $ are in $\cF^{-1}\cD(\bR^d)$ for all $\varphi \in \cD(\bR^d)$ due to well-known properties of the Fourier and inverse Fourier transforms of a Schwartz function.
The continuity of $\cF[u]$ and $\cF^{-1}[u]$ on $\cD(\bR^d)$ to be distributions is easily obtained from properties of the Fourier transform and the Schwartz functions as well.
\end{defn}

It is essential to reiterate that, for any $u \in \cD'(\bR^d)$, the conventional definition of its Fourier transform or inverse Fourier transform is not applicable without considering analytic continuations as discussed previously.
Nonetheless, we can ultimately establish a meaningful definition for the Fourier and inverse Fourier transforms of all distributions, treating them as elements in $(\cF^{-1}\cD)'(\bR^d)$ without the Paley-Wiener theorem.

\begin{defn}[Fourier and inverse Fourier transforms of $u \in \cD' (\bR^d)$]
								\label{fourier defn 2}
Let $u \in \cD'(\bR^d)$.
We define the Fourier and inverse Fourier transforms of $u$ as follows:
\begin{align*}
\langle \cF[u], \varphi \rangle 
:=\langle u,  \cF^{-1}[\varphi] \rangle \quad \forall \varphi \in \cF^{-1}\cD(\bR^d) 
\end{align*}
and
\begin{align*}
\langle \cF^{-1}[u], \varphi \rangle 
:=\langle u,  \cF[\varphi] \rangle  \quad \forall \varphi \in \cF^{-1}\cD(\bR^d).
\end{align*}
It is obvious that both transforms above are well-defined since
$\cF[\varphi]$ and $\cF^{-1}[\varphi]$ are in $\cD(\bR^d)$ for all $\varphi \in \cF^{-1}\cD(\bR^d)$. 
\end{defn}
 
Due to Definitions \ref{fourier defn 1} and \ref{fourier defn 2}, any distribution $u$ on $\bR^d$ becomes the Fourier transform (or inverse Fourier transform) of an element in $\left(\cF^{-1}\cD(\bR^d)\right)'$ and any element in $\left(\cF^{-1}\cD(\bR^d)\right)'$ is the Fourier transform (or inverse Fourier transform) of a distribution on $\bR^d$.
In this sense, we could say that $\left(\cF^{-1}\cD(\bR^d)\right)'$ is the space of all Fourier transforms or all inverse Fourier transforms of distributions on $\bR^d$.
Thus we use the notation $\cF^{-1}\cD'(\bR^d)$ or $\cF\cD'(\bR^d)$ instead of $\left(\cF^{-1}\cD(\bR^d)\right)'$.

Besides, recall that both Fourier and inverse Fourier transforms are homeomorphisms on the class of all tempered distributions with respect to weak*-topologies.
Since $\cD(\bR^d) \subset \cS(\bR^d)$ and $\cF^{-1}\cD(\bR^d) \subset \cS(\bR^d)$, it is obvious that both spaces $\cD'(\bR^d)$ and $\cF^{-1}\cD'(\bR^d)$ are lager than $\cS'(\bR^d)$. 
We state below that both transforms still become homeomorphisms with respect to weak*-topologies.
We do not give a detailed proof since the proof is almost identical with that of the tempered distributions
 based on properties of transforms for functions in the Schwartz class (\textit{cf}. \cite[Section 2.2.3]{Grafakos 2014} and \cite[Theorem 7.1.10]{Hormander 1990}). 

\begin{thm}
\begin{enumerate}[(i)]
\item Both Fourier and inverse Fourier transforms are homeomorphisms from $\cD'(\bR^d)$ onto $\cF^{-1}\cD'(\bR^d)$ with respect to weak*-topologies.
\item Both Fourier and inverse Fourier transforms are homeomorphisms from $\cF^{-1}\cD'(\bR^d)$ onto $\cD'(\bR^d)$ with respect to weak*-topologies.
\end{enumerate}
In total, both transforms $\cF$ and $\cF^{-1}$ are mappings from $\cD'(\bR^d) \cup \cF^{-1}\cD'(\bR^d)$ onto $\cD'(\bR^d) \cup \cF^{-1}\cD'(\bR^d)$.
\end{thm}
Therefore, we could freely apply the Fourier inversion formula to all distributions and Fourier transforms of distributions on $\bR^d$ due to the above theorem. In particular, $\cF[u]=\cF[v]$ implies $u=v$ for all distributions $u$ and $v$ on $\bR^d$.

It is crucial to grasp how distributions are manifested in practice. 
In our main theorem, the mathematical conditions are given by realizations of Fourier transforms of distributions.
Recall that any locally integrable function $v$ on $\bR^d$ is a distribution on $\bR^d$ due to the identification $v$ with the mapping
$$
\varphi \in \cD(\bR^d) \mapsto \int_{\bR^d} v(x) \overline{\varphi(x)} \mathrm{d}x.
$$ 
Here  $v$ is locally integrable on $\bR^d$ if and only if
\begin{align*}
\int_{|\xi|<R}v(\xi) \mathrm{d} \xi < \infty \quad \forall R \in (0,\infty),
\end{align*}
where $B_R$ denotes the Euclidean ball whose center is zero and radius is $R$, \textit{i.e.}
\begin{align*}
B_R:= \{y \in \bR^d: |y| < R\}.
\end{align*}
We write $ v \in L_{1,\ell oc}(\bR^d)$ if $v$ is locally integrable on $\bR^d$.
Especially, if there exists a $v \in L_{1,\ell oc}(\bR^d)$ such that
\begin{align*}
\langle \cF[u], \varphi \rangle 
= \left(v, \varphi  \right)_{L_2(\bR^d)}
:= \int_{\bR^d} v(\xi) \overline{\varphi(\xi)} \mathrm{d}\xi \quad \forall \varphi \in \cD(\bR^d),
\end{align*}
then we identify $\cF[u] = v$ and consider $\cF[u]$ as a function defined $(a.e.)$ on $\bR^d$,
where $\overline{\varphi(\xi)}$ denotes the complex conjugate of $\varphi(\xi)$. 
In this case, we usually say that the distribution $\cF[u]$ has the realization $v$ on $\bR^d$. 
This identification even works for any element $u \in \cF^{-1}\cD'(\bR^d)$.
In other words, we say that $u \in \cF^{-1}\cD'(\bR^d)$ is locally integrable if and only if there is a $ v \in L_{1,\ell oc}(\bR^d)$
\begin{align*}
\langle u, \varphi \rangle 
:= \int_{\bR^d} v(\xi) \overline{\varphi(\xi)} \mathrm{d}\xi \quad \forall \varphi \in \cF^{-1}\cD(\bR^d).
\end{align*}
This function $v$ is uniquely determined $(a.e.)$ on $\bR^d$ since both $\cD(\bR^d)$ and $\cF^{-1}\cD(\bR^d)$ are dense in $L_\infty(B_R)$ for all $R \in (0,\infty)$ (by considering restrictions of $\cD(\bR^d)$ or $\cF^{-1}\cD(\bR^d)$ on each $B_R$).
Due to this identification, we could introduce an important subclass of $\cF^{-1}\cD'(\bR^d)$ consisting of all functions whose Fourier transform is locally integrable.
Here is the precise definition.
We use $\cF^{-1}L_{1,\ell oc}(\bR^d)$ to denote the subspace of $\cF^{-1} \cD'(\bR^d)$ whose Fourier transform is in $L_{1,\ell oc}(\bR^d)$, \textit{i.e.} $ u \in \cF^{-1}L_{1,\ell oc}(\bR^d)$ if and only if
\begin{align*}
\left\langle u, \varphi \right\rangle
=\left\langle \cF[u], \cF[\varphi] \right\rangle
= \int_{\bR^d} \cF[u](\xi) \overline{\cF[\varphi](\xi)} \mathrm{d}\xi \quad \forall \varphi \in \cF^{-1}\cD(\bR^d).
\end{align*}
Since $\cF^{-1}L_{1,\ell oc}(\bR^d) \subset \cF^{-1}\cD'(\bR^d)$, there is the subspace topology on $\cF^{-1}L_{1,\ell oc}(\bR^d)$ and the Borel sets generated by this topology. 
Additionally, this identification with a locally integrable function could be easily performed in any domain of $\bR^d$ 
due to the localizations of distributions. 
More precisely, if there exist an open set $\cU \subset \bR^d$ and a locally integrable function $v$ on $\cU$ so that 
\begin{align}
									\label{20240319 01}
\langle \cF[u] , \varphi \rangle = \int_{\fR^d} v(\xi) \overline{\varphi(\xi)} \mathrm{d}\xi \quad \forall \varphi \in \cD(\cU),
\end{align}
then we identify $\cF[u]$ with $v$ on $\cU$, where $\cD \left(\cU\right)$ denotes the set of all infinitely differentiable functions on $\bR^d$ with compact supports in $\cU$. In such a case, we say that $\cF[u]$ has a realization $v$ on $\cU$ and consider $\cF[u]$ as a locally integrable function on $\cU$. 
In other words, for any $u$ in $\cF^{-1}\cD'(\fR^d)$, we say that  $\cF[u]$ has a realization $v$ on $\cU$ 
or simply $\cF[u]$ is locally integrable on $\cU$ if \eqref{20240319 01} holds.

According to these realizations of elements in $\cF^{-1}\cD'(\bR^d)$, we are ready to define the operator $\psi(t,-i \nabla)$.
To define the operator universally for all $t$, we slightly abuse the notation by considering $\psi(-i\nabla)$ with a function $\psi(\xi)$ on $\bR^d$ below.
\begin{defn}
							\label{defn psi operator}
Let $g \in \cF^{-1}\cD'(\bR^d)$ so that $\cF[g]$ has a realization on an open set $\cU$ covering the support of $\psi$ and
$\xi \in \bR^d \mapsto \psi(\xi)\cF[g](\xi) \in \bC$ is a locally integrable function on $\bR^d$.
Then we define $\psi(-i\nabla)g$ as an element in $\cF^{-1}\cD'(\bR^d)$ so that
\begin{align*}
\langle \psi(-i \nabla)g, \varphi \rangle
=\int_{\fR^d} \psi(\xi)\cF[g](\xi) \cF[\varphi](\xi) \mathrm{d}\xi \quad \forall \varphi \in \cF^{-1}\cD'(\bR^d).
\end{align*}
\end{defn}
Due to the definition provided above, equation \eqref{ab eqn} can be considered valid even without imposing any regularity requirement on the symbol.
This will be further elaborated upon in the subsequent discussion regarding the definition of a solution, as outlined with more depth in Definition \ref{space weak solution}.

Next we consider classes comprising Borel measurable functions with values in $\cF^{-1}\cD'(\bR^d)$, defined almost everywhere on the interval $(0,T)$. 
\begin{defn}
								\label{local l1 fd valued space}
We use the notation $L_{0}\left((0,T);\cF^{-1}\cD'(\bR^d)\right)$ to represent the set of all Borel measurable functions $f$ with values in $\cF^{-1}\cD'(\bR^d)$ defined almost everywhere on the interval $(0,T)$.
In particular, $f \in L_{0}\left((0,T);\cF^{-1}\cD'(\bR^d)\right)$ implies that 
\begin{align*}
\langle f(t,\cdot), \varphi \rangle :=\langle f(t), \varphi \rangle < \infty \quad (a.e.)~t \in (0,T) \quad \text{and} \quad \forall \varphi \in \cF^{-1}\cD(\bR^d).
\end{align*}
\end{defn}

\begin{defn}
$L_{1,t\text{-}loc}\left((0,T);\cF^{-1}\cD'(\bR^d)\right)$ is used to denote the class of all Borel measurable functions $f$ with values in $\cF^{-1}\cD'(\bR^d)$ defined $(a.e.)$ on $(0,T)$ such that
\begin{align*}
\int_0^t \langle f(s,\cdot), \varphi \rangle \mathrm{d}s:=\int_0^t \langle f(s), \varphi \rangle \mathrm{d}s < \infty \quad \forall t \in (0,T)\quad \text{and} \quad \forall \varphi \in \cF^{-1}\cD(\bR^d).
\end{align*}
\end{defn}
It is obvious that $L_{1,t\text{-}loc}\left((0,T);\cF^{-1}\cD'(\bR^d)\right) \subset L_{0}\left((0,T);\cF^{-1}\cD'(\bR^d)\right)$.
Naturally, we could consider subspaces of $L_{0}\left((0,T);\cF^{-1}\cD'(\bR^d)\right)$ and $L_{1,t\text{-}loc}\left((0,T);\cF^{-1}\cD'(\bR^d)\right)$ with values in $\cF^{-1}L_{1, \ell oc}(\bR^d)$.
\begin{defn}
We define $L_{0}\left((0,T);\cF^{-1}L_{1,\ell oc}(\bR^d)\right)$ as the subspace of $L_{0}\left((0,T);\cF^{-1}\cD'(\bR^d)\right)$ consisting of all Borel measurable functions $f$ with values in $\cF^{-1}L_{1,\ell oc}(\bR^d)$ defined $(a.e.)$ on $(0,T)$ so that
\begin{align*}
 \int_{B_R} |\cF[f(t,\cdot)](\xi)|\mathrm{d}\xi
 := \int_{B_R} |\cF[f(t)](\xi)|\mathrm{d}\xi < \infty \quad (a.e.)~t \in (0,T) \quad \text{and} \quad \forall R \in (0,\infty).
\end{align*}
\end{defn}
\begin{defn}
								\label{local 11 space}
We write $ f \in L_{1,t\text{-}loc}\left((0,T);\cF^{-1}L_{1,\ell oc}(\bR^d)\right)$ if $f$ is a $\cF^{-1}L_{1,\ell oc}(\bR^d)$-valued Borel measurable function  defined $(a.e.)$ on $(0,T)$ such that
\begin{align*}
\int_0^t  \int_{B_R} |\cF[f(s, \cdot)](\xi)|\mathrm{d}\xi \mathrm{d}s 
:=\int_0^t  \int_{B_R} |\cF[f(s)](\xi)|\mathrm{d}\xi \mathrm{d}s 
< \infty \quad \forall (t,R) \in (0,T)\times (0,\infty).
\end{align*}
\end{defn}
Generally, $f \in L_{1,t\text{-}loc}\left( (0,T) ;  \cF^{-1}L_{1, \ell oc}(\bR^d) \right)$ cannot be identified with a complex-valued function defined on $(0,T) \times \fR^d$ in the typical almost everywhere sense. 
Thus the terms $\cF[f(s, \cdot)]$ and $\cF[f(s, \cdot)](\xi)$ should be understood as the notation to denote $\cF[f(s)]$ and $\cF[f(s)](\xi)$, respectively. On the other hand, if $f$ is a nice function defined on $(0,T)\times \fR^d$, then
one can understand $\cF[f(s)](\xi)$ as
\begin{align*}
\cF[f(s, \cdot)](\xi) := \frac{1}{(2\pi)^{d/2}} \int_{\bR^d} \mathrm{e}^{-i x\cdot \xi}f(s,x) \mathrm{d}x. 
\end{align*}
We use both notations $\cF[f(s)](\xi)$ and $\cF[f(s, \cdot)](\xi)$ to give the same meaning in either way.

Now assume that $u_0 \in L_{1,\ell oc}(\bR^d)$ and $f \in L_{1,t\text{-}loc}\left( (0,T) ;  \cF^{-1}L_{1, \ell oc}(\bR^d) \right)$.
In such cases, all the components on the right side of equation \eqref{make sense representation} are well-defined and meaningful, provided that the symbol $\psi(t,\xi)$ is locally bounded with the help of the extension of the inverse Fourier transform.

Finally, we suggest a definition of a weak solution to equation \eqref{ab eqn}. 
We call it a \emph{Fourier-space weak solution} to \eqref{ab eqn}.
\begin{defn}[Fourier-space weak solution]
									\label{space weak solution}
Let $u_0 \in \cF^{-1}\cD'(\bR^d)$, $f \in L_{1,t\text{-}loc}\left( (0,T) ;  \cF^{-1}\cD'(\bR^d) \right)$, and $u \in L_{0}\left((0,T);\cF^{-1}\cD'(\bR^d)\right)$.
We say that $u$ is a \emph{Fourier-space weak solution} to equation \eqref{ab eqn} if for any $\varphi \in \cF^{-1}\cD(\bR^d)$,
\begin{align}
										\label{weak formulation}
\left\langle u(t,\cdot),\varphi \right\rangle 
= \langle u_0, \varphi \rangle +  \int_0^t \left\langle  \psi(s,-i \nabla)u(s,\cdot) , \varphi \right\rangle\mathrm{d}s 
+ \int_0^t \left\langle f(s,\cdot),\varphi\right\rangle \mathrm{d}s
\quad (a.e.)~t\in (0,T).
\end{align}
\end{defn}
The term ``Fourier-space weak solution" as described above is evidently derived from the fact that the Fourier transform is applied to the spatial variable of a solution $u$.

\begin{rem}
Note that the term 
$\int_0^t \left\langle \psi(s,-i \nabla)u(s,\cdot) , \varphi \right\rangle\mathrm{d}s $
is not well-defined in general even though we consider the extended Fourier and inverse Fourier transforms.
It is because the multiplication action 
$$
\langle \psi(t,\cdot) \cF[u](t,\cdot) , \varphi \rangle
:=\langle  \cF[u](t,\cdot) , \psi(t,\cdot) \varphi \rangle
$$
is not well-defined since the symbol $\psi(t,\xi)$ does not satisfy a regularity condition in general.
The easiest way to make the term well-defined is to give a sufficient smooth assumption on the symbol $\psi(t,\xi)$.
However, it is not the case we are interested in our weak formulation by recalling the goal that we want to remove all regularity conditions on the symbol $\psi(t,\xi)$.
Fortunately, there is a method to make the term well-defined without any regularity condition on $\psi(t,\xi)$ by using a realization of $\cF[u(t,\cdot)]$ as discussed in Definition \ref{defn psi operator}, which will be revisited in the next remark.
\end{rem}

\begin{rem}
								\label{solution realization}
Let $u$ be a Fourier-space weak solution to \eqref{ab eqn}.
Assume that for each $t \in (0,T)$, the function $\psi(t,\xi) \cF[u(t,\cdot)](\xi)$ is locally integrable with respect to $\xi$.
In this case, it is clear that $\psi(t,\xi) \cF [u(t,\cdot)](\xi)$ belongs to the space of distributions $\cD'(\bR^d)$.
Consequently, we can proceed to consider the inverse Fourier transform of $\psi(t,\xi) \cF [u(t,\cdot)](\xi)$ due to Definition \ref{fourier defn 2}. This implies that the main operator component, $\psi(t,i\nabla)u(t,\cdot)$, becomes meaningful as a function with values in $\cF^{-1}\cD'(\bR^d)$.
Additionally according to definitions of actions on $\cF^{-1}\cD'(\bR^d)$, we have
\begin{align}
								\label{20240201 01}
\int_0^t \langle \psi(s,-i\nabla)u(s,\cdot) , \varphi \rangle \mathrm{d}s 
=\int_0^t \left(\psi(s,\cdot) \cF[u(s,\cdot)] ,\cF[\varphi](\cdot) \right)_{L_2(\bR^d)} \mathrm{d}s 
\end{align}
for all $\varphi \in \cF^{-1}\cD(\bR^d)$, where $(\cdot,\cdot)_{L_2(\bR^d)}$ denotes the $L_2(\bR^d)$-inner product, \textit{i.e.},
\begin{align*}
\left(\psi(s,\cdot)\cF[u(s,\cdot)] ,\cF[\varphi](\cdot) \right)_{L_2(\bR^d)} 
= \int_{\bR^d} \psi(s,\xi) \cF[u(s,\cdot)](\xi)  \overline{\cF[\varphi](\xi)} \mathrm{d} \xi.
\end{align*}
Therefore, we can use the identity in \eqref{20240201 01}  to define 
$\langle \psi(s,-i\nabla)u(s,\cdot) , \varphi \rangle$ as $\left(\psi(s,\cdot) \cF[u(s,\cdot)] ,\cF[\varphi] \right)_{L_2(\bR^d)}$.
In other words, the term $\psi(s,-i\nabla)g$  could be defined for any $g$ whose  Fourier transform  has a realization so that
$\psi(s,\xi) \cF[g](\xi)$ is locally integrable even though the symbol $\psi(t,\xi)$ is not regular at all. 
It seems obvious that this realization of $g$ is sufficient to exist within the support of the symbol $\psi$ which will be discussed further in the subsequent remark.
\end{rem}

\begin{rem}
							\label{cutoff distribution}
Let $u \in L_{1,t\text{-}loc}\left( (0,T) ;  \cF^{-1}\cD'(\bR^d) \right)$.
Then even the term  
\begin{align}
										\label{20230624 01}
\int_0^t \left(\psi(s,\cdot)\cF[u(s,\cdot)] , \cF[\varphi](\cdot) \right)_{L_2(\bR^d)} \mathrm{d}s
=\int_0^t \left(\cF[u(s,\cdot)] , \overline \psi(s,\cdot)\cF[\varphi](\cdot) \right)_{L_2(\bR^d)} \mathrm{d}s
\end{align}
is not well-defined in general.
Moreover, there exists the possibility of finding a function $u$ that does not belong to the space $L_{1,t\text{-}loc}\left( (0,T) ; \cF^{-1}L_{1,\ell oc}(\bR^d) \right)$, yet under the influence of $\psi(s,\cdot)$, the expression \eqref{20230624 01} can still be made meaningful.
To address this issue, we aim to find appropriate functions, which will be presented in our main theorem and belong to spaces whose Fourier transforms are in weighted local $L_{p,q}$-spaces which are subspaces of $L_{0}\left((0,T);\cF^{-1}\cD'(\bR^d)\right)$.
In particular, if $u$ is a Fourier-space weak solution to \eqref{ab eqn} so that the term in \eqref{20230624 01} is well-defined,
then at least, it implies that for each $s \in (0,T)$, there exists a realization of $\cF[u(s,\cdot)](\xi)$ on the support of $\psi(s,\xi)$.
In other words, for each $s$, there exists a locally integrable function $v$ on $\bR^d$ and a neighborhood $\cU$ of $\supp \psi(s,\cdot)$ so that 
\begin{align*}
\langle \cF[u(s,\cdot)], \varphi \rangle = \int_{\bR^d} v(\xi) \varphi(\xi) \mathrm{d}\xi \quad \forall \varphi \in \cD \left( \cU\right),
\end{align*}
where $\supp \psi(s,\cdot) = \overline{\{ x \in \bR^d : \psi(s,x) \neq 0\}}$,
$\overline{\{ x \in \bR^d : \psi(s,x) \neq 0\}}$ is the closure of ${\{ x \in \bR^d : \psi(s,x) \neq 0\}}$.
However, we can give any value to $\cF[u(s,\cdot)]$ outside of $\supp \psi(s,\cdot)$ since it does not affect the terms in \eqref{20230624 01}.
Thus it might suggest that the space $L_{1,t\text{-}loc}\left( (0,T) ;  \cF^{-1}\cD'(\bR^d) \right)$ is not appropriate to guarantee the uniqueness of a solution $u$.
By the way, there is a typical action to extend by putting
\begin{align*}
\langle \cF[u(s,\cdot)] ,\varphi \rangle =0  \quad \forall \varphi \in \cD \left( \bR^d \setminus  \supp \psi(s,\cdot)\right).
\end{align*}
Then the support of the distribution $\cF[u(s,\cdot)]$ becomes a subset of $supp \psi(s,\cdot)$ and 
the realization $v$ satisfies $v=v \cdot 1_{supp \psi(s,\cdot)}$, where $1_{supp \psi(s,\cdot)}$ denotes the indicator function on the set $supp \psi(s,\cdot)$, \textit{i.e.}
$$
1_{supp \psi(s,\cdot)}(x) =
\begin{cases}
&1  \quad x \in supp \psi(s,\cdot) \\
&0 \quad x \in \bR^d \setminus supp \psi(s,\cdot).
\end{cases}
$$
\end{rem}
\begin{rem}
In the classical weak formulation of partial differential equations, the commonly preferred class of test functions is $\cD(\bR^d)$. 
This preference is shaped by several factors.
One of the primary reasons is that $\cD(\bR^d)$ serves as a well-behaved dense subset. 
Specifically, $\cD(\bR^d)$ is dense in $L_p(\bR^d)$ for any $p \in [1,\infty)$, and it is also dense in $\cS(\bR^d)$. 
These properties of denseness are significant considerations in favor of utilizing $\cD(\bR^d)$ as a space of test functions. 
They play a crucial role in facilitating a connection between strong and weak solutions.

It is worth noting that the class $\cF^{-1}\cD(\bR^d)$ serves as an excellent alternative for a class of test functions, as it possesses the same dense properties. These properties can be easily shown based on well-known properties of $\cD(\bR^d)$ and $\cS(\bR^d)$. 
Here is a brief description of a proof. Recall that $\cD(\bR^d)$ is dense in $\cS(\bR^d)$ concerning the topology generated by the Schwartz semi-norms. 
Additionally, the Fourier transform acts as a homeomorphism from $\cS(\bR^d)$ to itself. 
Consequently, it is almost straightforward to prove that $\cF^{-1}\cD(\bR^d)$ is dense in $\cS(\bR^d)$. An interesting observation is that both $\cD(\bR^d)$ and $\cF^{-1}\cD(\bR^d)$ are dense in $\cS(\bR^d)$, but their intersection contains only the zero element.
\end{rem}

We classify our solutions and data by examining their Fourier transforms. 
Specifically, we focus on an initial data $u_0$ and an inhomogeneous data $f$ in equation \eqref{ab eqn} whose Fourier transforms belong to weighted $L_p$-spaces. 
To establish clear definitions for these data spaces, we need to review notations associated with $L_p$-spaces related to weighted Lebesgue measures.
It is important to note that our weighted $L_p$-spaces include the special case of $p=\infty$. 
Nevertheless, we temporarily exclude this extreme case in order to simplify the comprehension of functions solely in terms of integrations.

Let $p,q \in [1,\infty)$, $w$, and $W$ be non-negative (Borel) measurable functions defined $(a.e.)$ on $(0,T)$ and $(0,T) \times \bR^d$, respectively.
We write
 $u \in L_{p,q}\left( (0,T) \times \bR^d,  w^p(t) \mathrm{d}t W^q(t,x)\mathrm{d}x \right)$ if and only if
\begin{align*}
\|u(t,x)\|_{ L_{p,q}\left( (0,T) \times \bR^d,  w^p(t) \mathrm{d}t W^q(t,x) \mathrm{d}x \right)}
:= \left(\int_{0}^{T} \left(\int_{\bR^d} |u(t,x)|^q W^q(t,x) \mathrm{d}x \right)^{p/q}  w^p(t)\mathrm{d}t \right)^{1/p} < \infty.
\end{align*}
It is obvious that
$$
L_{p,q}\left( (0,T) \times \bR^d,  w^p(t)\mathrm{d}t  W^q(t,x)\mathrm{d}x \right)
=L_{p,q}\left( (0,T) \times \bR^d,  \mathrm{d}t  w^q(t) W^q(t,x)\mathrm{d}x \right).
$$
In particular, 
$$
L_{p,p}\left( (0,T) \times \bR^d,  w^p(t)\mathrm{d}t  W^p(t,x)\mathrm{d}x \right)
=L_{p,p}\left( (0,T) \times \bR^d,  \mathrm{d}t  w^p(t) W^p(t,x)\mathrm{d}x \right).
$$
In this case, our space becomes identical to the $L_p$-space with the measure $w^p(t) W^p(t,x)\mathrm{d}t\mathrm{d}x$, \textit{i.e.} 
\begin{align*}
L_{p}\left( (0,T) \times \bR^d,    w^p(t) W^p(t,x)\mathrm{d}t\mathrm{d}x \right)
&=L_{p,p}\left( (0,T) \times \bR^d,  w^p(t)\mathrm{d}t  W^p(t,x)\mathrm{d}x \right)\\
&=L_{p,p}\left( (0,T) \times \bR^d,  \mathrm{d}t  w^p(t) W^p(t,x)\mathrm{d}x \right).
\end{align*}
In addition, observe that
\begin{align*}
\|u\|_{L_{p,q}\left( (0,T) \times \bR^d,  w^p(t)\mathrm{d}t  W^q(t,x)\mathrm{d}x \right)} 
&= \|u\|_{L_{p}\left( (0,T), w^p(t) \mathrm{d}t ;  L_q \left(\bR^d, W^q(t,x)\mathrm{d}x \right)\right)} \\
&=\left(\int_{0}^{T} \left(\int_{\bR^d} |u(t,x)|^q W^q(t,x) \mathrm{d}x \right)^{p/q}  w^p(t)\mathrm{d}t \right)^{1/p} \\
&=\left(\int_{0}^{T} \left(\int_{\bR^d} | w(t)W(t,x)u(t,x)|^q  \mathrm{d}x \right)^{p/q} \mathrm{d}t \right)^{1/p} \\
&= \|wWu\|_{L_{p}\left( (0,T) ;  L_q \left(\bR^d \right)\right)}.
\end{align*}
To include the extreme case that $p=\infty$ or $q=\infty$, 
we define
\begin{align*}
\|u(t,x)\|_{L_{p,q}\left( (0,T) \times \bR^d,  w^p(t)\mathrm{d}t  W^q(t,x)\mathrm{d}x \right)} 
:= \|wWu\|_{L_{p}\left( (0,T) ;  L_q \left(\bR^d \right)\right)}.
\end{align*}
For $p,q \in [1,\infty)$, we remark that  the simpler notations
$$
L_{p,q}\left( (0,T) \times \bR^d,  w(t)\mathrm{d}t  W(t,x)\mathrm{d}x \right)
:=
L_{p,q}\left( (0,T) \times \bR^d,  (w^{1/p})^p(t)\mathrm{d}t  (W^{1/q})^q(t,x)\mathrm{d}x \right).
$$
are used on occasion for the sake of consistency with usual notations of weighted $L_{p,q}$-spaces.

We usually identify 
$L_{p,q}\left( (0,T) \times \bR^d, w^p(t) \mathrm{d}t   W^q(t,x) \mathrm{d}x \right)$ 
and $L_{p}\left( (0,T), w^p(t) \mathrm{d}t ;  L_q \left(\bR^d, W^q(t,x)\mathrm{d}x \right)\right)$
even though the iterated measurability does not imply the joint measurability.
We also would like to point out that we would not consider the intricate matters between Borel and Lebesgue measurabilities by opting for the completion of a measure space. 
Furthermore, we adopt a conventional method to establish the equivalence in the almost everywhere sense across all $L_{p,q}$-spaces
without delving into significant details. 
In particular, we do not make a strict distinction between functions defined throughout $(0,T) \times \mathbb{R}^d$ and those defined almost everywhere within $(0,T) \times \mathbb{R}^d$.  
Similarly, this relaxed approach extends to functions defined on $\mathbb{R}^d$ as well.
With these prerequisites in place, we can now introduce a concept of our weighted local $L_{p,q}$-spaces. 
It is worth noting that this local characteristic is presented in a slightly different manner with respect to the time variable, which will be further explored in Remark \ref{local revisit}.
\begin{defn}[Weighted local $L_{p,q}$-spaces]
Let $p,q \in [1,\infty]$.
\begin{enumerate}[(i)]
\item We define $L_{p,q,t\text{-}loc,x\text{-}\ell oc}\left( (0,T) \times \bR^d, w^p(t)\mathrm{d}t  W^q(t,x)\mathrm{d}x \right)$ 
as the class of all measurable functions $u$ on $(0,T) \times \bR^d$ such that
\begin{align*}
\|u\|_{L_{p,q}\left( (0,t_1) \times B_R,  w^p(t)\mathrm{d}t  W^q(t,x)\mathrm{d}x \right)}
:=\|1_{(0,t_1)}(t) 1_{B_R}(x)u(t,x)\|_{L_{p,q}\left( (0,T) \times \bR^d,  w^p(t)\mathrm{d}t  W^q(t,x)\mathrm{d}x \right)} < \infty
\end{align*}
for all $0<t_1<T$ and $R \in (0,\infty)$, where $B_R:=\{x \in \bR^d : |x| < R\}$ and $1_{B_R}$ denotes the indicator function of $B_R$, \textit{i.e.}
$$
1_{B_R}(x) =
\begin{cases}
&1  \quad x \in B_R, \\
&0 \quad x \in \bR^d \setminus B_R.
\end{cases}
$$

\item 
We define $L_{p,q,t\text{-}loc}\left( (0,T) \times \bR^d, w^p(t)\mathrm{d}t  W^q(t,x)\mathrm{d}x \right)$ 
as the class of all measurable functions $u$  on $(0,T) \times \bR^d$ such that
\begin{align*}
\|u(t,x)\|_{L_{p,q}\left( (0,t_1) \times \bR^d,  w^p(t)\mathrm{d}t  W^q(t,x)\mathrm{d}x \right)} < \infty
\end{align*}
for all $0<t_1<T$.
\end{enumerate}
\end{defn}

We now turn our attention to a significant subset within $L_{0}\left( (0,T) ; \cF^{-1}\cD'(\bR^d) \right)$, which comprises functions $u$ whose Fourier transform is in a weighted local $L_{p,q}$-space.
This subset can be formally expressed as:
\begin{align}
										\label{20240324 20}
\left\{ u : \cF[u(t,\cdot)](\xi) \in L_{p,q,t\text{-}loc,x\text{-}\ell oc}\left( (0,T) \times \bR^d, w^p(t)\mathrm{d}t  W^q(t,x)\mathrm{d}x \right)  \right\}.
\end{align}
However, precisely defining the expression above (\ref{20240324 20}) poses challenges due to the potential degeneracy of our weight function $W(t,\xi)$ for certain $\xi \in \bR^d$.
It requires employing sophisticated techniques involving identifications with complex-valued functions, as illustrated in (\ref{20240319 01}).

Let $u \in L_{0}\left( (0,T) ; \cF^{-1}\cD'(\bR^d) \right)$.
Then for each $t \in (0,T)$, we have $\cF[u(t)] \in \cD'(\bR^d)$.
Additionally, assume that for each $t \in (0,T)$, $\cF[u(t)]$ has a realization on an open set $\cU_t$ covering the support of $W(t,\cdot)$, \textit{i.e.}  $\overline{\{\xi \in \bR^d : W(t,\xi)\neq 0\}}  \subset \cU_t$, where 
$\overline{\{\xi \in \bR^d : W(t,\xi)\neq 0\}}$ denotes the closure of $\{\xi \in \bR^d : W(t,\xi)\neq 0\}$.
Thus by identifying $\cF[u(t)]$ with the realization on $\cU$, we may assume that
$(t,\xi) \mapsto \cF[u(t,\cdot)](\xi):= \cF[u](t,\xi):= \cF[u(t)](\xi)$ is a complex-valued function defined on 
$\{ (t,\xi) \in (0,T) \times \bR^d : (t,\xi) \in (0,T) \times \cU_t \}$. In particular,
the term $\|wW\cF[u]\|_{L_{p}\left( (0,T) ;  L_q \left(\bR^d \right)\right)}$
perfectly makes sense for all $p,q \in [1,\infty]$ by letting the product $wW\cF[u]$ be zero outside the set $\{ (t,\xi) \in (0,T) \times \bR^d : (t,\xi) \in (0,T) \times \cU_t \}$
since $W(t,\xi)=0$ for all $(t,\xi) \notin (0,T) \times \cU_t$.
Therefore simply we may say that 
$$
\cF[u(t,\cdot)](\xi) \in L_{p,q}\left( (0,T) \times \bR^d,  w^p(t)\mathrm{d}t  W^q(t,\xi)\mathrm{d}\xi \right),
$$
if and only if  $u$  is an element in $L_{0}\left( (0,T) ; \cF^{-1}\cD'(\bR^d) \right)$ so that for each $t \in (0,T)$, $\cF[u(t)]$ has a realization on an open set $\cU_t$ covering the support of $W(t,\cdot)$ and $\|wW\cF[u]\|_{L_{p}\left( (0,T) ;  L_q \left(\bR^d \right)\right)}<\infty$.
Finally, we define subspaces of $L_{0}\left( (0,T) ; \cF^{-1}\cD'(\bR^d) \right)$ consisting of $u$ whose Fourier transform lies in a weighted (local) $L_{p,q}$-spaces in the following definition based on the aforementioned criteria. 

\begin{defn}
\begin{enumerate}[(i)]
\item We define subclasses of $L_{0}\left( (0,T) ; \cF^{-1}\cD'(\bR^d) \right)$ consisting of elements whose  Fourier transforms with respect to the space variable are in weighted (local) $L_{p,q}$-spaces.
We write 
$$
u \in \cF^{-1}L_{p,q}\left( (0,T) \times \bR^d,  w^p(t)\mathrm{d}t  W^q(t,\xi)\mathrm{d}\xi \right),
$$
$$
u \in \cF^{-1}L_{p,q,t\text{-}loc}\left( (0,T) \times \bR^d,  w^p(t)\mathrm{d}t  W^q(t,\xi)\mathrm{d}\xi \right),
$$
and 
$$
u \in \cF^{-1}L_{p,q,t\text{-}loc,x\text{-}\ell oc}\left( (0,T) \times \bR^d,  w^p(t)\mathrm{d}t  W^q(t,\xi)\mathrm{d}\xi \right)
$$
 if
$$
\cF[u(t,\cdot)](\xi) \in L_{p,q}\left( (0,T) \times \bR^d,  w^p(t)\mathrm{d}t  W^q(\xi)\mathrm{d}\xi\right),
$$
$$
\cF[u(t,\cdot)](\xi) \in L_{p,q,t\text{-}loc}\left( (0,T) \times \bR^d,  w^p(t)\mathrm{d}t  W^q(t,\xi) \mathrm{d}\xi \right),
$$ 
and
$$
\cF[u(t,\cdot)](\xi) \in L_{p,q,t\text{-}loc,x\text{-}\ell oc}\left( (0,T) \times \bR^d,  w^p(t)\mathrm{d}t  W^q(t,\xi) \mathrm{d}\xi \right),
$$ 
respectively.
In other words, 
$$
u \in \cF^{-1}L_{p,q,t\text{-}loc,x\text{-}\ell oc}\left( (0,T) \times \bR^d,w^p(t)\mathrm{d}t  W^q(t,\xi)\mathrm{d}x \right)
$$
if and only if
\begin{align*}
\|\cF[u]\|_{L_{p,q}\left( (0,t_1) \times B_R,  w^p(t)\mathrm{d}t  W^q(t,\xi)\mathrm{d}\xi \right)}
:=\|1_{B_R}(\xi)\cF[u(t,\cdot)](\xi)\|_{L_{p,q}\left( (0,t_1) \times \bR^d,  w^p(t)\mathrm{d}t  W^q(t,\xi)\mathrm{d}\xi \right)} < \infty
\end{align*}
for all  $t_1 \in (0,T)$ and $R \in (0,\infty)$.
\end{enumerate}
\end{defn}
The weights $w^p(t)$, $W^q(t,x)$, and $W^q(t,\xi)$ in the notations are skipped if they are merely $1$. 
If all weights are $1$, then we also omit $\mathrm{d}t\mathrm{d}x$ and $\mathrm{d}t\mathrm{d}\xi$ in the notations.

These local spaces are employed to give weaker conditions on symbols and find a class of solutions to \eqref{ab eqn} where the existence and uniqueness are guaranteed.
Recall Definition \ref{local 11 space}. Then it is obvious that
\begin{align*}
L_{1,t\text{-loc}}\left((0,T);\cF^{-1}L_{1,\ell oc}(\bR^d)\right)
=\cF^{-1}L_{1,1,t\text{-}loc,x\text{-}\ell oc}\left( (0,T) \times \bR^d \right).
\end{align*}
In addition, by Definition \ref{local l1 fd valued space} and H\"older's inequality,
we have the inclusion that
\begin{align*}
\cF^{-1}L_{p,q,t\text{-}loc,x\text{-}\ell oc}\left( (0,T) \times \bR^d\right) 
\subset \cF^{-1}L_{1,1,t\text{-}loc,x\text{-}\ell oc}\left( (0,T) \times \bR^d \right) 
\subset L_{1,t\text{-}loc}\left( (0,T) ;  \cF^{-1}\cD'(\bR^d) \right)
\end{align*}
for all  $p,q \in [1,\infty]$.
However, 
\begin{align}
										\label{20240131 01}
\cF^{-1}L_{p,q,t\text{-}loc,x\text{-}\ell oc} \left( (0,T) \times \bR^d,w^p(t)\mathrm{d}t  W^q(t,\xi)\mathrm{d}\xi \right) 
&\not\subset \cF^{-1}L_{1,1,t\text{-}loc,x\text{-}\ell oc}\left( (0,T) \times \bR^d\right)
\end{align}
in general, since even there is no strict positivity assumption on the weights $w$ and $W$.
\begin{rem}
There are numerous mathematical conveniences when we exclusively consider positive weights. 
Nevertheless, it is crucial not to disregard the possibility of weights becoming zero, as this allows us to include the interesting cases in which the symbol $\psi(t,\xi)$ may become degenerate. 
Furthermore, it is true that, for all values of $p$ and $q$ within the range of $[1,\infty]$,
\begin{align*}
\cF^{-1}L_{p,q,t\text{-}loc,x\text{-}\ell oc} \left( (0,T) \times \bR^d,w^p(t)\mathrm{d}t W^q(t,\xi)\mathrm{d}\xi \right)
\subset L_0\left( (0,T) ; \cF^{-1}\cD'(\bR^d) \right)
\end{align*}
due to the definition of $\cF^{-1}L_{p,q,t\text{-}loc,x\text{-}\ell oc} \left( (0,T) \times \bR^d,w^p(t)\mathrm{d}t W^q(t,\xi)\mathrm{d}\xi \right)$.
If both weights $w$ and $W$ are strictly positive, then 
\begin{align*}
\cF^{-1}L_{p,q,t\text{-}loc,x\text{-}\ell oc} \left( (0,T) \times \bR^d,w^p(t)\mathrm{d}t W^q(t,\xi)\mathrm{d}\xi \right)
\subset L_0\left( (0,T) ; \cF^{-1}L_{1,\ell oc}(\bR^d) \right)
\end{align*}
for all  $p,q \in [1,\infty]$ since $|\cF[u(t,\cdot)](\xi)| < \infty$ $(a.e.)$ $(t,\xi) \in (0,T) \times \fR^d$ based on realizations on $\bR^d$ for any 
$$
u \in \cF^{-1}L_{p,q,t\text{-}loc,x\text{-}\ell oc} \left( (0,T) \times \bR^d,w^p(t)\mathrm{d}t W^q(t,\xi)\mathrm{d}\xi \right).
$$
However, the space $\cF^{-1}L_{p,q,t\text{-}loc,x\text{-}\ell oc} \left( (0,T) \times \bR^d,w^p(t)\mathrm{d}t W^q(t,\xi)\mathrm{d}\xi \right)$ is not a metric space in general if one of the weights is degenerate since $\cF[u(t,\cdot)]$ could have any values outside of the supports of the weights
as discussed in Remark \ref{cutoff distribution}. 
\end{rem}

\begin{rem}
We skip specifying the exact definitions of corresponding weighted $L_q$-spaces treating data given on $\bR^d$  since they are easily induced from corresponding weighted  $L_{p,q}$-spaces on $(0,T) \times \bR^d$. For instance, $\cF^{-1}L_{q,\ell oc}(\bR^d)$ denotes the subspace of $\cF^{-1}\cD'(\bR^d)$ so that
\begin{align*}
\|\cF[u]\|_{L_q(B_R)} < \infty \quad \forall R \in (0,\infty).
\end{align*}
\end{rem}

\begin{rem}

Let $\mathscr{E}'(\bR^d)$ represent the dual space of $C^\infty(\bR^d)$. Then, for any element $u$ belonging to $\mathscr{E}'(\bR^d)$, its Fourier transform, denoted by $\cF[u]$, can be interpreted as a function that is locally bounded. To clarify, for every $\xi$ within $\bR^d$, the function $\mathrm{e}^{-i x \cdot \xi}$ is a smooth function with respect to the variable $x$, and the mapping from $\xi$ to $\langle u, \mathrm{e}^{-i x \cdot \xi}\rangle$ is a continuous function. Moreover, it is widely recognized that this space $\mathscr{E}'(\bR^d)$ corresponds to a subset of $\cD'(\bR^d)$ with compact supports, as discussed in \cite[Section 2.3]{Hormander 1990}.
As a result, we can conclude that $\mathscr{E}'(\bR^d)$ is contained within $\cF^{-1}L_{q,\ell oc}(\bR^d)$ for all $q \in [1,\infty]$. 
\end{rem}
\begin{rem}
								\label{local revisit}
We can describe a function $u$ belonging to the space $L_{p,q,t\text{-}loc,x\text{-}\ell oc}\left( (0,T) \times \bR^d \right)$ as having a strong local integrability with respect to the variable $t$ and a weaker form of local integrability with respect to the variable $x$.
To clarify further, consider a topological space denoted by $X$, and recall that when we say a property holds locally, it means that the property is satisfied for all compact subsets $K \subset X$. In this context, our function $u \in L_{p,q,t\text{-}loc,x\text{-}\ell oc}\left( (0,T) \times \bR^d\right)$ can be described as locally integrable with respect to $x$, but it satisfies a somewhat stronger condition compared to the conventional notion of local integrability concerning the time variable $t$. This distinction is why we employed different notations, namely, ``$loc$" and `$\ell oc$", to indicate the local integrability properties of $t$ and $x$, respectively.
\end{rem}

We finally reach our main theorem.

\begin{thm}
							\label{weak solution thm}
Let $u_0 \in \cF^{-1}L_{1,\ell oc}(\bR^d)$ and $f \in \cF^{-1}L_{1,1,t\text{-}loc,x\text{-}\ell oc}\left((0,T) \times \bR^d \right)$.
Suppose that for all $t \in (0,T)$ and $R \in (0,\infty)$,  
\begin{align}
& \int_{B_R}\exp\left(\int_0^t\Re[\psi(r,\xi)]\mathrm{d}r \right) \cF[u_0](\xi) \mathrm{d}\xi
+  \int_{B_R} \int_0^t  \exp\left(\int_s^t\Re[\psi(r,\xi)]\mathrm{d}r \right) \left|\cF[f(s,\cdot)](\xi) \right| \mathrm{d}s \mathrm{d}\xi  
									\label{20230624 10}
&< \infty, 
\end{align}
where $B_R:= \{ \xi \in \bR^d : |\xi| <R\}$. 
Additionally, assume that for all $ t \in (0,T)$ and $R \in (0,\infty)$,
\begin{align}
										\notag
& \int_{B_R}\int_0^t |\psi(\rho,\xi)| \exp\left(\int_0^\rho \Re[\psi(r,\xi)]\mathrm{d}r \right) \cF[u_0](\xi)\mathrm{d}\rho \mathrm{d}\xi \\
									\label{20230624 11}
&+\int_{B_R}\int_0^t |\psi(\rho,\xi)| \int_0^\rho \exp\left(\int_s^\rho\Re[\psi(r,\xi)]\mathrm{d}r \right)  \left|\cF[f(s,\cdot)](\xi) \right| \mathrm{d}s \mathrm{d}\rho \mathrm{d}\xi < \infty.
\end{align}
Then there exists a unique Fourier-space weak solution 
$$
u \in \cF^{-1}L_{\infty,1,t\text{-}loc, x\text{-}\ell oc}\left( (0,T) \times \bR^d \right) \cap \cF^{-1}L_{1,1,t\text{-}loc, x\text{-}\ell oc}\left( (0,T) \times \bR^d, \mathrm{d}t |\psi(t,\xi)|\mathrm{d}\xi \right)
$$ 
to \eqref{ab eqn}, which is given by
\begin{align}
\langle u(t,\cdot), \varphi \rangle
									\notag
&=\int_{\bR^d} \exp\left(\int_0^t\psi(r,\xi)\mathrm{d}r \right) \cF[u_0](\xi) \cF[\varphi](\xi)  \mathrm{d}\xi \\
									\label{20240313 01}
&\quad +\int_{\bR^d}\left(\int_0^t  \exp\left(\int_s^t\psi(r,\xi)\mathrm{d}r \right) \cF[f(s,\cdot)](\xi)\mathrm{d}s\right)\cF[\varphi](\xi) \mathrm{d}\xi
\quad \forall \varphi \in \cF^{-1}\cD(\bR^d).
\end{align}

\end{thm}
The proof of Theorem \ref{weak solution thm} is given in Section \ref{pf main thm 1}.

\begin{rem}
								\label{kernel remark}
The solution $u$ to \eqref{ab eqn} in Theorem \ref{weak solution thm} is not defined as a complex-valued function on $[0,T) \times \bR^d$.
It is a function in the class $L_{0}\left((0,T);\cF^{-1}\cD'(\bR^d)\right)$, which is clarified by the action in \eqref{20240313 01}. 
Therefore, formally, we may say that
\begin{align}
										\label{solution candidate}
u(t,x)=  \cF^{-1}\left[  \exp\left(\int_0^t\psi(r,\cdot)\mathrm{d}r \right) \cF[u_0] \right](x)
+\cF^{-1}\left[ \int_0^t  \exp\left(\int_s^t\psi(r,\cdot)\mathrm{d}r \right) \cF[f(s,\cdot)]\mathrm{d}s \right](x)
\end{align}
is a solution to \eqref{ab eqn}.
It is obvious that this representation for the solution $u$ is identical to the classical one if $\psi$, $u_0$, and $f$ satisfy nice mathematical conditions by considering convolutions with kernels as follow :
\begin{align*}
u(t,x) =  \int_{\bR^d} K(0,t,x) u_0(x-y) \mathrm{d}y  + \int_0^t \int_{\bR^d} K(s,t,y)f(s,x-y) \mathrm{d}y \mathrm{d}s,
\end{align*}
where the kernel $K(s,t,x)$ is given by
\begin{align*}
K(s,t,x) =   \cF^{-1}\left[\exp\left(\int_s^t\psi(r,\cdot)\mathrm{d}r \right)\right](x).
\end{align*}
This kernel representation shows that it is not hopeful to expect that $u$ has a realization on $[0,T) \times \bR^d$ if the symbol $\psi(t,\xi)$ is not negative. 
\end{rem}

\begin{rem}
We intentionally split the mathematical assumptions in Theorem \ref{weak solution thm} into two parts, \eqref{20230624 10} and \eqref{20230624 11}, 
since the condition \eqref{20230624 11} is sufficient to guarantee the uniqueness of a solution, which is specified in Theorem \ref{unique weak sol}.
On the other hand, \eqref{20230624 10} guarantee the function $u(t,x)$ defined as \eqref{solution candidate} is well-defined as $\cF^{-1}\cD'(\bR^d)$-valued function.
A naive sufficient condition for both \eqref{20230624 10} and \eqref{20230624 11} is a local bounded property on the symbol $\psi(t,\xi)$.
More precisely, if 
$\left\| \psi(\rho,\xi) \right\|_{L_{\infty,\infty}\left( (0,t) \times B_R , \mathrm{d}\rho \mathrm{d} \xi \right) } < \infty$
for all $t \in (0,T)$ and $R \in (0,\infty)$, then both \eqref{20230624 10} and \eqref{20230624 11} are satisfied, which will be relaxed in Corollary \ref{cor 202307014 01} with positive weights and general exponents.
We also want to mention that both \eqref{20230624 10} and \eqref{20230624 11} are given by combinations of properties on the symbol $\psi$ and data $u_0$ and $f$ in a complicated way on purpose to emphasize that the condition on the symbol could be weaken if data $u_0$ and $f$ are in better spaces.
One of the good examples showing weakening condition on the symbol $\psi$ by enhancing the condition on data $u_0$ and $f$ is presented in Corollary \ref{corollary weight pq}. 
Additionally, these complicated conditions \eqref{20230624 10} and \eqref{20230624 11} make us handle logarithmic operators instead of merely considering locally bounded symbols, which will be specified in Theorem \ref{main log} below.
\end{rem}

\begin{rem}
We revisit to mention the importance of the extension of the inverse Fourier transform to all distributions.
Assume that the symbol $\psi(t,\xi)$ is locally bounded.
Then the function $\exp\left(\int_0^t\psi(r,\cdot)\mathrm{d}r \right) \cF[u_0]$ is merely locally integrable with respect to $\xi$.
Thus, there is no assurance that $\exp\left(\int_0^t\psi(r,\cdot)\mathrm{d}r \right) \cF[u_0]$ qualifies as a tempered distribution on $\bR^d$.
Consequently, the function $u$ in equation \eqref{solution candidate} cannot be defined because the inverse Fourier transform is typically applicable only to tempered distributions on $\bR^d$ (if we do not consider an extension to $\bC^d$).
However, it is important to note that we extended the inverse Fourier transform to all distributions as an element in $\cF^{-1}\cD'(\bR^d)$. 
Therefore, the function $u$ is well-defined as a new extension of the inverse Fourier transform to a distribution even if the functions $\exp\left(\int_0^t\psi(r,\cdot)\mathrm{d}r \right) \cF[u_0]$ and $\int_0^t  \exp\left(\int_s^t\psi(r,\cdot)\mathrm{d}r \right) \cF[f(s,\cdot)]\mathrm{d}s$ are merely locally integrable. 
Especially, it is easy to check that 
$$
u \in \cF^{-1}L_{\infty,1,t\text{-}loc, x\text{-}\ell oc}\left( (0,T) \times \bR^d \right) \cap \cF^{-1}L_{1,1,t\text{-}loc, x\text{-}\ell oc}\left( (0,T) \times \bR^d, \mathrm{d}t |\psi(t,\xi)|\mathrm{d}\xi \right)
$$ 
due to \eqref{20230624 10} and \eqref{20230624 11}.
\end{rem}

\begin{rem}
It is not trivial to understand realizations of data $u_0$ and $f$ since all data $u_0$ and $f$ are given based on Fourier transforms acting on $\cF^{-1}\cD'(\bR^d)$. 
However, it is important to note that $\cF^{-1}\cD'(\bR^d)$ is a broader class than $\cS'(\bR^d)$. 
Within $\cS'(\bR^d)$, there exist various well-known subspaces that help characterize regularities of elements in $\cS'(\bR^d)$. 
One notable example is the Bessel potential space.
Given that $\cF^{-1}\cD'(\bR^d)$ is larger than any subspace of $\cS'(\bR^d)$, we could apply our theorem to data $u_0$ and $f$ residing in Bessel potential spaces. 
We suggest a weighted Bessel potential space as an important example of realizations of data $u_0$ and $f$ in the next section.
The specific results will be explicitly presented in Corollaries \ref{p small coro}, \ref{q small coro}, and \ref{p big coro}.
\end{rem}

It is readily verifiable that the logarithmic Laplacian operator can be regarded as a specific example of the operator $\psi(t,-i\nabla)$.
Recall that the logarithmic Laplacian operator could be defined as follows
\begin{align*}
\cF[\log (-\Delta)u] = \log |\xi|^2\cF[u] 
\end{align*}
for a nice function $u$. 
This means that the logarithmic Laplacian operator $\log (-\Delta)$ is a pseudo-differential operator characterized by the symbol $\log |\xi|^2$.
The logarithmic Laplacian operator is an intriguing example of $\psi(t,-i \nabla)$ because its symbol is negative for $|\xi| <1$ and positive for $|\xi| >1$. Additionally, the symbol is not a polynomial.
In other words, the logarithmic Laplacian operator is a simple illustration of a pseudo-differential operator with a non-polynomial symbol that changes its sign.
As an application of Theorem \ref{weak solution thm}, we explore more generalized logarithmic-type operators with complex-valued coefficients, which naturally encompass the logarithmic Laplacian operator.
We now present an evolution equation involving a logarithmic-type operator as a specific case of \eqref{ab eqn} below:
\begin{equation}
\begin{cases}
								\label{log eqn}
\partial_tu(t,x)=\beta(t)\log \left( \psi_{exp}(t,-i \nabla) \right)u(t,x)+f(t,x),\quad &(t,x)\in(0,T)\times\mathbb{R}^d,\\
u(0,x)=u_0,\quad & x\in\mathbb{R}^d,
\end{cases}
\end{equation}
where
\begin{align}
								\label{20230807 20}
\log \left( \psi_{exp}(t,-i \nabla)  \right)u(t,x)
:=\cF^{-1}\left[\log \left(\psi_{exp}(t,\xi)\right)\cF[u(t,\cdot)] \right](x),
\end{align}
$\beta(t)$ is a complex-valued measurable function defined on $(0,T)$, and $\psi_{exp}(t,\xi)$ is a complex-valued measurable function defined  on $(0,T) \times \bR^d$.
In particular, if $\beta(t)=1$ and $\psi_{exp}(t,\xi) = |\xi|^2$ for all $t$, then
\begin{align*}
\beta(t)\log \left( \psi_{exp}(t,-i \nabla)  \right)u(t,x)
= \log (-\Delta) u(t,x).
\end{align*}
A comprehensive understanding of a weak solution $u$ to the equation \eqref{log eqn} can be achieved by referring to Definition \ref{space weak solution}.
Moreover, we assume
\begin{align}
							\label{exp make}
\psi_{exp}(t,\xi) \neq 0   \quad (a.e.)~ (t,\xi) \in (0,T) \times \bR^d.
\end{align}
It is essential for the right-hand side of \eqref{20230807 20} to be meaningful that the condition stated in \eqref{exp make} is satisfied. This requirement is necessary due to the presence of the logarithm's domain.
\begin{rem}
Let's define $\psi(t,\xi)$ as the natural logarithm of $\psi_{exp}(t,\xi)$, \text{i.e.}
$$
\psi(t,\xi)= \log \left(\psi_{exp}(t,\xi)\right).
$$
From a heuristic standpoint, we can express this relationship as
$$
\exp \left(\psi(t,\xi) \right)= \psi_{exp}(t,\xi),
$$
which explains the reason using the notation $\psi_{exp}(t,\xi)$. 
In other words, the subscript in $\psi_{exp}(t,\xi)$ signifies the exponential operation.
\end{rem}
We restate a theorem opting for \eqref{log eqn}, which is one of the major applications of Theorem \ref{weak solution thm}.
\begin{thm}
								\label{main log}
Let $u_0 \in \cF^{-1}L_{1,\ell oc}(\bR^d)$ and $f \in \cF^{-1}L_{1,1,t\text{-}loc,x\text{-}\ell oc}\left((0,T) \times \bR^d \right)$.
Assume that \eqref{exp make} holds and
\begin{align}
							\label{log con 0}
\int_0^t |\beta(s)| \mathrm{d}s < \infty \quad \forall t \in (0,T).
\end{align}
Additionally, suppose that 
\begin{align}
								\label{log con 1}
 \sup_{(s,\xi) \in [0,t) \times B_R } \left(\frac{1}{t-s}\int_s^t |\psi_{exp}(r,\xi)|^{(t-s) \Re[\beta(r)]} \mathrm{d}r\right) <\infty
\end{align}
and
\begin{align}
									\label{log con 2}
 \sup_{(s,\xi) \in [0,t) \times B_R } \left(\int_s^t |\beta(\rho)|\left(1+\left| \log \left(|\psi_{\exp}(\rho,\xi)|\right) \right|\right) \left(\frac{1}{\rho-s}\int_s^\rho\left(|\psi_{exp}(r,\xi)|^{(\rho-s) \Re[\beta(r)]}\right) \mathrm{d}r \right) 
 \mathrm{d}\rho  \right)  
<\infty
\end{align}
for all $t \in (0,T)$ and $R \in (0,\infty)$.
Then there exists a unique Fourier-space weak solution $u$ to \eqref{log eqn}  in 
$$
\cF^{-1}L_{\infty,1,t\text{-}loc, x\text{-}\ell oc}\left( (0,T) \times \bR^d \right) \cap \cF^{-1}L_{1,1,t\text{-}loc, x\text{-}\ell oc}\left( (0,T) \times \bR^d, |\beta(t)| \mathrm{d}t |\log \left(\psi_{exp}(t,\xi)\right)|\mathrm{d}\xi \right).
$$ 
\end{thm}
The proof of this theorem will be given in Section \ref{log thm pf}.

\begin{rem}
It is notable that the complex-valued function $\psi_{exp}(t,\xi)$ could have a negative real value.
It is possible because we have not imposed any regularity conditions on the symbol $\log \left(\psi_{exp}(t,\xi)\right)$, which means there is no need to consider the analytic continuation of the natural logarithm.
\end{rem}

\begin{rem}
The condition \eqref{log con 2} could be replaced by
\begin{align*}
 \sup_{(s,\xi) \in [0,t) \times B_R } \left(\int_s^t |\beta(\rho)|\left| \log \left(\psi_{\exp}(\rho,\xi)\right) \right| \left(\frac{1}{\rho-s}\int_s^\rho\left(|\psi_{exp}(r,\xi)|^{(\rho-s) \Re[\beta(r)]}\right) \mathrm{d}r \right) 
 \mathrm{d}\rho  \right)  
<\infty
\end{align*}
since it is obvious that
$$
\left(1+\left| \log \left(|\psi_{\exp}(\rho,\xi)|\right) \right|\right) \approx \left| \log \left(\psi_{\exp}(\rho,\xi)\right) \right|
$$
by considering a branch cut. 
\end{rem}

\begin{rem}
Consider the logarithmic Laplacian operator 
\begin{align*}
\log (-\Delta) u(t,x)
\end{align*}
as a special case of
\begin{align*}
\beta(t)\log \left( \psi_{exp}(t,-i \nabla)  \right)u(t,x)
\end{align*}
in Theorem \ref{main log}. 
It is easy to confirm that \eqref{log con 1} is valid for this simple case since $\beta(t)=1$ and $\psi_{exp}(t,\xi) =|\xi|^2$ for all values of $t$ and $\xi$. Nonetheless, it might seem that \eqref{log con 2} is not straightforward. 
However, it can be readily established through elementary calculations, as demonstrated in the proof of Corollary \ref{cor log second}.
Furthermore, we present a straightforward extension of this result for dealing with operators that result from a composition involving the $\log$ function and a general second-order differential operator, which is detailed in Corollary \ref{cor log second}.
\end{rem}

\vspace{3mm}
Next, we consider another important case handling a second-order differential operator without ellipticity.
We investigate a second-order evolution equation with complex-valued coefficients $a^{ij}(t)$, $b^i(t)$, and $c(t)$ as follows:
\begin{equation}
\begin{cases}
								\label{second eqn}
\partial_tu(t,x)=a^{ij}(t)u_{x^ix^j}(t,x)+b^j(t)u_{x^j}(t,x)+c(t)u(t,x)+f(t,x),\quad &(t,x)\in(0,T)\times\mathbb{R}^d,\\
u(0,x)=u_0,
\end{cases}
\end{equation}
where Einstein's summation convention is being enforced here.
Suppose that all coefficients $a^{ij}(t)$, $b^j(t)$, and $c(t)$ are  measurable and defined on $(0,T)$.
Formally, due to some properties of the Fourier transform and the inverse Fourier transform,
\begin{align*}
a^{ij}(t)u_{x^ix^j}(t,x)+b^j(t)u_{x^j}(t,x)+c(t)u(t,x)
=\cF^{-1}\left[ \left(-a^{ij}(t)\xi^i\xi^j +i b^j(t) + c(t)\right) \cF[u(t,\cdot)](\xi)\right](x).
\end{align*}
Hence, equation \eqref{second eqn} can be seen as a specific instance of equation \eqref{ab eqn} when considering the symbol:
\begin{align}
							\label{20240109 01}
\psi(t,\xi) = -a^{ij}(t)\xi^i\xi^j +i b^j(t)\xi^j + c(t),
\end{align}
where $i$ beside $b^j(t)$ is the imaginary unit so that $i^2=-1$. 
To maintain clarity and facilitate readers' comprehension, we provide a modified version of the definition of a Fourier-space weak solution for second-order cases. 
This indicates that the structure of our weak formulation closely parallels that of the classical one.
\begin{defn}[A weak solution tested by $\cF^{-1}\cD(\bR^d)$]
								\label{second weak}
Let $u \in L_{0}\left((0,T);\cF^{-1}\cD'(\bR^d)\right)$, $u_0 \in  \cF^{-1}\cD'(\bR^d)$,  and $f \in L_{1,t\text{-}loc}\left( (0,T) ;   \cF^{-1}\cD'(\bR^d) \right)$. Then we say that $u$ is a \emph{weak solution} (tested by $\cF^{-1}\cD(\bR^d)$) to \eqref{second eqn} if for any $\varphi \in \cF^{-1}\cD(\bR^d)$,
\begin{align}
								\notag
\langle u(t,\cdot), \varphi \rangle 
=\langle u_0, \varphi \rangle + \int_0^t \left(a^{ij}(s)\langle u(s,\cdot) , \varphi_{x^ix^j} \rangle  - b^j(s) \langle u(s,\cdot) , \varphi_{x^j} \rangle
+ c(s)\langle u(s,\cdot) , \varphi \rangle \right)\mathrm{d}s + \int_0^t \langle f(s,\cdot) , \varphi \rangle \mathrm{d}s \\
								\label{20230729 01}
(a.e.)~ t \in (0,T).
\end{align}
\end{defn}

\begin{rem}
Let us reiterate that a test function is conventionally given by a smooth function with compact support in a weak formulation.
In simpler terms, we commonly say that $u$ qualifies as a (classical) weak solution to equation \eqref{second eqn} if \eqref{20230729 01} is valid for all $\varphi \in \cD(\bR^d)$ instead of $\varphi \in \cF^{-1}\cD(\bR^d)$. To ensure the meaningfulness of this traditional weak formulation, it is essential for $u$ to be valued in $\cD'(\bR^d)$ and defined almost everywhere on $(0,T)$. 
Sufficiently assume that $u$ is a $\cS'(\bR^d)$-valued measurable function on $(0,T)$ and it is a classical weak solution.
Then this classical solution $u$ becomes identical with the weak solution (tested by $\cF^{-1}\cD(\bR^d)$) in Definition \ref{second weak} since $\cF^{-1}\cD(\bR^d)$ is dense within $\cS(\bR^d)$ concerning the topology defined by the Schwartz semi-norms.
However, it is crucial to highlight that for a solution $u$ to equation \eqref{second eqn}, $u(t,\cdot)$ does not belong to $\cD'(\bR^d)$ (and thus $u(t,\cdot) \notin \cS'(\bR^d)$) if there is no ellipticity in the leading coefficients $a^{ij}(t)$. 
This emphasizes the necessity for a novel concept of a weak solution, such as the Fourier-space weak solution.
\end{rem}
\begin{rem}
							\label{equivalent solution}
 The weak solution (tested by $\cF^{-1}\cD(\bR^d)$) is equivalent to the Fourier-space weak solution introduced in Definition \ref{space weak solution} for the particular symbol $\psi$ defined in \eqref{20240109 01} if the Fourier transform of $u$ is locally integrable (with respect to the space variable) since
 \begin{align*}
&a^{ij}(s)\langle u(s,\cdot) , \varphi_{x^ix^j} \rangle  - b^j(s) \langle u(s,\cdot) , \varphi_{x^j} \rangle
+ c(s)\langle u(s,\cdot) , \varphi \rangle  \\
&= a^{ij}(s)\left( \cF[u(s,\cdot)] , \cF[\varphi_{x^ix^j}] \right)_{L_2(\bR^d)}
-b^j(s) \left( \cF[u(s,\cdot)] , \cF[\varphi_{x^j}] \right)_{L_2(\bR^d)}
+c(s) \left( \cF[u(s,\cdot)] , \cF[\varphi] \right)_{L_2(\bR^d)} \\
&= \left( \left(-a^{ij}(s)\xi^i\xi^j +i b^j(s)\xi^j + c(s) \right)\cF[u(s,\cdot)] ,\cF[\varphi] \right)_{L_2(\bR^d)} \\
&= \left\langle  \psi(s,-i\nabla)u(s,\cdot) ,\varphi \right\rangle,
  \end{align*}
 which are easily derived from some properties of the Fourier transform.
\end{rem}
We are now prepared to present the theorem regarding equation \eqref{second eqn}. 
It is important to note that all the conditions specified for the coefficients in the theorem merely involve local integrability properties.
Especially, all coefficients could be complex-valued, and the matrix consisting of the leading coefficients $a^{ij}(t)$ does not need to be positive semi-definite.
\begin{thm}
								\label{main second}
Let $p,q \in [1,\infty]$, $W_0(\xi)$ and $W_1(\xi)$ be positive measurable functions on $\bR^d$, $W_2(t,\xi)$ be a positive measurable function on $(0,T) \times \bR^d$, $u_0 \in \cF^{-1}L_{q,\ell oc}(\bR^d,  W_0^q(\xi) \mathrm{d}\xi)$, and 
$$
f \in \cF^{-1}L_{p,q,t\text{-}loc, x \text{-}\ell oc}\left(  (0,T) \times \bR^d, \mathrm{d}t W_2^q(t,\xi)\mathrm{d}\xi\right).
$$
Assume that $a^{ij}, b^j, c \in L_{p,loc}\left( (0,T) \right)$ for all $i,j \in \{1,\ldots,d\}$, \textit{i.e.}
\begin{align}
								\label{second as 1}
\sum_{i=1}^d \sum_{j=1}^d  \left(\|a^{ij}\|_{L_p\left((0,t) \right)} + \|b^{j}\|_{L_p\left((0,t) \right)} + \|c\|_{L_p\left((0,t) \right)} \right) < \infty \quad \forall t \in (0,T).
\end{align}
Additionally, suppose that $W_1$ is a lower bound of $W_2$ and both $W_0$ and $W_1$ have local lower bounds, \textit{i.e.}
for each $t \in (0,T)$ and $R \in (0,\infty)$, there exist positive constants $\kappa_0(R)$ and $\kappa_1(R)$ so that
\begin{align}
								\label{second as 2}
W_0(\xi) \geq \kappa_0(R) \quad \forall \xi \in B_R
\end{align}
and
\begin{align}
								\label{second as 2-2}
W_2(s,\xi) \geq W_1(\xi) \geq \kappa_1(R) \quad \forall (s,\xi) \in (0,t) \times B_R.
\end{align}
Then there exists a unique weak solution $u$ $($tested by $\cF^{-1}\cD(\bR^d))$ to \eqref{second eqn}  in 
$$
 \cF^{-1}L_{p,q,t\text{-}loc, x\text{-}\ell oc}\left( (0,T) \times \bR^d, \mathrm{d}t \left(1+|a^{ij}(t)\xi^i\xi^j-ib^j(t)\xi^j -c(t)|\right)\mathrm{d}\xi \right)
$$
and the solution $u$ is given by
\begin{align*}
u(t,x)
&= \cF^{-1}\left[  \exp\left(\int_0^t\left( -a^{ij}(r)\xi^i\xi^j +i b^j(r)\xi^j + c(r) \right)\mathrm{d}r \right) \cF[u_0](\xi) \right](x) \\
&\quad +\cF^{-1}\left[ \int_0^t  \exp\left(\int_s^t\left( -a^{ij}(r)\xi^i\xi^j +i b^j(r)\xi^j + c(r) \right) \mathrm{d}r \right) \cF[f(s,\cdot)]\mathrm{d}s \right](x).
\end{align*}
Here the solution $u$ is given formally as discussed in Remark \ref{kernel remark}.
\end{thm}
We will provide the proof of this theorem in Section \ref{second thm pf}. 
It is worth noting that when dealing with data $u_0$ and $f$ in weighted Bessel potentials,  this theorem can be readily applied. 
We demonstrate these applications as corollaries of Theorem \ref{main second} in Section \ref{second thm pf}.

\mysection{Fourier transforms and weighted Bessel potential spaces} 
									\label{fourier section}

In this section, our objective is to present two types of weighted Bessel potential spaces. 
These spaces exhibit numerous intriguing properties that can be effectively characterized through the use of Fourier and inverse Fourier transforms. To begin, we recall the definitions and some properties of the Fourier and inverse Fourier transforms.
For a measurable function $f$ on $\bR^d$, we denote the $d$-dimensional Fourier transform of $f$ by 
\[
\cF[f](\xi) := \frac{1}{(2\pi)^{d/2}}\int_{\bR^{d}} \mathrm{e}^{-i\xi \cdot x} f(x) \mathrm{d}x
\]
and the $d$-dimensional inverse Fourier transform of $f$ by 
\[
\cF^{-1}[f](x) := \frac{1}{(2\pi)^{d/2}}\int_{\bR^{d}} \mathrm{e}^{ ix \cdot \xi} f(\xi) \mathrm{d}\xi.
\]									
Due to the Parseval-Plancherel identity, 
\begin{align*}
\|f\|_{L_2(\bR^d)}=\|\cF[f]\|_{L_2(\bR^d)} = \|\cF^{-1}[f]\|_{L_2(\bR^d)} \quad \forall f \in L_1(\bR^d) \cap L_2(\bR^d).
\end{align*}
The above equalities imply that the Fourier and inverse Fourier transforms are $L_2(\bR^d)$-isometries 
by considering $L_2$-extensions based  on completeness of $L_2(\bR^d)$.
Additionally, due to the Riesz–Thorin theorem, there exist the extensions of the Fourier and inverse Fourier transforms which are bounded from $L_1(\bR^d)+L_2(\bR^d)$ to $L_\infty(\bR^d)+L_2(\bR^d)$.
In particular, for any $\theta \in [0,1]$,
\begin{align*}
\|\cF[f]\|_{L_{2/\theta}} \leq \left(\frac{1}{2\pi}\right)^{\frac{d(1-\theta)}{2}} \|f\|_{L_{2/(2-\theta) }} 
\quad \forall f \in L_{2 / (2-\theta) }
\end{align*}
and
\begin{align*}
\|\cF^{-1}[f]\|_{L_{ 2 / \theta}} \leq \left(\frac{1}{2\pi}\right)^{\frac{d(1-\theta)}{2}} \|f\|_{L_{2 / (2-\theta) }}
\quad \forall f \in L_{2/(2-\theta) },
\end{align*}
where $\frac{2}{0}:= \infty$.
Therefore, for any $p \in [1,2]$ and $f \in L_p(\bR^d)$, 
it is possible to regard the Fourier transform and the inverse Fourier transform of $f$ as functions within $L_{p'}(\bR^d)$,
where $p'=\frac{p}{p-1}$ and $\frac{1}{0}:=\infty$.
This is a crucial foundation for our theory, as it allows us to address our solutions and data in a strong sense.
To be more specific, for any $p \in [2,\infty]$ and any function $f$ in $L_{p'}(\bR^d)$, we have the following inequalities:
\begin{align}
								\label{Riesz Thorin inequality 1}
\|\cF[f]\|_{L_p} \leq \left(\frac{1}{2\pi}\right)^{\frac{d(p-2)}{2p}} \|f\|_{L_{p'}} 
\end{align}
and
\begin{align}
								\label{Riesz Thorin inequality 2}
\|\cF^{-1}[f]\|_{L_p} \leq \left(\frac{1}{2\pi}\right)^{\frac{d(p-2)}{2p}} \|f\|_{p'}.
\end{align}

We now turn our attention to the Fourier and inverse Fourier transforms applied to the space of tempered distributions on $\bR^d$.
These transformations are established in a weak sense, making use of the advantageous properties of the Schwartz class $\cS(\bR^d)$.
Suppose we have a function $f$ belonging to $\cS'(\bR^d)$. In that case, we define the Fourier and inverse Fourier transforms as follows:
\begin{align*}
\langle \cF[f] , \varphi \rangle = \langle f, \cF^{-1}[\varphi] \rangle \quad \forall \varphi \in \cS(\bR^d)
\end{align*}
and
\begin{align*}
\langle \cF^{-1}[f] , \varphi \rangle = \langle f, \cF[\varphi] \rangle  \quad \forall \varphi \in \cS(\bR^d).
\end{align*}
Here, the notation $\langle f, \cF^{-1}[\varphi] \rangle$ represents the duality pairing, which means that $\langle f, \cF^{-1}[\varphi] \rangle = f \left( \cF^{-1}[\varphi]\right)$ is the value assigned to $\cF^{-1}[\varphi]$ under the action of the linear functional $f$.
It is important to note that the Fourier and inverse Fourier transforms of $f$ once again become tempered distributions on $\bR^d$, thanks to the well-established properties of functions within the Schwartz class. 
Furthermore, it is worth noting that $(1+|\cdot|^2)^{\gamma_2/2}f$ remains a tempered distribution as well if we define the action on test functions as follows:
\begin{align*}
\left\langle (1+|\cdot|^2)^{\gamma_2/2}f, \varphi \right\rangle
:=\left\langle f, (1+|\cdot|^2)^{\gamma_2/2} \varphi(\cdot) \right\rangle \quad \text{for all} \quad \varphi \in \cS(\bR^d).
\end{align*}
This is possible because the function $(1+|x|^2)^{\gamma_2/2} \varphi(x)$ belongs to the Schwartz class $\cS(\bR^d)$ for any $\varphi \in \cS(\bR^d)$.
As a result, these operations on tempered distributions enable us to establish the Bessel potential spaces.

Let $\gamma \in \bR$ and $p \in [1,\infty]$.
Recall that we define the Bessel potential space with order $\gamma$ and exponent $p$, denoted by $H_p^\gamma(\bR^d)$, as the class of all tempered distributions $f$ on $\bR^d$ satisfying the condition:
\begin{align*}
(I-\Delta)^{\gamma/2}f := \cF^{-1} \left[\left( 1+|\cdot|^2 \right)^{\gamma/2} \cF[f] \right] \in L_p(\bR^d).
\end{align*}
For a more comprehensive understanding, further details on these concepts can be found in references such as \cite{Hormander 1990} and \cite{Grafakos 2014}.
Now, we proceed to introduce two generalizations of the Bessel potential spaces with additional weighting factors below. 
We call these spaces an \textit{inner weighted Bessel potential space} and an \textit{outer weighted Bessel potential space}, respectively. 
It is worth noting that we have not found any previous literature that discusses these types of weighted Bessel potential space in general. 
However, certain outer weighted Bessel potential spaces can be regarded as particular cases of weighted Bessel potential spaces with Muckenhoupt's weight.
We will provide further clarification of this relationship before presenting Proposition \ref{outer embedding} later on.
We are now ready to present the definitions for both inner and outer weighted Bessel potential spaces.

\begin{defn}[Inner weighted Bessel potential space]
Let $\gamma_1, \gamma_2 \in \bR$ and $p \in [1,\infty]$.
We use the notation $H_{p,in}^{\gamma_1,\gamma_2}(\bR^d)$ to denote the space of all tempered distributions $f$ on $\bR^d$ such that
\begin{align*}
(I-\Delta)^{\gamma_1/2}\left(  (1+|\cdot|^2)^{\gamma_2/2}f \right) := \cF^{-1}\left[ \left( 1+|\cdot|^2 \right)^{\gamma_1/2} \cF\left[ (1+|\cdot|^2)^{\gamma_2/2} f\right]  \right]\in L_p(\bR^d),
\end{align*}
which implies that there exists a $g \in L_p(\bR^d)$ such that
\begin{align*}
\left\langle  \cF^{-1} \left[\left( 1+|\cdot|^2 \right)^{\gamma_1/2} \cF \left[\left(1+|\cdot|^2\right)^{\gamma_2/2}f\right]  \right], \varphi \right\rangle 
&:=
\left\langle  f ,  \left(1+|\cdot|^2\right)^{\gamma_2/2} (I-\Delta)^{\gamma_1/2}\varphi \right\rangle \\
&= \int_{\bR^d} g(x) \varphi(x) \mathrm{d}x \qquad \forall \varphi \in \cS(\bR^d).
\end{align*}
Then $H_{p,in}^{\gamma_1,\gamma_2}(\bR^d)$ becomes a Banach space with the norm
\begin{align*}
\|f\|_{H_{p,in}^{\gamma_1,\gamma_2}(\bR^d)}
:= \left\|(I-\Delta)^{\gamma_1/2}\left(  (1+|\cdot|^2)^{\gamma_2/2}f \right) \right\|_{L_p(\bR^d)}.
\end{align*}
We say that $\gamma_1$, $\gamma_2$, and $p$ are the \textit{regularity exponent}, \textit{weight exponent}, and \textit{integrability exponent} of $H_{p,in}^{\gamma_1,\gamma_2}(\bR^d)$, respectively.
\end{defn}

\begin{defn}[Outer weighted Bessel potential space]
Let $\gamma_1, \gamma_2 \in \bR$ and $p \in [1,\infty]$.
We use the notation $H_{p,out}^{\gamma_1,\gamma_2}(\bR^d)$ to denote the space of all tempered distributions $f$ on $\bR^d$ such that
\begin{align*}
(1+|\cdot|^2)^{\gamma_2/2}(I-\Delta)^{\gamma_1/2}  f  \in L_p(\bR^d).
\end{align*}
Then $H_{p,out}^{\gamma_1,\gamma_2}(\bR^d)$ becomes a Banach space with the norm
\begin{align*}
\|f\|_{H_{p,out}^{\gamma_1,\gamma_2}(\bR^d)}
:= \left\|\left(1+|\cdot|^2\right)^{\gamma_2/2} (I-\Delta)^{\gamma_1/2}f \right\|_{L_p(\bR^d)}.
\end{align*}
The same terminology is used for exponents $\gamma_1$, $\gamma_2$, and $p$.
In other words, $\gamma_1$, $\gamma_2$, and $p$ are called the \textit{regularity exponent}, \textit{weight exponent}, and \textit{integrability exponent} of $H_{p,out}^{\gamma_1,\gamma_2}(\bR^d)$, respectively.
\end{defn}
We would not delve into the specific details demonstrating that these two spaces are indeed Banach spaces. 
This can be readily proved by relying on the completeness of $L_p$-spaces.
In addition, it is obvious that 
$$
H_p^{\gamma_1}(\bR^d) 
=H_{p,in}^{\gamma_1,0}(\bR^d)
=H_{p,out}^{\gamma_1,0}(\bR^d)
$$
and
$$
H_{p,in}^{0,\gamma_2}(\bR^d)
=H_{p,out}^{0,\gamma_2}(\bR^d).
$$

\begin{rem}
Let $\gamma_1>0$.
It is a well-known fact that the Bessel potentials $(I-\Delta)^{-\gamma_1/2}$ have elegant representations in terms of Green functions (as discussed in \cite[Section 1.2.2]{Grafakos 2014-2}). In essence, this means that there exists a smooth function $G_{\gamma_1}(x)$ defined on $\bR^d$ except for the origin, such that 
\begin{align*}
(I-\Delta)^{-\gamma_1/2} g (x)
=\int_{\bR^d} G_{\gamma_1}(y) g(x-y) \mathrm{d}y.
\end{align*}
for any function $g$ in $L_p(\bR^d)$.
For this particular case, it is possible to understand realizations of all tempered distributions within these weighted spaces solely through integrals. In other words, a tempered distribution $f$ on $\bR^d$ belongs to $H_{p,in}^{\gamma_1,\gamma_2}(\bR^d)$ if and only if there exists a function $g$ in $L_p(\bR^d)$ such that
\begin{align*}
f(x)
=(1+|x|^2)^{-\gamma_2/2}\int_{\bR^d} G_{\gamma_1}(y) g(x-y) \mathrm{d}y.
\end{align*}
Similarly,
$f \in H_{p,out}^{\gamma_1,\gamma_2}(\bR^d)$ if and only if
there exist a $g \in L_p(\bR^d)$ such that
\begin{align*}
f(x)
=\int_{\bR^d} G_{\gamma_1}(y) (1+|x-y|^2)^{-\gamma_2/2}g(x-y) \mathrm{d}y.
\end{align*}
The function $G_{\gamma_1}(x)$ is given by
\begin{align*}
G_{\gamma_1}(x) 
= \frac{1}{\Gamma(\frac{\gamma_1}{2})}\int_0^\infty \mathrm{e}^{-t} \mathrm{e}^{-|x|^2/(4t)}t^{(\gamma_1-d)/2} \frac{ \mathrm{d}t}{t},
\end{align*}
where $\Gamma$ denotes the Gamma function.
Additionally, the behaviors of $G$ is characterized as follows (\cite[Proposition 1.2.5]{Grafakos 2014-2}):
For large values of $|x|$, it satisfies
\begin{align*}
G_{\gamma_1}(x) \lesssim \mathrm{e}^{-|x|^2/2}
\end{align*}
and  for small values of $|x|$, the behavior of $G_{\gamma_1}(x)$ can be described by the following cases according to the dimension $d$:
\begin{align*}
G_{\gamma_1}(x) 
\approx 
\begin{cases}
&|x|^{\gamma_1 -d}  + 1 + O(|x|^{\gamma_1-d+2}) \quad \text{if $0<\gamma_1 <d$} \\
&\log \frac{2}{|x|} + 1 + O(|x|^2) \quad \text{if $\gamma_1 =d$} \\
& 1+ O(|x|^{\gamma_1 -d}) \quad \text{if $\gamma_1 >d$}.
\end{cases}
\end{align*}
\end{rem}
We establish a few straightforward embedding inequalities for these two weighted spaces, beginning with the inner spaces.
The first embedding presented below is effective concerning the regularity exponents of inner spaces.
\begin{prop}
						\label{inner embedding}
Let $\gamma_1,\gamma_2,\tilde \gamma_1\in \bR$,  $p \in (1,\infty]$, and $f \in H_{p,in}^{\gamma_1,\gamma_2}(\bR^d)$.
Assume that $\gamma_1 \geq \tilde \gamma_1$.
Then 
\begin{align*}
\left\| f\right\|_{H_{p,in}^{\tilde \gamma_1,\gamma_2}(\bR^d)}
\leq \left\| f\right\|_{H_{p,in}^{\gamma_1,\gamma_2}(\bR^d)}.
\end{align*}

\end{prop}
\begin{proof}
Let $\varphi \in \cS(\bR^d)$.
Then by H\"older's inequality and an $L_p$-boundedness of the Bessel potential (\emph{cf}. \cite[Section 1.2.2]{Grafakos 2014-2}), 
\begin{align*}
\left|\left\langle f ,\varphi \right\rangle \right|
&=\left|\left\langle  (I-\Delta)^{\gamma_1/2}\left( \left(1+|\cdot|^2\right)^{\gamma_2/2}  f \right), (I-\Delta)^{-\gamma_1/2} \left(\left(1+|\cdot|^2\right)^{-\gamma_2/2}\varphi \right) \right\rangle \right| \\
&\leq  \left\| f\right\|_{H_{p,in}^{\gamma_1,\gamma_2}(\bR^d)}\left\|(I-\Delta)^{-\gamma_1/2} \left(\left(1+|\cdot|^2\right)^{-\gamma_2/2}\varphi \right)\right\|_{L_{p'}(\bR^d)}  \\
&=  \left\| f\right\|_{H_{p,in}^{\gamma_1,\gamma_2}(\bR^d)}\left\|(I-\Delta)^{(\tilde \gamma_1-\gamma_1)/2} \left(I-\Delta \right)^{-\tilde \gamma_1/2}\left(\left(1+|\cdot|^2\right)^{-\gamma_2/2}\varphi \right) \right\|_{L_{p'}(\bR^d)}  \\
&\leq  \left\| f\right\|_{H_{p,in}^{\gamma_1,\gamma_2}(\bR^d)}\left\| \left(I-\Delta \right)^{-\tilde \gamma_1/2}\left(\left(1+|\cdot|^2\right)^{-\gamma_2/2}\varphi \right) \right\|_{L_{p'}(\bR^d)}. 
\end{align*} 
Thus considering
\begin{align*}
\left(1+|\cdot|^2\right)^{\gamma_2/2} \left(I-\Delta \right)^{\tilde \gamma_1/2}\varphi
\end{align*}
instead of $\varphi$, we have
\begin{align*}
\left|\left\langle \left(I-\Delta \right)^{\tilde \gamma_1/2} \left(1+|\cdot|^2\right)^{\gamma_2/2} f ,\varphi \right\rangle\right|
&=\left|\left\langle  f , \left(1+|\cdot|^2\right)^{\gamma_2/2} \left(I-\Delta \right)^{\tilde \gamma_1/2}\varphi \right\rangle\right| \\
&\leq  \left\| f\right\|_{H_{p,in}^{\gamma_1,\gamma_2}(\bR^d)}\left\| \varphi \right\|_{L_{p'}(\bR^d)}.
\end{align*}
Taking the supremum over all $\varphi $ such that $\|\varphi  \|_{L_{p'}(\bR^d)}=1$, we have
\begin{align*}
\left\| f\right\|_{H_{p,in}^{\tilde \gamma_1,\gamma_2}(\bR^d)}
\leq \left\| f\right\|_{H_{p,in}^{\gamma_1,\gamma_2}(\bR^d)}.
\end{align*}
\end{proof}
An analogous embedding inequality is valid for outer spaces. 
This embedding is provided in relation to the weight exponents of outer spaces.
\begin{prop}
							\label{outer embedding}
Let $\gamma_1,\gamma_2, \tilde \gamma_2 \in \bR$,  $p \in (1,\infty]$, and $f \in H_{p,out}^{ \gamma_1,\gamma_2}(\bR^d)$.
Assume that $\gamma_2 \geq \tilde \gamma_2$.
Then
\begin{align*}
\left\| f\right\|_{H_{p,out}^{ \gamma_1,\tilde \gamma_2}(\bR^d)}
\leq \left\| f\right\|_{H_{p,out}^{\gamma_1,\gamma_2}(\bR^d)}.
\end{align*}
\end{prop}
\begin{proof}
Let $\varphi \in \cS(\bR^d)$. Then
\begin{align*}
\left|\left\langle f ,\varphi \right\rangle \right|
&=\left|\left\langle  \left( 1+ |\cdot|^2 \right)^{\gamma_2/2} (I-\Delta)^{\gamma_1/2}f ,  \left( 1+ |\cdot|^2 \right)^{-\gamma_2/2}(I-\Delta)^{-\gamma_1/2}\varphi \right\rangle \right| \\
&\leq \left\| f \right\|_{H_{p,out}^{\gamma_1,\gamma_2}} \left\| \left( 1+ |\cdot|^2 \right)^{-\gamma_2/2}(I-\Delta)^{-\gamma_1/2}\varphi \right\|_{L_{p'}(\bR^d)} \\
&\leq \left\| f \right\|_{H_{p,out}^{\gamma_1,\gamma_2}} \left\| \left( 1+ |\cdot|^2 \right)^{(\tilde \gamma_2-\gamma_2)/2} \left( 1+ |\cdot|^2 \right)^{-\tilde \gamma_2/2}(I-\Delta)^{-\gamma_1/2}\varphi \right\|_{L_{p'}(\bR^d)} \\
&\leq \left\| f \right\|_{H_{p,out}^{\gamma_1,\gamma_2}} \left\|  \left( 1+ |\cdot|^2 \right)^{-\tilde \gamma_2/2}(I-\Delta)^{-\gamma_1/2}\varphi \right\|_{L_{p'}(\bR^d)}.
\end{align*}
Thus
\begin{align*}
\left|\left\langle \left( 1+ |\cdot|^2 \right)^{\tilde \gamma_2/2}(I-\Delta)^{\gamma_1/2} f ,\varphi \right\rangle \right|
&=\left|\left\langle f , (I-\Delta)^{\gamma_1/2} \left( 1+ |\cdot|^2 \right)^{\tilde \gamma_2/2} \varphi \right\rangle \right|\\
&\leq \left\| f \right\|_{H_{p,out}^{\gamma_1,\gamma_2}} \left\| \varphi \right\|_{L_{p'}(\bR^d)}.
\end{align*}
Finally, taking the supremum over all $\varphi $ such that $\|\varphi  \|_{L_{p'}(\bR^d)}=1$, we obtain the result.
The proposition is proved. 
\end{proof}
\begin{rem}
The case $p=1$ is excluded in both Propositions \ref{inner embedding} and \ref{outer embedding}
since the dual space of $L_\infty(\bR^d)$ is strictly larger than $L_1(\bR^d)$. 
\end{rem}

The two embedding inequalities mentioned above appear quite straightforward, as evidenced by their proofs. 
However, we will now delve into a non-trivial embedding that relies on a weighted multiplier. 
To begin, let's revisit the definition of Muckenhoupt's weight.
We say that a non-negative measurable function $w$ on $\bR^d$ belongs to $A_p$ with $p \in (1,\infty)$ if
\begin{align*}
\sup \left( \frac{1}{|B|} \int_B w(x) \mathrm{d}x \right)\left( \frac{1}{|B|} \int_B w(x)^{-1/(p-1)} \mathrm{d}x \right)^{p-1} < \infty,
\end{align*} 
where the supremum is taken over all Euclidean balls on $\bR^d$ and $|B|$ denotes the Lebesgue measure of $B$.
It is well-known that the mapping $\xi \in \bR^d \mapsto |\xi|^\gamma$ is in $A_p$  if and only if $\gamma \in (-d ,d(p-1))$
(\textit{cf}. \cite[Example 7.1.7]{Grafakos 2014}).
Thus it is easy to verify that $(1+|\xi|^2)^{\gamma/2}$ is in $A_p$ for all $\gamma \in (-d ,d(p-1))$.

It is noteworthy that there is a scarcity of research papers dedicated to the exploration of weighted Bessel potential spaces with Muckenhoupt's weights, despite these spaces representing natural extensions of weighted Sobolev spaces. 
Furthermore, handling these spaces becomes relatively straightforward by applying weighted multiplier theories in conjunction with classical Bessel potential spaces.

For additional insights and properties of weighted Bessel potential spaces with Muckenhoupt's weights, we recommend referring to \cite[Appendix]{CJH KID 2023} and \cite[Section 3]{Kartin 2009}.
The proposition we are about to present can be derived as a specific case of the results cited above.
Nonetheless, we provide a simple proof to illustrate that it is an easily applicable result within the context of a weighted multiplier theory.

\begin{prop}
							\label{outer embedding 2}
Let $\gamma_1, \tilde \gamma_1 \in \bR$, $p \in (1,\infty)$, $\gamma_2 \in \left(-\frac{d}{p} ,\frac{d(p-1)}{p}\right)$, and $f \in H_{p,out}^{\gamma_1,\gamma_2}(\bR^d)$.
Assume that $\gamma_1 \geq \tilde \gamma_1$.
Then
\begin{align*}
\left\| f\right\|_{H_{p,out}^{ \tilde \gamma_1, \gamma_2}(\bR^d)}
\lesssim \left\| f\right\|_{H_{p,out}^{\gamma_1,\gamma_2}(\bR^d)}.
\end{align*}
\end{prop}
\begin{proof}
It is obvious that $(1+|\xi|^2)^{ (p\gamma_2)/2}$ is in $A_p$ since $-d < p\gamma_2 <d(p-1)$.
Moreover, it is also easy to show that $(1+|\xi|^2)^{(\tilde \gamma_1 - \gamma_1)/2}$ is a weighted $L_p$-multiplier 
since $\gamma_1 \geq \tilde \gamma_1$ (\textit{cf}. \cite{Kurtz 1980}).
Therefore, we have
\begin{align*}
\left\| f\right\|_{H_{p,out}^{ \tilde \gamma_1, \gamma_2}(\bR^d)}
&=\left\|\left(1+|\cdot|^2\right)^{\gamma_2/2} (I-\Delta)^{\tilde \gamma_1/2}f \right\|_{L_p(\bR^d)} \\
&=\left\|\left(1+|\cdot|^2\right)^{\gamma_2/2} (I-\Delta)^{(\tilde \gamma_1 -\gamma_1)/2}(I-\Delta)^{\gamma_1/2} f \right\|_{L_p(\bR^d)} \\
&=\left\| (I-\Delta)^{(\tilde \gamma_1 -\gamma_1)/2}(I-\Delta)^{\gamma_1/2} f (x)\right\|_{L_p(\bR^d,\left(1+|x|^2\right)^{(p\gamma_2)/2} \mathrm{d}x)} \\
&\lesssim \left\|(I-\Delta)^{\gamma_1/2} f (x) \right\|_{L_p(\bR^d,\left(1+|x|^2\right)^{(p\gamma_2)/2} \mathrm{d}x)} \\
&= \left\|\left(1+|\cdot|^2\right)^{\gamma_2/2} (I-\Delta)^{\gamma_1/2} f \right\|_{L_p(\bR^d)} \\
&= \left\| f\right\|_{H_{p,out}^{\gamma_1,\gamma_2}(\bR^d)}.
\end{align*}
The proposition is proved.
\end{proof}
Subsequently, we highlight an evident observation that both Fourier and inverse Fourier transforms operate similarly within these weighted spaces. 
This observation will be employed extensively in the paper without repeated explicit mention.
\begin{prop}
					\label{inverse identity}
Let $\gamma_1, \gamma_2 \in \bR$, $p \in [1,\infty]$, and $f$ be a tempered distribution on $\bR^d$.
\begin{enumerate}[(i)]
\item
\begin{align*}
\cF[f] \in H_{p,out}^{\gamma_1,\gamma_2}(\bR^d)
~\text{if and only if}~\cF^{-1}[f] \in H_{p,out}^{\gamma_1,\gamma_2}(\bR^d).
\end{align*}
Additionally,  
\begin{align*}
\left\|\cF[f] \right\|_{H_{p,out}^{\gamma_1,\gamma_2}(\bR^d)}
=\left\|\cF^{-1}[f] \right\|_{H_{p,out}^{\gamma_1,\gamma_2}(\bR^d)}.
\end{align*}
\item 
\begin{align*}
\cF[f] \in H_{p,in}^{\gamma_1,\gamma_2}(\bR^d)
~\text{if and only if}~\cF^{-1}[f] \in H_{p,in}^{\gamma_1,\gamma_2}(\bR^d).
\end{align*}
Additionally,
\begin{align*}
\left\|\cF[f] \right\|_{H_{p,in}^{\gamma_1,\gamma_2}(\bR^d)}
=\left\|\cF^{-1}[f] \right\|_{H_{p,in}^{\gamma_1,\gamma_2}(\bR^d)}.
\end{align*}
\end{enumerate}
\end{prop}
\begin{proof}
It is trivial since
\begin{align*}
\cF[\varphi](\xi) = \cF^{-1}[\varphi](-\xi).
\end{align*}
for all $\varphi \in \cS(\bR^d)$ and $\xi \in \bR^d$.
\end{proof}
We also investigate relations between inner and outer spaces using Riesz-Thorin inequalities \eqref{Riesz Thorin inequality 1} and \eqref{Riesz Thorin inequality 2}.
The Fourier transform operates as a bridge by interchanging regularity and weight exponents when we connect inner and outer spaces.
It is important to note that the range of  $p$ naturally becomes restrictive in accordance with the constraints imposed by Riesz-Thorin's theorem.
\begin{prop}
								\label{20230701 10}
Let $p \in [1,2]$, $\gamma_1, \gamma_2 \in \bR$, and $f$ be a tempered distribution on $\bR^d$.
\begin{enumerate}[(i)]
\item If $f \in H_{p,in}^{\gamma_1,\gamma_2}(\bR^d)$, then
\begin{align*}
 \|\cF[f]\|_{H^{\gamma_2,\gamma_1}_{p',out}(\bR^d)} \leq \left(\frac{1}{2\pi}\right)^{\frac{d(2-p)}{2p}} \|f\|_{H^{\gamma_1,\gamma_2}_{p,in}(\bR^d)}.
\end{align*}
\item If $\cF[f] \in H_{p,in}^{\gamma_1,\gamma_2}(\bR^d)$, then
\begin{align*}
 \|f\|_{H^{\gamma_2,\gamma_1}_{p',out}(\bR^d)} \leq \left(\frac{1}{2\pi}\right)^{\frac{d(2-p)}{2p}} \|\cF[f]\|_{H^{\gamma_1,\gamma_2}_{p,in}(\bR^d)}.
\end{align*}
\end{enumerate}
Here $p'$ denotes the H\"older conjugate of $p$.
\end{prop}
\begin{proof}
First, we prove (i).
Let $\varphi \in \cS(\bR^d)$.
By H\"older's inequality,
\begin{align*}
\left|\left\langle \cF[f] , \varphi \right\rangle \right|
&=\left|\left\langle  (I-\Delta)^{\gamma_1/2}\left( \left(1+|\cdot|^2\right)^{\gamma_2/2}  f \right), (I-\Delta)^{-\gamma_1/2} \left(\left(1+|\cdot|^2\right)^{-\gamma_2/2}\cF^{-1}[\varphi] \right) \right\rangle \right|\\
&=\left|\int_{\bR^d}  (I-\Delta)^{\gamma_1/2}\left( \left(1+|\cdot|^2\right)^{\gamma_2/2}  f \right)(x) (I-\Delta)^{-\gamma_1/2} \left(\left(1+|\cdot|^2\right)^{-\gamma_2/2}\cF^{-1}[\varphi] \right)(x) \mathrm{d}x \right|\\
&\leq \|f\|_{H^{\gamma_1,\gamma_2}_{p,in}(\bR^d)} \left\|(I-\Delta)^{-\gamma_1/2} \left(\left(1+|\cdot|^2\right)^{-\gamma_2/2}\cF^{-1}[\varphi] \right)\right\|_{L_{p'}(\bR^d)}.
\end{align*}
Additionally, since $p' \in [2,\infty]$, we have
\begin{align*}
\left\|(I-\Delta)^{-\gamma_1/2} \left(\left(1+|\cdot|^2\right)^{-\gamma_2/2}\cF^{-1}[\varphi] \right) \right\|_{L_{p'}(\bR^d)}
\leq \left(\frac{1}{2\pi}\right)^{\frac{d(2-p)}{2p}}   \left\| \left(1+|x|^2\right)^{-\gamma_1/2} \left(I-\Delta \right)^{-\gamma_2/2}\varphi \right\|_{L_p (\bR^d, \mathrm{d}x)}.
\end{align*}
Considering $\left(I-\Delta \right)^{\gamma_2/2} \left(\left(1+|\cdot|^2\right)^{\gamma_1/2} \varphi \right)(x)$ instead of $\varphi(x)$, we have
\begin{align*}
\left| \left\langle  \left(1+|\cdot|^2\right)^{\gamma_1/2} \left(I-\Delta \right)^{\gamma_2/2} \cF[f] , \varphi \right\rangle \right|
&=\left| \left\langle \cF[f] , \left(I-\Delta \right)^{\gamma_2/2} \left(\left(1+|\cdot|^2\right)^{\gamma_1/2} \varphi \right) \right\rangle \right| \\
&\leq \left(\frac{1}{2\pi}\right)^{\frac{d(2-p)}{2p}} \|f\|_{H^{\gamma_1,\gamma_2}_{p,in}(\bR^d)}  \|\varphi \|_{L_p(\bR^d)}.
\end{align*}
Therefore taking the supremum over all $\varphi \in L_p(\bR^d)$ so that $\|\varphi\|_{L_p(\bR^d)}=1$, we have
\begin{align*}
 \|\cF[f]\|_{H^{\gamma_2,\gamma_1}_{p',out}(\bR^d)} \leq \left(\frac{1}{2\pi}\right)^{\frac{d(2-p)}{2p}} \|f\|_{H^{\gamma_1,\gamma_2}_{p,in}(\bR^d)}.
\end{align*}
Next, we prove (ii).
However, it is an easy consequence of $(i)$ by taking $\cF[f]$ instead of $f$ due to Proposition \ref{inverse identity}(i) and the Fourier inversion theorem.
\end{proof}

The Fourier and inverse Fourier transforms exhibit particularly favorable behaviors as isometries within the $L_2(\mathbb{R}^d)$-space. This interesting characteristic can also be extended to our weighted Bessel potential spaces. 
More specifically, when the exponent of integrability $p$ equals 2, the Fourier transform $\mathcal{F}$ acts as an isometry preserving the norms between $H_{2,in}^{\gamma_1,\gamma_2}(\mathbb{R}^d)$ and $H_{2,out}^{\gamma_2,\gamma_1}(\mathbb{R}^d)$.
To begin, we act the Fourier transform on elements in $H_{2,in}^{\gamma_1,\gamma_2}(\mathbb{R}^d)$.
\begin{prop}
								\label{l2 isometry}
Let  $\gamma_1, \gamma_2 \in \bR$.
Then
\begin{align*}
f \in H_{2,in}^{\gamma_1,\gamma_2}(\bR^d) ~\text{if and only if}~\cF[f] \in H_{2,out}^{\gamma_2,\gamma_1}(\bR^d).
\end{align*}
Additionally,  
\begin{align*}
\left\|f \right\|_{H_{2,in}^{\gamma_1,\gamma_2}(\bR^d)}
=\left\|\cF[f] \right\|_{H_{2,out}^{\gamma_2,\gamma_1}(\bR^d)}.
\end{align*}
\end{prop}
\begin{proof}
In Proposition \ref{20230701 10}(i), it is already shown that 
\begin{align*}
 \|\cF[f]\|_{H^{\gamma_2,\gamma_1}_{2,out}(\bR^d)} \leq \|f\|_{H^{\gamma_1,\gamma_2}_{2,in}(\bR^d)} \quad \forall f \in H^{\gamma_1,\gamma_2}_{2,in}(\bR^d).
\end{align*}
Thus it suffices to show that
\begin{align*}
\|f\|_{H^{\gamma_1,\gamma_2}_{2,in}(\bR^d)} \leq \|\cF[f]\|_{H^{\gamma_2,\gamma_1}_{2,out}(\bR^d)} 
\end{align*}
for all $f \in \cS'(\bR^d)$ such that $\cF[f] \in H^{\gamma_2,\gamma_1}_{2,out}(\bR^d)$.
Let $f \in \cS'(\bR^d)$ so that $\cF[f] \in H^{\gamma_2,\gamma_1}_{2,out}(\bR^d)$.
Then
\begin{align*}
\left|\left\langle f , \varphi \right\rangle\right|
&=\left|\left\langle \left( \left(1+|\cdot|^2\right)^{\gamma_1/2} (I-\Delta)^{\gamma_2/2}  \cF[f] \right),  \left(\left(1+|\cdot|^2\right)^{-\gamma_1/2} (I-\Delta)^{-\gamma_2/2}\cF[\varphi] \right) \right\rangle \right| \\
&=\left|\int_{\bR^d}  \left(1+|\cdot|^2\right)^{\gamma_1/2} (I-\Delta)^{\gamma_2/2}  \cF[f] (x)\left(1+|\cdot|^2\right)^{-\gamma_1/2} (I-\Delta)^{-\gamma_2/2} \cF[\varphi](x) \mathrm{d}x \right| \\
&\leq \|\cF[f]\|_{H^{\gamma_2,\gamma_1}_{2,out}(\bR^d)} \left\|\left(\left(1+|\cdot|^2\right)^{-\gamma_2/2} (I-\Delta)^{-\gamma_1/2} \cF[\varphi] \right)\right\|_{L_{2}(\bR^d)} \\
&= \|\cF[f]\|_{H^{\gamma_2,\gamma_1}_{2,out}(\bR^d)} \left\| (I-\Delta)^{-\gamma_2/2} \left(1+|\cdot|^2\right)^{-\gamma_1/2} \varphi \right\|_{L_{2}(\bR^d)},
\end{align*}
where Plancherel's theorem is used in the last equality. 
Thus considering $\left(1+|\cdot|^2\right)^{\gamma_2/2} (I-\Delta)^{\gamma_1/2} \varphi$ instead of $\varphi$, we have
\begin{align*}
\left|\left\langle (I-\Delta)^{\gamma_1/2} \left(1+|\cdot|^2\right)^{\gamma_2/2}  f ,  \varphi \right\rangle\right|
=\left|\left\langle f ,  \left(1+|\cdot|^2\right)^{\gamma_2/2} (I-\Delta)^{\gamma_1/2} \varphi \right\rangle \right|
\leq \|\cF[f]\|_{H^{\gamma_2,\gamma_1}_{2,out}(\bR^d)} \left\|  \varphi \right\|_{L_{2}(\bR^d)},
\end{align*}
which implies
\begin{align*}
\|f\|_{H^{\gamma_1,\gamma_2}_{2,in}(\bR^d)} \leq \|\cF[f]\|_{H^{\gamma_2,\gamma_1}_{2,out}(\bR^d)}.
\end{align*}
The proposition is proved.
\end{proof}
Likewise, in the case where the integrability exponent $p$ is equal to 2, the Fourier transform $\mathcal{F}$ operates from an outer weighted space to an inner weighted space as an isometry.
\begin{prop}
Let  $\gamma_1, \gamma_2 \in \bR$.
Then
\begin{align*}
f \in H_{2,out}^{\gamma_1,\gamma_2}(\bR^d) ~\text{if and only if}~\cF[f] \in H_{2,in}^{\gamma_2,\gamma_1}(\bR^d).
\end{align*}
Additionally,  
\begin{align*}
\left\|f \right\|_{H_{2,out}^{\gamma_1,\gamma_2}(\bR^d)}
=\left\|\cF[f] \right\|_{H_{2,in}^{\gamma_2,\gamma_1}(\bR^d)}.
\end{align*}
\end{prop}
\begin{proof}
It is an easy consequence of Proposition \ref{l2 isometry} by taking $\cF[f]$ instead of $f$ with the help of  Proposition \ref{inverse identity} and the Fourier inversion theorem.
\end{proof}
Remarkable connections endure between the inner and outer spaces through the Fourier transform even if the integrability exponent $p$ surpasses 2.
\begin{prop}
								\label{p large embedding}
Let $p \in (2,\infty]$, $\gamma_1,\gamma_2 \in \bR$, and $\delta \in \left(\frac{d(p-2)}{2(p-1)},\infty \right)$.
\begin{enumerate}[(i)]
\item If $f \in H_{p,out}^{\gamma_1,\gamma_2}(\bR^d)$, then
\begin{align*}
\cF[f] \in H_{2,in}^{\gamma_2-\delta,\gamma_1}(\bR^d).
\end{align*}
Additionally,
\begin{align*}
\|\cF[f]\|_{H_{2,in}^{\gamma_2-\delta,\gamma_1}(\bR^d)}
\leq \left\| \left(1+|\cdot|^2\right)^{-\delta/2}  \right\|_{L_{2/(2-p')}(\bR^d)} \|f\|_{H_{p,out}^{\gamma_1,\gamma_2}(\bR^d)}.
\end{align*}

\item If $\cF[f] \in H_{p,out}^{\gamma_1,\gamma_2}(\bR^d)$, then
\begin{align*}
f \in H_{2,in}^{\gamma_2-\delta,\gamma_1}(\bR^d).
\end{align*}
Additionally,
\begin{align*}
\|f\|_{H_{2,in}^{\gamma_2-\delta,\gamma_1}(\bR^d)}
\leq \left\| \left(1+|\cdot|^2\right)^{-\delta/2}  \right\|_{L_{2/(2-p')}(\bR^d)} \|\cF[f]\|_{H_{p,out}^{\gamma_1,\gamma_2}(\bR^d)}.
\end{align*}

\end{enumerate}
\end{prop}
\begin{proof}
It is sufficient to prove (i) since (ii) easily comes from (i) due to Proposition \ref{inverse identity} and the Fourier inversion theorem.
Let $f \in H^{\gamma_1,\gamma_2}_{p,out}(\bR^d)$. Then by H\"older's inequality,
\begin{align*}
\left| \langle \cF[f] , \varphi \rangle \right|
&=\left| \left\langle f , \cF^{-1}[\varphi] \right\rangle \right| \\
&=\left| \left\langle  \left( 1+ |\cdot|^2 \right)^{\gamma_2/2} (I-\Delta)^{\gamma_1/2}f ,  \left( 1+ |\cdot|^2 \right)^{-\gamma_2/2}(I-\Delta)^{-\gamma_1/2}\cF^{-1}[\varphi] \right\rangle \right| \\
&= \left| \int_{\bR^d}  \left( 1+ |x|^2 \right)^{\gamma_2/2} (I-\Delta)^{\gamma_1/2}f(x)  \left( 1+ |x|^2 \right)^{-\gamma_2/2}(I-\Delta)^{-\gamma_1/2}\cF^{-1}[\varphi](x) \mathrm{d}x \right|\\
&\leq \|f\|_{H^{\gamma_1,\gamma_2}_{p,out}(\bR^d)} \left\| \left( 1+ |x|^2 \right)^{-\gamma_2/2}(I-\Delta)^{-\gamma_1/2}\cF^{-1}[\varphi](x) \right\|_{L_{p'}(\bR^d, \mathrm{d}x)}.
\end{align*}
Recalling $p' \in [1,2)$, we apply H\"older's inequality again and obtain
\begin{align*}
&\left\| \left( 1+ |x|^2 \right)^{-\gamma_2/2}(I-\Delta)^{-\gamma_1/2}\cF^{-1}[\varphi](x) \right\|_{L_{p'}(\bR^d, \mathrm{d}x)} \\
&\leq \| (1+|\cdot|^2)^{-\delta/2} \|_{L_{2/(2-p')}(\bR^d)}   
\left\|(1+|\cdot|^2)^{(\delta-\gamma_2)/2}(I-\Delta)^{-\gamma_1/2}\cF^{-1}[\varphi] \right\|_{L_{2}(\bR^d)},
\end{align*}
where $2/0:= \infty$.
Since $\delta > \frac{d(p-2)}{2(p-1)}$, we have
\begin{align*}
\left\| \left(1+|\cdot|^2\right)^{-\delta/2}  \right\|_{L_{2/(2-p')}(\bR^d)}  < \infty.
\end{align*}
Thus applying the Plancherel theorem, we have
\begin{align*}
\left|\left\langle \cF[f] , \varphi \right\rangle \right|
\leq \left\| \left(1+|\cdot|^2\right)^{-\delta/2}  \right\|_{L_{2/(2-p')}(\bR^d)}
\|f\|_{H^{\gamma_1,\gamma_2}_{p,out}(\bR^d)} \left\|\cF\left[(1+|\cdot|^2)^{(\delta-\gamma_2)/2}(I-\Delta)^{-\gamma_1/2}\cF^{-1}[\varphi] \right]\right\|_{L_{2}(\bR^d)}.
\end{align*}
Finally, considering 
\begin{align*}
&\cF\left[(I-\Delta)^{\gamma_1/2} \left(1+|\cdot|^2\right)^{-(\delta-\gamma_2)/2}  \cF^{-1}[\varphi] \right]\\
&= \left( 1+ |\cdot|^2 \right)^{\gamma_1/2} \left(I - \Delta \right)^{-(\delta-\gamma_2)/2}\varphi
\end{align*}
instead of $\varphi$, we have
\begin{align*}
&\left|\left\langle   \left(I - \Delta \right)^{(\gamma_2 - \delta)/2}  \left( 1+ |\cdot|^2 \right)^{\gamma_1/2} \cF[f] , \varphi \right\rangle \right|\\
&=\left|\left\langle \cF[f] , \left( 1+ |\cdot|^2 \right)^{\gamma_1/2} \left(I - \Delta \right)^{-(\delta-\gamma_2)/2}\varphi \right\rangle\right| \\
&=\left|\left\langle \cF[f] , \cF\left[(I-\Delta)^{\gamma_1/2} \left(1+|\cdot|^2\right)^{-(\delta-\gamma_2)/2}  \cF^{-1}[\varphi] \right] \right\rangle \right|\\
&\leq
\left\| \left(1+|\cdot|^2\right)^{-\delta/2}  \right\|_{L_{2/(2-p')}(\bR^d)}\|f\|_{H^{\gamma_1,\gamma_2}_{p,out}(\bR^d)} \left\|\varphi\right\|_{L_{2}(\bR^d)}.
\end{align*}
Finally, taking the supremum with respect to $\varphi \in \cS(\bR^d)$ so that $\|\varphi\|_{L_2(\bR^d)}$, we obtain
\begin{align*}
\|\cF[f]\|_{H_{2,in}^{\gamma_2-\delta,\gamma_1}} 
\leq \left\| \left(1+|\cdot|^2\right)^{-\delta/2}  \right\|_{L_{2/(2-p')}(\bR^d)} \|f\|_{H^{\gamma_1,\gamma_2}_{p,out}(\bR^d)}.
\end{align*}
The proposition is proved.
\end{proof}
We are now ready to establish spaces for handling data within our evolutionary framework, utilizing our newly introduced weighted Bessel potential spaces as a foundation.

We slightly abuse the notation in the following definition as follows :
for any $\cS'(\bR^d)$-valued (Borel) measurable function $u$ on $(0,T)$, we put
$u(t,\cdot) = u(t)$, \textit{i.e.} 
\begin{align*}
\langle u(t,\cdot) , \varphi \rangle := \langle u(t), \varphi \rangle \quad \forall \varphi \in \cS(\bR^d).
\end{align*}
This notation is useful especially when $u(t,\cdot)$ has a realization for each $t \in (0,T)$ by putting
$u(t,\xi) = u(t)(\xi)$.

\begin{defn}[Weighted Bessel potential-valued spaces]
\begin{enumerate}[(i)]

Let $p,q \in [1,\infty]$, $\gamma_1, \gamma_2 \in \bR$, and $w(t)$ be a non-negative measurable function on $(0,T)$. 
\item 

We define $\fH^{\gamma_1,\gamma_2}_{p,q,in}\left( (0,T) \times \bR^d, w^p(t)\mathrm{d}t \right)$ 
as the class of all $\cS'(\bR^d)$-valued (Borel) measurable functions $u$  on $(0,T)$ such that
\begin{align*}
\|u\|_{\fH^{\gamma_1,\gamma_2}_{p,q,in}\left( (0,T) \times \bR^d, w^p(t)\mathrm{d}t \right)}
&:=\|u\|_{L_{p,q}\left( (0,T) ,  w^p(t)\mathrm{d}t ;  H^{\gamma_1,\gamma_2}_{q,in}(\bR^d)\right)} \\
&:=\int_0^T \|u(t,\cdot)\|^p_{ H^{\gamma_1,\gamma_2}_{q,in}(\bR^d)} w^p(t) \mathrm{d}t \\
&=\int_0^T \left\|(I-\Delta)^{\gamma_1/2}\left(  (1+|\cdot|^2)^{\gamma_2/2}u(t,\cdot) \right) \right\|^p_{ L_q} w^p(t) \mathrm{d}t < \infty.
\end{align*}

\item 

We define $\fH^{\gamma_1,\gamma_2}_{p,q,out}\left( (0,T) \times \bR^d, w^p(t)\mathrm{d}t \right)$ 
as the class of all $\cS'(\bR^d)$-valued (Borel) measurable functions $u$  on $(0,T)$ such that
\begin{align*}
\|u\|_{\fH^{\gamma_1,\gamma_2}_{p,q,out}\left( (0,T) \times \bR^d, w^p(t)\mathrm{d}t\right)}
&:=
\|u\|_{L_{p,q}\left( (0,T) ,  w^p(t)\mathrm{d}t ;  H^{\gamma_1,\gamma_2}_{q,out}(\bR^d)\right)} \\
&:=\int_0^T \|u(t,\cdot)\|^p_{ H^{\gamma_1,\gamma_2}_{q,out}(\bR^d)} w^p(t) \mathrm{d}t \\
&=\int_0^T \left\|(1+|\cdot|^2)^{\gamma_2/2} (I-\Delta)^{\gamma_1/2}\left(  u(t,\cdot) \right) \right\|^p_{ L_q} w^p(t) \mathrm{d}t < \infty.
\end{align*}

\item 

We define $\fH^{\gamma_1,\gamma_2}_{p,q,in,t\text{-}loc}\left( (0,T) \times \bR^d, w^p(t)\mathrm{d}t \right)$ 
as the class of all $\cS'(\bR^d)$-valued (Borel) measurable functions $u$ on $(0,T)$ such that
\begin{align*}
\|u\|_{\fH^{\gamma_1,\gamma_2}_{p,q,in}\left( (0,t_1) \times \bR^d, w^p(t)\mathrm{d}t \right)}
< \infty
\end{align*}
for all $0<t_1<T$.

\item 

We define $\fH^{\gamma_1,\gamma_2}_{p,q,out,t\text{-}loc}\left( (0,T) \times \bR^d, w^p(t)\mathrm{d}t \right)$ 
as the class of all $\cS'(\bR^d)$-valued (Borel) measurable functions $u$  on $(0,T)$ such that
\begin{align*}
\|u\|_{\fH^{\gamma_1,\gamma_2}_{p,q,out}\left( (0,t_1) \times \bR^d, w^p(t)\mathrm{d}t \right)}
< \infty.
\end{align*}
for all $0<t_1<T$.

\item We define the classes of $\cS'(\bR^d)$-valued (Borel) measurable functions $u$ on $(0,T)$ whose Fourier transform with respect to the space variable is in weighted Bessel potential spaces.
We write 
$$
u \in \cF^{-1}\fH^{\gamma_1,\gamma_2}_{p,q,in}\left( (0,T) \times \bR^d, w^p(t)\mathrm{d}t \right),
$$
$$
u \in \cF^{-1}\fH^{\gamma_1,\gamma_2}_{p,q,out}\left( (0,T) \times \bR^d, w^p(t)\mathrm{d}t \right),
$$
$$
u \in \cF^{-1}\fH^{\gamma_1,\gamma_2}_{p,q,in,t\text{-}loc}\left( (0,T) \times \bR^d, w^p(t)\mathrm{d}t \right),
$$
and
$$
u \in \cF^{-1}\fH^{\gamma_1,\gamma_2}_{p,q,out,t\text{-}loc}\left( (0,T) \times \bR^d, w^p(t)\mathrm{d}t \right),
$$
 if
$$
\cF[u(t,\cdot)] \in \fH^{\gamma_1,\gamma_2}_{p,q,in}\left( (0,T) \times \bR^d, w^p(t)\mathrm{d}t \right),
$$
$$
\cF[u(t,\cdot)] \in \fH^{\gamma_1,\gamma_2}_{p,q,out}\left( (0,T) \times \bR^d, w^p(t)\mathrm{d}t \right),
$$ 
$$
\cF[u(t,\cdot)] \in \fH^{\gamma_1,\gamma_2}_{p,q,in,t\text{-}loc}\left( (0,T) \times \bR^d, w^p(t)\mathrm{d}t \right),
$$ 
and
$$
\cF[u(t,\cdot)] \in \fH^{\gamma_1,\gamma_2}_{p,q,out,t\text{-}loc}\left( (0,T) \times \bR^d, w^p(t)\mathrm{d}t \right),
$$ 
respectively.
\end{enumerate}
Here we consider the corresponding essential supremum instead of the integration if $p=\infty$.
For instance,
\begin{align*}
\|u\|_{L_{\infty,q}\left( (0,T) ,  w^p(t)\mathrm{d}t ;  H^{\gamma_1,\gamma_2}_{q,in}(\bR^d)\right)}
&:= \inf\left\{ M \in \bR :  \left|\left\{ t \in (0,T) : w(t)\|u(t,\cdot)\|^p_{ H^{\gamma_1,\gamma_2}_{q,in}(\bR^d)} > M \right\} \right|  =0 \right\}.
\end{align*}
Moreover, we omit $in$ and $out$ if $\gamma_2=0$ for simpler notation and identification with the unweighted Bessel potential spaces.  
$w^p(t)\mathrm{d}t$ is also skipped if $w(t) \equiv 1$. 
\end{defn}
Particularly, inhomogeneous data represented by $f$ to \eqref{ab eqn} could belong to a weighted Bessel potential-valued space, as will be demonstrated in numerous corollaries of Theorem \ref{weak solution thm} in the subsequent sections.

\mysection{Uniqueness of a weak solution with general data}
\label{24.03.29.12.58}

In this section, we establish the uniqueness of a Fourier-space weak solution to \eqref{ab eqn} by introducing a representation of a solution obtained through the use of Sobolev's mollifier (approximations to the identity). 
It seems that the space $\mathcal{F}^{-1}\mathcal{D}'(\mathbb{R}^d)$ is too expansive to claim the uniqueness of a weak solution. Therefore, our initial step involves constraining the space $\mathcal{F}^{-1}\mathcal{D}'(\mathbb{R}^d)$ to a more manageable one where Sobolev's mollifiers can be effectively applied. 
Since distributions on $\mathbb{R}^d$ include all locally integrable functions on $\mathbb{R}^d$, it is obvious that
\begin{align*}
\cF^{-1}L_{1,\ell oc}(\bR^d) \subset \cF^{-1}\cD'(\bR^d)
\end{align*}
and
\begin{align*}
\cF^{-1}L_{1,1,t\text{-}loc,x\text{-}\ell oc}\left( (0,T) \times \bR^d  \right)
\subset L_{1,t\text{-}loc}\left( (0,T) ;  \cF^{-1}\cD'(\bR^d) \right).
\end{align*}
Our proof of uniqueness begins by demonstrating the following representation.
\begin{lem}[A representation of a solution]
							\label{fourier repre}
Let $u_0 \in \cF^{-1}L_{1,\ell oc}(\bR^d)$, $f \in \cF^{-1}L_{1,1,t\text{-}loc,x\text{-}\ell oc}\left( (0,T) \times \bR^d  \right)$, and $u$ be a Fourier-space weak solution to \eqref{ab eqn}.
Assume that
\begin{align*}
u \in \cF^{-1}L_{1,1,t\text{-}loc,x\text{-}\ell oc}\left( (0,T) \times \bR^d, \mathrm{d}t |\psi(t,\xi)|\mathrm{d}\xi  \right).
\end{align*}
Then
\begin{align}
										\label{20230617 02}
\cF[u(t,\cdot)](\xi)
= \cF[u_0](\xi)+\int_0^t \psi(s,\xi) \cF[u(s,\cdot)](\xi) \mathrm{d}s 
+\int_0^t \cF[f(s,\cdot)](\xi)  \mathrm{d}s\quad (a.e.)~(t,\xi) \in (0,T) \times \bR^d.
\end{align}
\end{lem}
\begin{proof}
Let $\varphi \in \cF^{-1}\cD(\bR^d)$.
Since $u_0 \in \cF^{-1}L_{1,\ell oc}(\bR^d)$, $f \in \cF^{-1}L_{1,1,t\text{-}loc,x\text{-}\ell oc}\left( (0,T) \times \bR^d  \right)$, and $u$ is a Fourier-space weak solution to \eqref{ab eqn} in the space $\cF^{-1}L_{1,1,t\text{-}loc,x\text{-}\ell oc}\left( (0,T) \times \bR^d, \mathrm{d}t |\psi(t,\xi)|\mathrm{d}\xi  \right)$, we have
\begin{align*}
\left\langle \cF[u(t,\cdot)], \cF[\varphi](\cdot)  \right\rangle
&= \left( \cF[u_0], \cF[\varphi]  \right)_{L_2(\bR^d)} + \int_0^t \left( \psi(s,\cdot)\cF[u(s,\cdot)] , \cF[\varphi](\cdot) \right)_{L_2(\bR^d)} \mathrm{d}s \\
&\quad + \int_0^t \left( \cF[f(s,\cdot)],\cF[\varphi](\cdot) \right)_{L_2(\bR^d)} \mathrm{d}s
\quad (a.e.)~t\in (0,T).
\end{align*}
Here we used realizations mentioned in Remark \ref{solution realization}.
The above equality implies that for almost every $t \in (0,T)$, the distribution $\cF[u(t,\cdot)]$ has a realization on $\bR^d$.
More precisely, by using the separability of $\cD(\bR^d)$ and Fubini's theorem, we have
\begin{align}
										\notag
\left( \cF[u(t,\cdot)], \cF[\varphi]  \right)_{L_2(\bR^d)}
&= \left( \cF[u_0], \cF[\varphi]  \right)_{L_2(\bR^d)} + \int_0^t \left( \psi(s,\cdot)\cF[u(s,\cdot)] , \cF[\varphi](\cdot) \right)_{L_2(\bR^d)} \mathrm{d}s \\
										\label{20230616 01}
&\quad + \int_0^t \left( \cF[f(s,\cdot)],\cF[\varphi] \right)_{L_2(\bR^d)} \mathrm{d}s
\quad (a.e.)~t\in (0,T).
\end{align}
We use Sobolev mollifiers with additional specific properties. 
Let $\chi$ be a function in $\cS(\bR^d)$ 
so that $\cF[\chi]$ is non-negative and symmetric, \emph{i.e.}
$\cF[\chi](\xi) =\cF[\chi](-\xi) \geq 0$ for all $\xi \in \bR^d$. 
Additionally, assume that $\cF[\chi]$ has a compact support  and 
$$
\int_{\bR^d} \chi(x)\mathrm{d}x=(2\pi)^{d/2}.
$$
For $\varepsilon \in (0,1)$, denote
\begin{align*}
\chi^\varepsilon(x) := \frac{1}{\varepsilon^d}\chi\left( \frac{x}{\varepsilon}\right).
\end{align*}
Fix $x \in \bR^d$ and put $\chi^\varepsilon(x-\cdot)$ in \eqref{20230616 01} instead of $\varphi$.
Then
\begin{align*}
&\int_{\bR^d} \cF[u(t,\cdot)](\xi) \overline{\cF[\chi^\varepsilon(x-\cdot)](\xi)}\mathrm{d}\xi \\
&=  \int_{\bR^d} \cF[u_0](\xi) \overline{\cF[\chi^\varepsilon(x-\cdot)](\xi)}\mathrm{d}\xi
+\int_0^t \int_{\bR^d} \psi(s,\xi)\cF[u(s,\cdot)](\xi)  \overline{\cF[\chi^\varepsilon(x-\cdot)](\xi)} \mathrm{d}\xi \mathrm{d}s \\
&\quad +\int_0^t \int_{\bR^d} \cF[f(s,\cdot)](\xi)  \overline{\cF[\chi^\varepsilon(x-\cdot)](\xi)} \mathrm{d}\xi \mathrm{d}s 
\quad (a.e.)~t \in (0,T).
\end{align*}
Thus recalling properties of $\cF[\chi]$ and the Fourier transform, we have
\begin{align*}
&\int_{\bR^d} \mathrm{e}^{i  x \cdot \xi} \cF[u(t,\cdot)](\xi) \cF[\chi^\varepsilon](\xi) \mathrm{d}\xi \\
&=\int_{\bR^d} \cF[u(t,\cdot)](\xi)  \overline{\mathrm{e}^{-i  x \cdot \xi}\cF[\chi^\varepsilon](-\xi)} \mathrm{d}\xi \\
&=\int_{\bR^d} \cF[u(t,\cdot)](\xi) \overline{\cF[\chi^\varepsilon(x-\cdot)](\xi)} \mathrm{d}\xi \\
&= \int_{\bR^d} \cF[u_0](\xi) \overline{\cF[\chi^\varepsilon(x-\cdot)](\xi)}\mathrm{d}\xi
+\int_0^t \int_{\bR^d} \psi(s,\xi)\cF[u(s,\cdot)](\xi)  \overline{ \cF[\chi^\varepsilon(x-\cdot)](\xi)} \mathrm{d}\xi\mathrm{d}s \\
&\quad +\int_0^t \int_{\bR^d} \cF[f(s,\cdot)](\xi)  \overline{\cF[\chi^\varepsilon(x-\cdot)](\xi)} \mathrm{d}\xi \mathrm{d}s \\
&= \int_{\bR^d} \mathrm{e}^{i  x \cdot \xi} \cF[u_0](\xi) \cF[\chi^\varepsilon](\xi) \mathrm{d}\xi
+\int_{\bR^d} \mathrm{e}^{ix \cdot \xi} \left(\int_0^t \psi(s,\xi) \cF[u(s,\cdot)](\xi)   \cF[\chi^\varepsilon](\xi) \mathrm{d}s\right)  \mathrm{d}\xi \\
&\quad + \int_{\bR^d} \mathrm{e}^{ix \cdot \xi} \left(\int_0^t \cF[f(s,\cdot)](\xi)   \cF[\chi^\varepsilon](\xi) \mathrm{d}s\right)  \mathrm{d}\xi \quad (a.e.)~ t \in (0,T).
\end{align*}
Here Fubini's theorem could be applicable to all terms above since
$\cF[\chi^\varepsilon] \in \cD(\bR^d)$, $u_0 \in \cF^{-1}L_{1,\ell oc}(\bR^d)$, $f \in \cF^{-1}L_{1,1,t\text{-}loc,x\text{-}\ell oc}\left( (0,T) \times \bR^d  \right)$, and
\begin{align*}
u \in \cF^{-1}L_{1,1,t\text{-}loc,x\text{-}\ell oc}\left( (0,T) \times \bR^d, \mathrm{d}t |\psi(t,\xi)|\mathrm{d}\xi  \right).
\end{align*}
Due to the Fourier inversion theorem,
\begin{align}
										\label{20230617 01}
 \cF[u(t,\cdot)](\xi) \cF[\chi^\varepsilon](\xi)
=\cF[u_0](\xi) \cF[\chi^\varepsilon](\xi)
+\int_0^t \psi(s,\xi) \cF[u(s,\cdot)](\xi)   \cF[\chi^\varepsilon](\xi) \mathrm{d}s
+\int_0^t \cF[f(s,\cdot)](\xi)   \cF[\chi^\varepsilon](\xi) \mathrm{d}s
\end{align}
for almost every $(t,\xi) \in (0,T) \times \bR^d$. 
Observe that
\begin{align*}
\cF[\chi^\varepsilon](\xi)=\cF[\chi](\varepsilon\xi)
\end{align*}
and
\begin{align*}
\cF[\chi](0)= (2\pi)^{-d/2} \int_{\bR^d} \chi(y) \mathrm{d}y = 1.
\end{align*}
Finally, taking $\varepsilon \downarrow 0$ in \eqref{20230617 01}, we have
\eqref{20230617 02}.
\end{proof}

\begin{rem}
It is easy to show that \eqref{20230617 02} with the Fubini theorem implies
\begin{align*}
u(t,x)
= u_0(x)+\int_0^t \psi(s,-i\nabla) u(s,\cdot)(x)\mathrm{d}s 
+\int_0^t f(s,x) \mathrm{d}s\quad (a.e.)~t \in (0,T).
\end{align*}
This equation should be interpreted as elements within $\cF^{-1}\cD(\bR^d)$, given that the term $\psi(s,-i\nabla) u(s,\cdot)(x)$ is understood as the inverse Fourier transform of a locally integrable function $\psi(s,\xi) \cF u(s,\cdot)$ with respect to $\xi$, as discussed in Remark \ref{solution realization}.
\end{rem}

We can now demonstrate the uniqueness of the solution using the provided representation and taking advantage of the linearity of \eqref{ab eqn}. 
It is obvious that \eqref{20230617 02} implies 
\begin{align}
								\label{20240227 01}
u \in \cF^{-1}L_{\infty,1,t\text{-}loc, x\text{-}\ell oc}\left( (0,T) \times \bR^d \right).
\end{align}
This property heavily depends on conditions of data $u_0$ and $f$ since they are important used in the representation formula.
However, these restrictions on the initial data $u_0$ and the inhomogeneous data $f$ can be eliminated due to the linearity if we merely have interest in the uniqueness. 
In particular, we can claim the uniqueness of a solution even for $\cF^{-1}\cD'(\bR^d)$-valued data.

\begin{thm}[Uniquness of a weak solution]
									\label{unique weak sol}
Let  $u_0 \in \cF^{-1}\cD'(\bR^d)$ and
$f \in L_{1,t\text{-}loc}\left( (0,T) ;  \cF^{-1}\cD'(\bR^d) \right)$.
Then a Fourier-space weak solution to \eqref{ab eqn} is unique in 
\begin{align*}
\cF^{-1}L_{1,1,t\text{-}loc,x\text{-}\ell oc}\left( (0,T) \times \bR^d, \mathrm{d}t |\psi(t,\xi)|\mathrm{d}\xi  \right).
\end{align*}
\end{thm}
\begin{proof}
Let $u_1$ and $u_2$ be Fourier-space weak solutions to \eqref{ab eqn} so that
\begin{align*}
u_1,u_2 \in \cF^{-1}L_{1,1,t\text{-}loc,x\text{-}\ell oc}\left( (0,T) \times \bR^d, \mathrm{d}t |\psi(t,\xi)|\mathrm{d}\xi  \right).
\end{align*}
Put
$$
u = u_1 - u_2.
$$
Then by Lemma \ref{fourier repre}, we have
\begin{align*}
\cF[u(t,\cdot)](\xi)
= \int_0^t \psi(s,\xi) \cF[u(s,\cdot)](\xi) \mathrm{d}s \quad (a.e.)~(t,\xi) \in (0,T) \times \bR^d.
\end{align*}
Due to Gr\"onwall's inequality, we show that
\begin{align*}
\cF[u(t,\cdot)](\xi) = 0 \quad (a.e.)~ (t,\xi) \in (0,T) \times \bR^d.
\end{align*}
Thus 
\begin{align}
							\label{20240201 10}
\cF[u_1(t,\cdot)](\xi)= \cF[u_2(t,\cdot)](\xi) \quad (a.e.) ~(t,\xi) \in (0,T) \times \bR^d,
\end{align}
which implies that $u_1=u_2$ as an element in 
\begin{align*}
\cF^{-1}L_{1,1,t\text{-}loc,x\text{-}\ell oc}\left( (0,T) \times \bR^d, \mathrm{d}t |\psi(t,\xi)|\mathrm{d}\xi  \right).
\end{align*}
\end{proof}

\begin{rem}
The uniqueness stated in Theorem \ref{unique weak sol} may appear uncertain when considering a solution $u$ for \eqref{ab eqn} as a function valued in $\cF^{-1}\cD'(\bR^d)$, defined over the interval $(0,T)$ if $\psi(t,\xi)$ vanishes at certain points as explained in Remark \ref{cutoff distribution}. 
Additionally, we do not know if $u_1$ and $u_2$ satisfy \eqref{20240227 01} in the proof of Theorem \ref{unique weak sol} since the conditions on $u_0$ and $f$ are weakened.
Nevertheless, we have shown that the realizations $\mathcal{F}[u_1(t,\cdot)](\xi)$ and $\mathcal{F}[u_2(t,\cdot)](\xi)$ coincide for almost every $(t,\xi) \in (0,T) \times \mathbb{R}^d$, considering all solutions $u_1$ and $u_2$ to the equation \eqref{ab eqn} in
\begin{align*}
\cF^{-1}L_{1,1,t\text{-}loc,x\text{-}\ell oc}\left( (0,T) \times \bR^d, \mathrm{d}t |\psi(t,\xi)|\mathrm{d}\xi  \right),
\end{align*}
which is described in \eqref{20240201 10}. 
This observation unequivocally implies that a solution $u$ to \eqref{ab eqn} is also unique when viewed as a $\cF^{-1}\cD'(\bR^d)$-valued function defined $(a.e.)$ on $(0,T$).

Moreover, this issue does not cause any problem in our main theorem, Theorem \ref{weak solution thm}, because we establish both the existence and uniqueness of a solution $u$ within the intersection of two function spaces $\cF^{-1}L_{\infty,1,t\text{-}loc, x\text{-}\ell oc}\left( (0,T) \times \bR^d \right)$ and
$\cF^{-1}L_{1,1,t\text{-}loc, x\text{-}\ell oc}\left( (0,T) \times \bR^d, \mathrm{d}t |\psi(t,\xi)|\mathrm{d}\xi \right)$.
More precisely, 
$$
u \in \cF^{-1}L_{1,1,t\text{-}loc, x\text{-}\ell oc}\left( (0,T) \times \bR^d, \mathrm{d}t |\psi(t,\xi)|\mathrm{d}\xi \right)
$$
implies 
$$
u \in \cF^{-1}L_{\infty,1,t\text{-}loc, x\text{-}\ell oc}\left( (0,T) \times \bR^d \right)
$$ 
since $u_0 \in \cF^{-1}L_{1,\ell oc}(\bR^d)$ and $f \in \cF^{-1}L_{1,1,t\text{-}loc,x\text{-}\ell oc}\left((0,T) \times \bR^d \right)$
as mentioned \eqref{20240227 01}.
In other words, the ambiguity of the uniqueness of $u$ as a $\cF^{-1}\cD'(\bR^d)$-valued function on $(0,T)$ could be resolved completely since $u$ is also an element of $\cF^{-1}L_{\infty,1,t\text{-}loc, x\text{-}\ell oc}\left( (0,T) \times \bR^d \right)$ in Theorem \ref{weak solution thm}.
\end{rem}

Next, we characterize conditions on symbols $\psi(t,\xi)$ to preserve the uniqueness of a solution whose Fourier transform is in a weighted $L_{p,q}$-space.
This is readily accomplished due to H\"older's inequality.

\begin{corollary}
								\label{unique corollary}
Let $p,q \in [1,\infty]$, $\gamma_1, \gamma_2 \in \bR$, $u_0 \in \cF^{-1}\cD'(\bR^d)$,
$f \in L_{1,t\text{-}loc}\left( (0,T) ;  \cF^{-1}\cD'(\bR^d) \right)$, and $W(t,x)$ be a positive measurable function on $(0,T) \times \bR^d$.
Assume that
\begin{align*}
\psi \in  L_{p',q',t\text{-}loc,x\text{-}\ell oc}\left( (0,T) \times \bR^d, \mathrm{d}t (1/W(t,\xi))^{q'}\mathrm{d}\xi  \right),
\end{align*}
where $p'$ and $q'$ are H\"older conjugates of $p$ and $q$, respectively.
Then a Fourier-space weak solution to \eqref{ab eqn} is unique in 
\begin{align*}
\cF^{-1}L_{p,q,t\text{-}loc,x\text{-}\ell oc}\left( (0,T) \times \bR^d, \mathrm{d}t (W(t,\xi))^{q}\mathrm{d}\xi  \right).
\end{align*}
\end{corollary}
\begin{proof}
Let 
\begin{align*}
u \in \cF^{-1}L_{p,q,t\text{-}loc,x\text{-}\ell oc}\left( (0,T) \times \bR^d, \mathrm{d}t (W(t,\xi))^q\mathrm{d}\xi  \right)
\end{align*}
be a Fourier-space weak solution to \eqref{ab eqn}.
It is sufficient to show that
\begin{align*}
u \in \cF^{-1}L_{1,1,t\text{-}loc,x\text{-}\ell oc}\left( (0,T) \times \bR^d, \mathrm{d}t |\psi(t,\xi)|\mathrm{d}\xi  \right)
\end{align*}
due to Theorem \ref{unique weak sol}.
However, it is an easy consequence of H\"older's inequality.
Indeed, for any $ t \in (0,T)$ and $R \in (0,\infty)$, we have
\begin{align*}
\int_0^t \int_{B_R} |\psi(s,\xi)| |\cF[u](\xi)| \mathrm{d}\xi \mathrm{d}s
&=\int_0^t \int_{B_R} |\psi(s,\xi)| \frac{1}{W(t,\xi)} |\cF[u](\xi)| W(t,\xi) \mathrm{d}\xi \mathrm{d}s \\
&\leq \|\psi\|_{L_{p',q'}\left( (0,t) \times B_R, \mathrm{d}t (1/W(t,\xi))^{q'}\mathrm{d}\xi  \right)}
\|\cF[u]\|_{L_{p,q}\left( (0,t) \times B_R, \mathrm{d}t (W(t,\xi))^q\mathrm{d}\xi  \right)} < \infty.
\end{align*}
\end{proof}

Take note that in the aforementioned corollary, the weight denoted as $W$ is assumed to be positive.
This requirement is maintained in order to ensure the existence of a solution in our theory, as seen in Theorem \ref{main second}.
However, there is a possibility to relax this positive constraint on weights by considering a more intricate limiting process, which seems to be an interesting topic for further investigations (\textit{cf}. \cite{KID KHK 2018,KID KHK 2023}).

\mysection{Proof  and Corollaries  of Theorem \ref{weak solution thm} }
											\label{pf main thm 1}

We are now ready to initiate the proof of our main theorem. 
It is important to recall that the uniqueness has already been established due to Theorem \ref{unique weak sol}. 
Thus our current goal becomes to show the existence of a Fourier-space weak solution $u$ to \eqref{ab eqn}.
A credible candidate for this solution $u$ is presented in \eqref{solution candidate}. 
Therefore, our task is to verify that the function $u$ as defined in \eqref{solution candidate} indeed satisfies the criteria for being our weak solution in the forthcoming proof.

\begin{proof}[Proof of Theorem \ref{weak solution thm}]
First, observe that the mapping 
\begin{align*}
(t,\xi) \mapsto   \exp\left(\int_0^t\psi(r,\xi)\mathrm{d}r \right) \cF[u_0](\xi)
+\int_0^t  \exp\left(\int_s^t\psi(r,\xi)\mathrm{d}r \right) \cF[f(s,\cdot)](\xi)\mathrm{d}s
\end{align*}
is in $L_{\infty,1,t\text{-}loc, x\text{-}\ell oc}\left( (0,T) \times \bR^d \right)$ due to \eqref{20230624 10}.
Here we used the continuity of the functions 
$$
t \mapsto \int_{B_R}\exp\left(\int_0^t\Re[\psi(r,\xi)]\mathrm{d}r \right) \cF[u_0](\xi) \mathrm{d}\xi
$$ 
and
\begin{align*}
t \mapsto  \int_{B_R} \int_0^t  \exp\left(\int_s^t\Re[\psi(r,\xi)]\mathrm{d}r \right) \left|\cF[f(s,\cdot)](\xi) \right| \mathrm{d}s \mathrm{d}\xi,
\end{align*}
 \textit{i.e.} for any $t \in (0,T)$, there exist $t_0, t_1 \in [0,t]$ so that 
\begin{align*}
\int_{B_R}\exp\left(\int_0^{t_0}\Re[\psi(r,\xi)]\mathrm{d}r \right) \cF[u_0](\xi) \mathrm{d}\xi= \sup_{\rho \in [0,t]} \int_{B_R}\exp\left(\int_0^\rho \Re[\psi(r,\xi)]\mathrm{d}r \right) \cF[u_0](\xi) \mathrm{d}\xi
\end{align*}
and
\begin{align*}
&\int_{B_R} \int_0^{t_1}  \exp\left(\int_s^{t_1}\Re[\psi(r,\xi)]\mathrm{d}r \right) \left|\cF[f(s,\cdot)](\xi) \right| \mathrm{d}s \mathrm{d}\xi \\
&=\sup_{ \rho \in [0,t]} \int_{B_R} \int_0^\rho  \exp\left(\int_s^\rho\Re[\psi(r,\xi)]\mathrm{d}r \right) \left|\cF[f(s,\cdot)](\xi) \right| \mathrm{d}s \mathrm{d}\xi. 
\end{align*}
Similarly, the mapping 
\begin{align*}
(t,\xi) \mapsto   \exp\left(\int_0^t\psi(r,\xi)\mathrm{d}r \right) \cF[u_0](\xi)
+\int_0^t  \exp\left(\int_s^t\psi(r,\xi)\mathrm{d}r \right) \cF[f(s,\xi)]\mathrm{d}s
\end{align*}
is in 
$$
L_{1,1,t\text{-}loc, x\text{-}\ell oc}\left( (0,T) \times \bR^d, \mathrm{d}t |\psi(t,\xi)|\mathrm{d}\xi \right)
$$ due to \eqref{20230624 11}.
Thus recalling that
\begin{align*}
u(t,x):=  
 \cF^{-1}\left[  \exp\left(\int_0^t\psi(r,\cdot)\mathrm{d}r \right) \cF[u_0] \right](x)
+\cF^{-1}\left[ \int_0^t  \exp\left(\int_s^t\psi(r,\cdot)\mathrm{d}r \right) \cF[f(s,\cdot)]\mathrm{d}s \right](x),
\end{align*}
we have
$$
u \in \cF^{-1}L_{\infty,1,t\text{-}loc, x\text{-}\ell oc}\left( (0,T) \times \bR^d \right) \cap \cF^{-1}L_{1,1,t\text{-}loc, x\text{-}\ell oc}\left( (0,T) \times \bR^d, \mathrm{d}t |\psi(t,\xi)|\mathrm{d}\xi \right).
$$ 
Thus it is sufficient to show that $u$ is a Fourier-space weak solution to $\eqref{ab eqn}$ due to the uniqueness theorem (Theorem \ref{unique weak sol})
as mentioned at the beginning of this section.
By the fundamental theorem of calculus, and the Fubini theorem,
\begin{align}
									\notag
&\cF[u(t,\cdot)](\xi) \\
									\notag
&= \exp\left(\int_0^t\psi(r,\xi)\mathrm{d}r \right) \cF[u_0](\xi) + \int_0^t  \exp\left(\int_s^t\psi(r,\xi)\mathrm{d}r \right) \cF[f(s,\cdot)](\xi)\mathrm{d}s \\
									\notag
&= \exp\left(\int_0^t\psi(r,\xi)\mathrm{d}r \right) \cF[u_0](\xi) + \int_0^t  \int_s^t \frac{\mathrm{d}}{\mathrm{d}\rho}\left[\exp\left(\int_s^\rho\psi(r,\xi)\mathrm{d}r \right)\right]  \mathrm{d}\rho \cF[f(s,\cdot)](\xi)\mathrm{d}s 
+\int_0^t  \cF[f(s,\cdot)](\xi)\mathrm{d}s \\
									\notag
&= \exp\left(\int_0^t\psi(r,\xi)\mathrm{d}r \right) \cF[u_0](\xi) + \int_0^t  \int_s^t \left[\psi(\rho,\xi)\exp\left(\int_s^\rho\psi(r,\xi)\mathrm{d}r \right)\right]  \mathrm{d}\rho \cF[f(s,\cdot)](\xi)\mathrm{d}s  \\
									\notag
&\quad +\int_0^t \cF[f(s,\cdot)](\xi)\mathrm{d}s \\
									\notag
&= \exp\left(\int_0^t\psi(r,\xi)\mathrm{d}r \right) \cF[u_0](\xi) + \int_0^t \psi(\rho,\xi) \int_0^\rho \exp\left(\int_s^\rho\psi(r,\xi)\mathrm{d}r \right)  \cF[f(s,\cdot)](\xi)\mathrm{d}s \mathrm{d}\rho
+\int_0^t \cF[f(s,\cdot)](\xi)\mathrm{d}s \\
									\notag
&=\exp\left(\int_0^t\psi(r,\xi)\mathrm{d}r \right) \cF[u_0](\xi) 
-\int_0^t \psi(\rho,\xi) \exp\left(\int_0^\rho\psi(r,\xi)\mathrm{d}r \right) \mathrm{d}\rho \cF[u_0](\xi) \\
									\notag
&\quad + \int_0^t \psi(\rho,\xi) \cF[u(\rho,\cdot)](\xi) \mathrm{d}\rho +\int_0^t \cF[f(s,\cdot)](\xi)\mathrm{d}s  \\
									\label{20230213 20}
&= \cF[u_0](\xi)  + \int_0^t \psi(\rho,\xi) \cF[u(\rho,\cdot)](\xi) \mathrm{d}\rho +\int_0^t \cF[f(s,\cdot)](\xi)\mathrm{d}s 
\end{align}
for almost every $(t,\xi) \in (0,T) \times \bR^d$. 
The condition \eqref{20230624 11} is used above to apply the Fubini theorem.
Finally, for any $\varphi \in \cF^{-1}\cD(\bR^d)$, taking the integration with $\overline{\cF[\varphi]}$ in \eqref{20230213 20} and applying the definitions of actions on the elements of the class $\cF^{-1}\cD'(\bR^d)$ with the Fubini theorem,  we have
\begin{align*}
\left\langle u(t,\cdot),\varphi \right\rangle 
= \langle u_0, \varphi \rangle +  \int_0^t \left(\psi(s,\cdot) \cF[u(s,\cdot)] , \cF[\varphi](\cdot) \right)_{L_2(\bR^d)} \mathrm{d}s 
+ \int_0^t \left\langle f(s,\cdot),\varphi\right\rangle \mathrm{d}s
\quad (a.e.)~t\in (0,T),
\end{align*}
which implies that $u$ is a Fourier-space weak solution \eqref{ab eqn} as discussed in Remark \ref{solution realization}.
\end{proof}
Now, let's examine data associated with general weights. We present a series of corollaries to demonstrate that our mathematical requirements for $\psi$ are generous enough to accommodate a wide range of weighted data. 
It is important to note that we are exclusively addressing positive weights.
\begin{corollary}
						\label{cor 202307014 01}
Let $q \in [1,\infty]$, $W_0(\xi)$  and $W_1(\xi)$ be  positive measurable functions on $\bR^d$, and $W_2(t,\xi)$ be a positive measurable function on $(0,T) \times \bR^d$  so that
both $W_0(\xi)$  and $W_1(\xi)$ are locally bounded below and $W_1(\xi)$ is a lower bound of $W_2(t,\xi)$, \textit{i.e.}
for each $t \in (0,T)$ and $R \in (0,\infty)$, there exist positive constants $\kappa_0(R)$ and $\kappa_1(R)$ so that
\begin{align}
							\label{20230808 01}
W_0(\xi) \geq \kappa_0(R) \quad \forall \xi \in B_R.
\end{align}
and
\begin{align}
							\label{20230808 01-2}
W_2(s,\xi) \geq  W_1(\xi) \geq \kappa_1(R) \quad \forall (s,\xi) \in (0,t) \times B_R.
\end{align}
Assume that
$u_0 \in \cF^{-1}L_{q,\ell oc}(\bR^d,  W_0^q(\xi) \mathrm{d}\xi)$ and $f \in \cF^{-1}L_{1,q,t\text{-}loc, x \text{-}\ell oc}\left(  (0,T) \times \bR^d, \mathrm{d}t W_2^q(t,\xi)\mathrm{d}\xi\right)$.
Additionally, suppose that the symbol $\psi(t,\xi)$ satisfies the following integrability conditions
 \begin{align}
								\notag
&\int_0^t 
 \left\|\frac{1+|\psi(\rho,\xi)|}{W_0(\xi)} \exp\left(\int_{0}^\rho\Re[\psi(r,\xi)] \right)\mathrm{d}r \right\|_{L_{q'}\left(  B_R , \mathrm{d} \xi \right) } 
 \mathrm{d} \rho \\
									\label{20230715 01}
&+\int_0^t 
 \left\|\frac{1+|\psi(\rho,\xi)|}{W_1(\xi)} \exp\left(\int_{0}^\rho|\Re[\psi(r,\xi)|] \right)\mathrm{d}r \right\|_{L_{q'}\left(  B_R , \mathrm{d} \xi \right) } 
 \mathrm{d} \rho 
  < \infty
 \end{align}
for all $t \in (0,T)$ and $R \in (0,\infty)$. Then there exists a unique Fourier-space weak solution to \eqref{ab eqn}  in 
$$
\cF^{-1}L_{\infty,1,t\text{-}loc, x\text{-}\ell oc}\left( (0,T) \times \bR^d \right) \cap \cF^{-1}L_{1,1,t\text{-}loc, x\text{-}\ell oc}\left( (0,T) \times \bR^d, \mathrm{d}t |\psi(t,\xi)|\mathrm{d}\xi \right).
$$
\end{corollary}
\begin{proof}
We use Theorem \ref{weak solution thm}.
First of all, it is easy to check that
$u_0 \in \cF^{-1}L_{1,\ell oc}(\bR^d)$ and 
$$
f \in \cF^{-1}L_{1,1,t\text{-}loc,x\text{-}\ell oc}\left((0,T) \times \bR^d \right)
$$
due to H\"older's inequality and the local lower bounds of $W_0$ and $W_2$ given in \eqref{20230808 01} and \eqref{20230808 01-2}.
Next we show that \eqref{20230715 01} implies \eqref{20230624 10} and \eqref{20230624 11}. Recall that $u_0 \in \cF^{-1}L_{q,\ell oc}(\bR^d,W_0^q(\xi)\mathrm{d}\xi)$ and
$$
f \in L_{1,q,t\text{-}loc, x \text{-}\ell oc}\left( (0,T) \times \bR^d, \mathrm{d}t W_2^q(t,\xi)\mathrm{d}\xi\right).
$$
Thus by H\"older's inequality and \eqref{20230715 01},
\begin{align*}
&\int_{B_R}\int_0^t (1+|\psi(\rho,\xi)|) \exp\left(\int_0^\rho \Re[\psi(r,\xi)]\mathrm{d}r \right) \cF[u_0](\xi)\mathrm{d}\rho \mathrm{d}\xi \\
&=\int_0^t \int_{B_R}\frac{1+|\psi(\rho,\xi)|}{W_0(\xi)} \exp\left(\int_0^\rho \Re[\psi(r,\xi)]\mathrm{d}r \right) \cF[u_0](\xi) W_0(\xi) \mathrm{d}\xi \mathrm{d}\rho \\
&\leq \int_0^t \left\|\frac{1+|\psi(\rho,\xi)|}{W_0(\xi)} \exp\left(\int_{0}^\rho\Re[\psi(r,\xi)]\mathrm{d}r  \right)\right\|_{L_{q'}\left( B_R , \mathrm{d} \xi \right) } \mathrm{d} \rho
\|\cF[u_0]\|_{L_{q} \left(B_R, \mathrm{d}s W_0^q(\xi) \mathrm{d}\xi \right)} < \infty
\end{align*}
for all $t \in (0,T)$ and $R \in (0,\infty)$.
Similarly,  by the Fubini theorem, the H\"older inequality, \eqref{20230808 01-2}, and \eqref{20230715 01},
\begin{align*}
&\int_{B_R}\int_0^t (1+|\psi(\rho,\xi)|) \int_0^\rho \exp\left(\int_s^\rho\Re[\psi(r,\xi)]\mathrm{d}r \right)  \left|\cF[f(s,\cdot)](\xi) \right| \mathrm{d}s \mathrm{d}\rho \mathrm{d}\xi \\
&=\int_0^t  \int_0^\rho  \int_{B_R}\frac{1+|\psi(\rho,\xi)|}{W_2(s,\xi)} \exp\left(\int_s^\rho\Re[\psi(r,\xi)]\mathrm{d}r \right)  \left|\cF[f(s,\cdot)](\xi) \right| W_2(s,\xi) \mathrm{d}\xi \mathrm{d}s \mathrm{d}\rho  \\
&\leq \int_0^t  \int_0^t \int_{B_R}\frac{1+|\psi(\rho,\xi)|}{W_1(\xi)} \exp\left(\int_0^\rho |\Re[\psi(r,\xi)]| \mathrm{d}r \right)    \left|\cF[f(s,\cdot)](\xi) W_2(s,\xi) \right|   \mathrm{d}\xi \mathrm{d}s \mathrm{d}\rho  \\
&\leq \left\|(1+|\psi(\rho,\xi)|) \exp\left(\int_0^\rho|\Re[\psi(r,\xi)]|\mathrm{d}r  \right)\right\|_{L_{1,q'}\left( (0,t) \times B_R , \mathrm{d}\rho (1/W_1(\xi))^{q'}\mathrm{d} \xi \right) } 
\|\cF[f(s,\cdot)(\xi)\|_{L_{1,q} \left( (0,t) \times B_R, \mathrm{d}s W_2^q(s,\xi) \mathrm{d}\xi \right)} \\
&< \infty
\end{align*}
for all $t \in (0,T)$ and $R \in (0,\infty)$. 
Therefor by Theorem \ref{weak solution thm},  there exists a unique Fourier-space weak solution $u$ to \eqref{ab eqn} in
$$
\cF^{-1}L_{\infty,1,t\text{-}loc, x\text{-}\ell oc}\left( (0,T) \times \bR^d \right) \cap \cF^{-1}L_{1,1,t\text{-}loc, x\text{-}\ell oc}\left( (0,T) \times \bR^d, \mathrm{d}t |\psi(t,\xi)|\mathrm{d}\xi \right).
$$
\end{proof}
\begin{rem}
We used a very rough estimate to control the term
$$
 \exp\left(\int_s^\rho\Re[\psi(r,\xi)]\mathrm{d}r \right)
 \leq  \exp\left(\int_0^\rho|\Re[\psi(r,\xi)]|\mathrm{d}r \right)
$$
in the proof of Corollary \ref{cor 202307014 01} since there is a possibility that the real part of $\psi(r,\xi)$ is non-negative for all $r$ and $\xi$. 
However, this rough estimate makes us lose all decay of the exponential term derived from negative values of 
 the real part of the symbol $\psi(s,\xi)$ in general. 
On the other hand, recall that the conditions \eqref{20230624 10} and \eqref{20230624 11}  in Theorem \ref{weak solution thm} do not require controls of the term $\exp\left(\int_0^\rho|\Re[\psi(r,\xi)]|\mathrm{d}r \right)$.
Especially, Theorem \ref{weak solution thm} works for symbols that are not locally bounded by confirming conditions \eqref{20230624 10} and \eqref{20230624 11}.
In other words, our main theorem could apply to more general sign changing symbols by estimating the term
$\exp\left(\int_s^\rho\Re[\psi(r,\xi)]\mathrm{d}r \right)$ for all $s<\rho$ instead of using the rough estimate above.
This kind of example can be found in the proof of Theorem \ref{main log} (see Section \ref{log thm pf} below).
\end{rem}

If our input data $f$ exhibits better integrability concerning the time variable, it may lead to the solution $u$ being confined to a more restricted space, thereby demonstrating enhanced integrability in relation to the time variable.
This can be accomplished through local estimates that are readily derived thanks to H\"older's inequality.

\begin{corollary}
								\label{corollary weight pq}
Let $p,q \in [1,\infty]$, $W_0(\xi)$  and $W_1(\xi)$ be  positive measurable functions defined on $\bR^d$, and $W_2(t,\xi)$ be a positive measurable function on $(0,T) \times \bR^d$  so that
both $W_0(\xi)$  and $W_1(\xi)$ are locally bounded below and $W_1(\xi)$ is a lower bound of $W_2(t,\xi)$, \textit{i.e.}
for each $t \in (0,T)$ and $R \in (0,\infty)$, there exist positive constants $\kappa_0(R)$ and $\kappa_1(R)$ so that
\begin{align*}
W_0(\xi) \geq \kappa_0(R) \quad \forall \xi \in B_R
\end{align*}
and
\begin{align}
						\label{special lower bound}
W_2(s,\xi) \geq  W_1(\xi) \geq \kappa_1(R) \quad \forall (s,\xi) \in (0,t) \times B_R.
\end{align}
Assume that
$u_0 \in \cF^{-1}L_{q,\ell oc}(\bR^d,  W_0^q(\xi) \mathrm{d}\xi)$ and $f \in \cF^{-1}L_{p,q,t\text{-}loc, x \text{-}\ell oc}\left(  (0,T) \times \bR^d, \mathrm{d}t W_2^q(t,\xi)\mathrm{d}\xi\right)$.
Additionally, suppose that the symbol $\psi(t,\xi)$ satisfies the following condition:
\begin{align}
								\notag
&\left\|\frac{1+|\psi(\rho,\xi)|}{W_0(\xi)} \exp\left(\int_{0}^\rho\Re[\psi(r,\xi)] \right)\mathrm{d}r \right\|_{L_{p,\infty}\left( (0,t) \times B_R , \mathrm{d}\rho \mathrm{d} \xi \right) } \\
									\label{20230715 10}
&+\left\|\frac{1+|\psi(\rho,\xi)|}{W_1(\xi)} \exp\left(\int_{0}^\rho \left|\Re[\psi(r,\xi)]\right| \right)\mathrm{d}r \right\|_{L_{p,\infty}\left( (0,t) \times B_R , \mathrm{d}\rho \mathrm{d} \xi \right) } 
< \infty
 \end{align}
for all $t \in (0,T)$ and $R \in (0,\infty)$. Then there exists a unique Fourier-space weak solution $u$ to \eqref{ab eqn}  in 
$$
 \cF^{-1}L_{p,q,t\text{-}loc, x\text{-}\ell oc}\left( (0,T) \times \bR^d, \mathrm{d}t (1+|\psi(t,\xi)|)\mathrm{d}\xi \right).
$$
Moreover, for each $t \in (0,T)$ and $R \in (0,\infty)$, the solution $u$ satisfies
\begin{align}
							\notag
&\left\| (1+|\psi(s,\xi)|) \cF[u(s,\cdot)](\xi) \right\|_{L_{p,q}\left( (0,t) \times B_R, \mathrm{d}s\mathrm{d}\xi \right)}\\
							\notag
&\leq \left\|\frac{1+|\psi(\rho,\xi)|}{W_0(\xi)} \exp\left(\int_{0}^\rho\Re[\psi(r,\xi)] \right)\mathrm{d}r \right\|_{L_{p,\infty}\left( (0,t) \times B_R , \mathrm{d}\rho \mathrm{d} \xi \right) } \|\cF[u_0]\|_{L_{q} \left(B_R, \mathrm{d}s W_0^q(\xi) \mathrm{d}\xi \right)} \\
							\label{20230715 20}
&\quad + t^{1/p'}\left\|\frac{1+|\psi(\rho,\xi)|}{W_1(\xi)} \exp\left(\int_0^\rho\Re[\psi(r,\xi)]\mathrm{d}r  \right)\right\|_{L_{p,\infty}\left( (0,t) \times B_R , \mathrm{d}\rho\mathrm{d} \xi \right) } 
\|\cF[f(s,\cdot)(\xi)\|_{L_{p,q} \left( (0,t) \times B_R, \mathrm{d}s W_2^q(\xi) \mathrm{d}\xi \right)},
\end{align}
where  $p'$ is the H\"older conjugate of $p$ and $1/\infty:=0$.
\end{corollary}
\begin{proof}
Observe that \eqref{20230715 10} implies \eqref{20230715 01} due to H\"older's inequality.
Thus by Corollary \ref{cor 202307014 01}, there exists a unique Fourier-space weak solution $u$ to \eqref{ab eqn} in 
$$
\cF^{-1}L_{\infty,1,t\text{-}loc, x\text{-}\ell oc}\left( (0,T) \times \bR^d \right) \cap \cF^{-1}L_{1,1,t\text{-}loc, x\text{-}\ell oc}\left( (0,T) \times \bR^d, \mathrm{d}t |\psi(t,\xi)|\mathrm{d}\xi \right),
$$
which is given by 
\begin{align*}
u(t,x)=  \cF^{-1}\left[  \exp\left(\int_0^t\psi(r,\cdot)\mathrm{d}r \right) \cF[u_0] \right](x)
+\cF^{-1}\left[ \int_0^t  \exp\left(\int_s^t\psi(r,\cdot)\mathrm{d}r \right) \cF[f(s,\cdot)]\mathrm{d}s \right](x)
\end{align*}
in the sense of \eqref{20240313 01}.
Thus taking the Fourier transforms above, we have
\begin{align}
									\label{20230715 40}
\cF[u(t,\cdot)](\xi)=    \exp\left(\int_0^t\psi(r,\xi)\mathrm{d}r \right) \cF[u_0](\xi) 
+\int_0^t  \exp\left(\int_s^t\psi(r,\xi)\mathrm{d}r \right) \cF[f(s,\cdot)](\xi)\mathrm{d}s
\end{align}
$(a.e.)~(t,\xi) \in (0,T)\times \bR^d$.
Note that the last two terms in \eqref{20230715 20} are finite
due to assumptions on weights and the conditions on $u_0$ and $f$.
It suffices to show \eqref{20230715 20}. 
By \eqref{20230715 40}, the H\"older inequality, the generalized Minkowski inequality, and \eqref{special lower bound},
\begin{align*}
&\|(1+|\psi(s,\xi)|) \cF[u(s,\cdot)](\xi)\|_{L_q(B_R)} \\
&\leq \left\|\frac{1+|\psi(s,\xi)|}{W_0(\xi)} \exp\left(\int_{0}^s\Re[\psi(r,\xi)] \right)\mathrm{d}r \right\|_{L_{\infty}\left(  B_R ,  \mathrm{d} \xi \right) } \|\cF[u_0]\|_{L_{q} \left(B_R,  W_0^q(\xi) \mathrm{d}\xi \right)} \\
&\quad + \left\|\frac{1+|\psi(s,\xi)|}{W_1(\xi)} \exp\left(\int_{0}^s|\Re[\psi(r,\xi)]|\right)\mathrm{d}r \right\|_{L_{\infty}\left(  B_R ,  \mathrm{d} \xi \right) } \int_0^s\|\cF[f(\rho,\cdot)](\xi)\|_{L_{q} \left(B_R,  W_2^q(\rho,\xi) \mathrm{d}\xi \right)} \mathrm{d}\rho
\end{align*}
$(a.e.)~s \in (0,T)$.
Therefore taking the $L_p$-norms with respect to the variable $s$, we have 
\begin{align}
							\notag
&\left\| (1+|\psi(s,\xi)|) \cF[u(s,\cdot)](\xi) \right\|_{L_{p,q}\left( (0,t) \times B_R, \mathrm{d}\xi \right)}\\
							\notag
&\leq \left\|\frac{1+|\psi(s,\xi)|}{W_0(\xi)} \exp\left(\int_{0}^s\Re[\psi(r,\xi)] \right)\mathrm{d}r \right\|_{L_{p,\infty}\left( (0,t) \times B_R , \mathrm{d}s \mathrm{d} \xi \right) } \|\cF[u_0]\|_{L_{q} \left(B_R, \mathrm{d}s W_0^q(\xi) \mathrm{d}\xi \right)} \\
							\label{20230715 30}
&\quad + \left\|\frac{1+|\psi(s,\xi)|}{W_1(\xi)} \exp\left(\int_0^s|\Re[\psi(r,\xi)]|\mathrm{d}r  \right)\right\|_{L_{p,\infty}\left( (0,t) \times B_R , \mathrm{d}s\mathrm{d} \xi \right) } 
\sup_{s \in (0,t)}\int_0^s\|\cF[f(\rho,\cdot)](\xi)\|_{L_{q} \left(B_R, W_2^q(\rho,\xi) \mathrm{d}\xi \right)} \mathrm{d}\rho.
\end{align}
Note that for any $t \in (0,T)$, 
\begin{align}
							\notag
\sup_{s \in (0,t)}\int_0^s\|\cF[f(\rho,\cdot)](\xi)\|_{L_{q} \left(B_R,  W_2^q(\rho,\xi) \mathrm{d}\xi \right)} \mathrm{d}\rho
&\leq \int_0^t\|\cF[f(\rho,\cdot)](\xi)\|_{L_{q} \left(B_R, W_2^q(\rho,\xi) \mathrm{d}\xi \right)} \mathrm{d}\rho \\
							\label{20230715 31}
&\leq t^{1/p'} \|\cF[f(\rho,\cdot)](\xi)\|_{L_{p,q} \left((0,t) \times B_R, \mathrm{d}\rho  W_2^q(\rho,\xi) \mathrm{d}\xi \right)} \mathrm{d}\rho.
\end{align}
Combining \eqref{20230715 30} and \eqref{20230715 31}, we finally have \eqref{20230715 20}.
The corollary is proved.
\end{proof}

At last, we are fully equipped to establish a link between the two weighted Bessel potential spaces introduced in Section \ref{fourier section} and our theories regarding well-posedness.

Recall that the exponents of weighted spaces play crucial roles in applying embedding inequalities.
Additionally, the weighted Bessel potential spaces become identical with classical Bessel potential spaces if the weighted exponents are zero.
We begin by presenting a corollary to show that data within the classical Bessel potential spaces can be included in our well-posedness theories when the integral exponents with respect to the spatial variable are constrained within the range of $[1,2]$.
\begin{corollary}
								\label{p small coro}
Let $p \in [1,\infty]$, $q \in [1,2]$, $\gamma_1, \tilde\gamma_1 \in \bR$,  $u_0 \in H_{q}^{\gamma_1}(\bR^d)$, and $f \in \fH^{\tilde\gamma_1}_{p,q,t\text{-}loc}\left( (0,T) \times \bR^d\right)$. Assume that the symbol $\psi(t,\xi)$ satisfies 
 \begin{align}
									\label{20230715 50-2}
 \left\|\frac{1+|\psi(\rho,\xi)|}{(1+|\xi|^2)^{\gamma/2}} \exp\left(\int_{0}^\rho|\Re[\psi(r,\xi)]| \right)\mathrm{d}r \right\|_{L_{p,\infty}\left( (0,t) \times B_R , \mathrm{d}\rho \mathrm{d} \xi \right) } < \infty
 \end{align}
for all $t \in (0,T)$ and $R \in (0,\infty)$, where $\gamma = \min\{\gamma_1,\tilde \gamma_1\}$.
 Then there exists a unique Fourier-space weak solution $u$ to \eqref{ab eqn}  in 
$$
 \cF^{-1}L_{p,q',t\text{-}loc, x\text{-}\ell oc}\left( (0,T) \times \bR^d, \mathrm{d}t (1+|\psi(t,\xi)|)\mathrm{d}\xi \right).
$$
Moreover, for any $t \in (0,T)$ and $R \in (0,\infty)$, the solution $u$ satisfies
\begin{align}
							\notag
&\left\| (1+|\psi(s,\xi)|) \cF[u(s,\cdot)](\xi) \right\|_{L_{p,q'}\left( (0,t) \times B_R, \mathrm{d}s\mathrm{d}\xi \right)}\\
							\notag
&\leq \left(\frac{1}{2\pi}\right)^{\frac{d(2-q)}{2q}}  \left\|\frac{1+|\psi(\rho,\xi)|}{(1+|\xi|^2)^{\gamma/2}} \exp\left(\int_{0}^\rho\Re[\psi(r,\xi)] \right)\mathrm{d}r \right\|_{L_{p,\infty}\left( (0,t) \times B_R , \mathrm{d}\rho \mathrm{d} \xi \right) }  \|u_0\|_{H^{\gamma_1}_{q}(\bR^d)} \\
							\notag
&\quad + \left(\frac{1}{2\pi}\right)^{\frac{d(2-q)}{2q}}  t^{1/p'}\left\|\frac{1+|\psi(\rho,\xi)|}{(1+|\xi|^2)^{\gamma/2}} \exp\left(\int_0^\rho|\Re[\psi(r,\xi)]|\mathrm{d}r  \right)\right\|_{L_{p,\infty}\left( (0,t) \times B_R , \mathrm{d}\rho\mathrm{d} \xi \right) } 
 \|f\|_{\fH^{\tilde\gamma_1}_{p,q,t\text{-}loc}\left( (0,t) \times \bR^d\right)},
\end{align}
where  $p'$ is the H\"older conjugate of $p$ and $1/\infty:=0$.
\end{corollary}
\begin{proof}
Put
\begin{align*}
W_0(\xi) =W_1(\xi)=W_2(t,\xi)= \left( 1+|\xi|^2  \right)^{\gamma/2}
\end{align*}
for all $t$ and $\xi$.
Due to \eqref{20230715 50-2}, it is obvious that
\begin{align*}
&\left\|\frac{1+|\psi(\rho,\xi)|}{W_0(\xi)} \exp\left(\int_{0}^\rho\Re[\psi(r,\xi)] \right)\mathrm{d}r \right\|_{L_{p,\infty}\left( (0,t) \times B_R , \mathrm{d}\rho \mathrm{d} \xi \right) } \\
&+\left\|\frac{1+|\psi(\rho,\xi)|}{W_1(\xi)} \exp\left(\int_{0}^\rho \left|\Re[\psi(r,\xi)]\right| \right)\mathrm{d}r \right\|_{L_{p,\infty}\left( (0,t) \times B_R , \mathrm{d}\rho \mathrm{d} \xi \right) } 
< \infty.
 \end{align*}
By Proposition \ref{20230701 10},
$u_0 \in \cF^{-1}H_{q',out}^{0,\gamma_1}(\bR^d)$
and
$f \in \cF^{-1}\fH^{0,\tilde\gamma_1}_{p,q',out,t\text{-}loc}\left( (0,T) \times \bR^d, \mathrm{d}t\mathrm{d}x \right)$.
Additionally,
\begin{align*}
 \|\cF[u_0]\|_{H^{0,\gamma_1}_{q',out}(\bR^d)} \leq \left(\frac{1}{2\pi}\right)^{\frac{d(2-q)}{2q}} \|u_0\|_{H^{\gamma_1}_{q}(\bR^d)}
\end{align*}
and
\begin{align*}
 \|\cF[f]\|_{\fH^{0,\tilde\gamma_1}_{p,q',out,t\text{-}loc}\left( (0,t) \times \bR^d \right)} \leq \left(\frac{1}{2\pi}\right)^{\frac{d(2-q)}{2q}} \|f\|_{\fH^{\tilde\gamma_1}_{p,q,t\text{-}loc}\left( (0,t) \times \bR^d\right)}.
\end{align*}
Thus it is easy to show that  for all $t \in (0,T)$ and $R \in (0,\infty)$,
\begin{align*}
\|\cF[u_0]\|_{L_{q'}(B_R,  W^{q'}(\xi) \mathrm{d}\xi)}
\leq \|\cF[u_0]\|_{L_{q'}(\bR^d,  W^{q'}(\xi) \mathrm{d}\xi)}
&\leq \|\cF[u_0]\|_{H^{0,\gamma_1}_{q',out}(\bR^d)}  \\
&\leq \left(\frac{1}{2\pi}\right)^{\frac{d(2-q)}{2q}} \|u_0\|_{H^{\gamma_1}_{q}(\bR^d)}
\end{align*}
and
\begin{align*}
\|\cF[f]\|_{L_{p,q'}\left(  (0,t) \times B_R, \mathrm{d}s W^{q'}(\xi)\mathrm{d}\xi\right)} 
\leq \|\cF[f]\|_{L_{p,q'}\left(  (0,t) \times \bR^d, \mathrm{d}s W^{q'}(\xi)\mathrm{d}\xi\right)} 
&\leq \|\cF[f]\|_{\fH^{0,\tilde\gamma_1}_{p,q',out,t\text{-}loc}\left( (0,t) \times \bR^d \right)} \\
&\leq \left(\frac{1}{2\pi}\right)^{\frac{d(2-q)}{2q}} \|f\|_{\fH^{\tilde\gamma_1}_{p,q,t\text{-}loc}\left( (0,t) \times \bR^d\right)}.
\end{align*}
Thus the conclusion directly comes from Corollary \ref{corollary weight pq}.
\end{proof}
\begin{rem}
The restriction on the integral exponent $q \in [1,2]$  is not crucial in general if the solution $u$ has a nice stability with respect to data $u_0$ and $f$ by applying a standard duality argument (\textit{cf.} \cite{KID SBL KHK 2015,KID SBL KHK 2016}).
\end{rem}

Next, we examine data within inner weighted spaces featuring non-zero weighted exponents.
Unfortunately, this imposes specific constraints on the range of exponents,  which arise as a natural consequence of the embedding inequalities.

\begin{corollary}
								\label{q small coro}
Let $p \in [1,\infty]$, $q \in (1,2]$, $\gamma_1, \tilde \gamma_1 \in \left(-\frac{d(p-1)}{p}, \infty\right)$, $\gamma_2, \tilde \gamma_2 \in [0,\infty)$,  $u_0 \in H_{q, in}^{\gamma_1,\gamma_2}(\bR^d)$, and $f \in \fH^{\tilde \gamma_1, \tilde \gamma_2}_{p,q,in,t\text{-}loc}\left( (0,T) \times \bR^d\right)$. Assume that the symbol $\psi(t,\xi)$ satisfies 
 \begin{align*}
 \left\|\frac{1+|\psi(\rho,\xi)|}{(1+|\xi|^2)^{\gamma/2}} \exp\left(\int_{0}^\rho |\Re[\psi(r,\xi)]| \right)\mathrm{d}r \right\|_{L_{p,\infty}\left( (0,t) \times B_R , \mathrm{d}\rho \mathrm{d} \xi \right) } < \infty
 \end{align*}
for all $t \in (0,T)$ and $R \in (0,\infty)$, where $\gamma = \min\{\gamma_1,\tilde \gamma_1\}$.
 Then there exists a unique Fourier-space weak solution $u$ to \eqref{ab eqn}  in 
$$
 \cF^{-1}L_{p,q',t\text{-}loc, x\text{-}\ell oc}\left( (0,T) \times \bR^d, \mathrm{d}t (1+|\psi(t,\xi)|)\mathrm{d}\xi \right).
$$
Moreover, for any $t \in (0,T)$ and $R \in (0,\infty)$, the solution $u$ satisfies
\begin{align}
							\notag
&\left\| (1+|\psi(s,\xi)|) \cF[u(s,\cdot)](\xi) \right\|_{L_{p,q'}\left( (0,t) \times B_R, \mathrm{d}s\mathrm{d}\xi \right)}\\
							\notag
&\lesssim   \left\|\frac{1+|\psi(\rho,\xi)|}{(1+|\xi|^2)^{\gamma/2}} \exp\left(\int_{0}^\rho\Re[\psi(r,\xi)] \right)\mathrm{d}r \right\|_{L_{p,\infty}\left( (0,t) \times B_R , \mathrm{d}\rho \mathrm{d} \xi \right) }  \|u_0\|_{H^{\gamma_1,\gamma_2}_{q,in}(\bR^d)} \\
							\notag
&\quad +  t^{1/p'}\left\|\frac{1+|\psi(\rho,\xi)|}{(1+|\xi|^2)^{\gamma/2}} \exp\left(\int_0^\rho|\Re[\psi(r,\xi)]|\mathrm{d}r  \right)\right\|_{L_{p,\infty}\left( (0,t) \times B_R , \mathrm{d}\rho\mathrm{d} \xi \right) } 
 \|f\|_{\fH^{\tilde \gamma_1, \tilde \gamma_2}_{p,q,in,t\text{-}loc}\left( (0,t) \times \bR^d\right)},
\end{align}
where  $p'$ is the H\"older conjugate of $p$ and $1/\infty:=0$.
\end{corollary}
\begin{proof}
The proof of this corollary is almost identical with that of Corollary \ref{p small coro}.
We only highlight the major difference.
Due to Proposition \ref{inner embedding}, we may assume that 
$$
\gamma_1, \tilde \gamma_1 \in (-d(p-1)/p, d/p)
$$
without loss of generality.
Moreover, Proposition \ref{outer embedding 2} is used to control $L_{q'}$-norms of the Fourier transforms of data in order to consider general regularity exponents of outer spaces instead of only 0.
Indeed, for all $t \in (0,T)$ and $R \in (0,\infty)$,
\begin{align*}
\|\cF[u_0]\|_{L_{q'}(B_R,  (1+|\xi|^2)^{(q'\gamma)/2} \mathrm{d}\xi)}
\leq \|\cF[u_0]\|_{L_{q'}(\bR^d,  (1+|\xi|^2)^{(q'\gamma)/2}\mathrm{d}\xi)}
\leq \|\cF[u_0]\|_{H^{0,\gamma_1}_{q',out}(\bR^d)}  
&\lesssim \|\cF[u_0]\|_{H^{\gamma_2,\gamma_1}_{q',out}(\bR^d)}   \\
&\lesssim \|u_0\|_{H^{\gamma_1,\gamma_2}_{q,in}(\bR^d)}
\end{align*}
and
\begin{align*}
\|\cF[f]\|_{L_{p,q'}\left(  (0,t) \times B_R, \mathrm{d}s (1+|\xi|^2)^{(q'\gamma)/2}\mathrm{d}\xi\right)} 
&\leq \|\cF[f]\|_{L_{p,q'}\left(  (0,t) \times \bR^d, \mathrm{d}s (1+|\xi|^2)^{(q'\gamma)/2}\mathrm{d}\xi\right)}  \\
&\leq \|\cF[f]\|_{\fH^{0,\tilde \gamma_1}_{p,q',out,t\text{-}loc}\left( (0,t) \times \bR^d \right)}  \\
&\lesssim \|\cF[f]\|_{\fH^{\tilde \gamma_2,\tilde \gamma_1}_{p,q',out,t\text{-}loc}\left( (0,t) \times \bR^d \right)} \\
&\lesssim \|f\|_{\fH^{\tilde\gamma_1, \tilde \gamma_2}_{p,q,in,t\text{-}loc}\left( (0,t) \times \bR^d\right)}.
\end{align*}
The corollary is proved.
\end{proof}

Moving forward, our focus shifts to the parameter $q$ which exceeds the value of 2.
In simpler terms, we are now capable of dealing with the data $u_0$ and $f$ in the outer weighted spaces, thanks to the embeddings developed in Section \ref{fourier section}, even when $q$ is within the range of $(2, \infty]$.
\begin{corollary}
									\label{p big coro}
Let $p \in [1,\infty]$, $q \in (2,\infty]$, $\gamma_1, \tilde \gamma_1 \in \bR$, $\gamma_2, \tilde \gamma_2 \in \left(\frac{d(q-2)}{2(q-1)},\infty \right)$,  $u_0 \in H_{q,out}^{\gamma_1,\gamma_2}(\bR^d)$, and $f \in \fH^{\tilde \gamma_1,\tilde \gamma_2}_{p,q,out,t\text{-}loc}\left( (0,T) \times \bR^d\right)$. Assume that the symbol $\psi(t,\xi)$ satisfies 
 \begin{align}
									\label{20230722 01}
 \left\|\frac{1+|\psi(\rho,\xi)|}{(1+|\xi|^2)^{\gamma/2}} \exp\left(\int_{0}^\rho|\Re[\psi(r,\xi)]| \right)\mathrm{d}r \right\|_{L_{p,\infty}\left( (0,t) \times B_R , \mathrm{d}\rho \mathrm{d} \xi \right) } < \infty
 \end{align}
for all $t \in (0,T)$ and $R \in (0,\infty)$, where $\gamma = \min\{\gamma_1,\tilde \gamma_1\}$.
 Then there exists a unique Fourier-space weak solution $u$ to \eqref{ab eqn}  in 
$$
 \cF^{-1}L_{p,2,t\text{-}loc, x\text{-}\ell oc}\left( (0,T) \times \bR^d, \mathrm{d}t (1+|\psi(t,\xi)|)\mathrm{d}\xi \right).
$$
Moreover, for any $t \in (0,T)$ and $R \in (0,\infty)$, the solution $u$ satisfies
\begin{align}
							\notag
&\left\| (1+|\psi(s,\xi)|) \cF[u(s,\cdot)](\xi) \right\|_{L_{p,2}\left( (0,t) \times B_R, \mathrm{d}s\mathrm{d}\xi \right)}\\
							\notag
&\lesssim_{d,p,q,\gamma_2,\tilde\gamma_2} \left\|\frac{1+|\psi(\rho,\xi)|}{(1+|\xi|^2)^{\gamma/2}} \exp\left(\int_{0}^\rho\Re[\psi(r,\xi)] \right)\mathrm{d}r \right\|_{L_{p,\infty}\left( (0,t) \times B_R , \mathrm{d}\rho \mathrm{d} \xi \right) } \|u_0\|_{H_{q,out}^{\gamma_1,\gamma_2}(\bR^d)} \\
							\label{20230722 02}
&\quad + t^{1/p'}\left\|\frac{1+|\psi(\rho,\xi)|}{(1+|\xi|^2)^{\gamma/2}} \exp\left(\int_0^\rho|\Re[\psi(r,\xi)]|\mathrm{d}r  \right)\right\|_{L_{p,\infty}\left( (0,t) \times B_R , \mathrm{d}\rho\mathrm{d} \xi \right) } 
\|f\|_{\fH^{\tilde \gamma_1,\tilde \gamma_2}_{p,q,out}\left( (0,t) \times \bR^d\right)},
\end{align}
where  $p'$ is the H\"older conjugate of $p$ and $1/\infty:=0$.
\end{corollary}
\begin{proof}
It is another easy application of Corollary \ref{corollary weight pq}.
First, we apply an embedding inequality developed in Proposition \ref{p large embedding}.
Put
\begin{align*}
\delta_1=\gamma_2 ~\text{and}~\delta_2=\tilde \gamma_2.
\end{align*}
Then by Proposition \ref{p large embedding}, we have
\begin{align}
									\label{20230722 10}
\|\cF[u_0]\|_{H_{2,in}^{0,\gamma_1}(\bR^d)}
\leq \left\| \left(1+|\cdot|^2\right)^{-\delta_1/2}  \right\|_{L_{2/(2-q')}(\bR^d)} \|u_0\|_{H_{q,out}^{\gamma_1,\gamma_2}(\bR^d)}
\end{align}
and
\begin{align}
									\label{20230722 11}
\|\cF[f](t,\cdot)\|_{H_{2,in}^{0,\tilde \gamma_1}(\bR^d)}
\leq \left\| \left(1+|\cdot|^2\right)^{-\delta_2/2}  \right\|_{L_{2/(2-q')}(\bR^d)} \|f(t,\cdot)\|_{H_{q,out}^{\tilde \gamma_1,\tilde \gamma_2}(\bR^d)} \quad \forall t \in (0,T).
\end{align}
In particular,
$u_0 \in \cF^{-1}L_{2,\ell oc}(\bR^d,  W^2(\xi) \mathrm{d}\xi)$
and $f \in \cF^{-1}L_{p,2,t\text{-}loc, x \text{-}\ell oc}\left(  (0,T) \times \bR^d, \mathrm{d}t W^2(\xi)\mathrm{d}\xi\right)$
with $W(\xi) = (1+|\xi|^2)^{\gamma/2}$. 
Additionally, due to \eqref{20230722 01}, we have
 \begin{align*}
 \left\|\frac{1+|\psi(\rho,\xi)|}{W(\xi)} \exp\left(\int_{0}^\rho\Re[\psi(r,\xi)] \right)\mathrm{d}r \right\|_{L_{p,\infty}\left( (0,t) \times B_R , \mathrm{d}\rho \mathrm{d} \xi \right) } < \infty
 \end{align*}
for all $t \in (0,T)$ and $R \in (0,\infty)$.
Thus by applying Corollary \ref{corollary weight pq} with
\begin{align*}
W(\xi)=W_0(\xi) =W_1(\xi)=W_2(t,\xi),
\end{align*}
there exists a unique Fourier-space weak solution $u$ to \eqref{ab eqn}  in 
$$
 \cF^{-1}L_{p,2,t\text{-}loc, x\text{-}\ell oc}\left( (0,T) \times \bR^d, \mathrm{d}t (1+|\psi(t,\xi)|)\mathrm{d}\xi \right)
$$
and this solution $u$ satisfies \eqref{20230722 02} due to \eqref{20230715 20}, \eqref{20230722 10}, and \eqref{20230722 11}.
\end{proof}

\mysection{Logarithmic operators with complex-valued coefficients and proof of Theorem \ref{main log}}
									\label{log thm pf}

Demonstrating the well-posedness of evolutionary equations involving the logarithmic Laplacian is an important application of our work.
We have already proposed natural extensions of the operator with coefficients in \eqref{log eqn}.
We begin this section by providing a proof for Theorem \ref{main log}.
\begin{proof}[Proof of Theorem \ref{main log}]

Put 
$$
\psi(t,\xi)= \beta(t) \log \left(\psi_{exp}(t,\xi)\right)
$$
and we show that the symbol $\psi(t,\xi)$ and data $u_0$ and $f$ satisfy \eqref{20230624 10} and \eqref{20230624 11}. 
Observe that
\begin{align*}
\Re[\psi(t,\xi)] = \Re[\beta(t)] \log \left(|\psi_{exp}(t,\xi)|\right)
-\Im[\beta(t)] \arg [\psi_{exp}(t,\xi) ] 
\end{align*}
and recall that the real-valued exponential function is convex on any open interval of $\bR$,
where $\arg [\psi_{exp}(t,\xi) ] $ denotes the argument of the complex number $\psi_{exp}(t,\xi) $.
Thus for all $t \in (0,T)$ and $R \in (0,\infty)$, applying Jensen's inequality and considering the standard branch cut, we have
\begin{align*}
&\int_{B_R}\exp\left(\int_0^t\Re[\psi(r,\xi)]\mathrm{d}r \right) \cF[u_0](\xi) \mathrm{d}\xi
+  \int_{B_R} \int_0^t  \exp\left(\int_s^t\Re[\psi(r,\xi)]\mathrm{d}r \right) \left|\cF[f(s,\cdot)](\xi) \right| \mathrm{d}s \mathrm{d}\xi  \\
&\leq \int_{B_R}\exp\left(\frac{1}{t}\int_0^t\log \left(|\psi_{exp}(r,\xi)|^{t \Re[\beta(r)]}  \right)\mathrm{d}r + 2\pi \int_0^t |\Im[\beta(r)]|\mathrm{d}r \right) \cF[u_0](\xi) \mathrm{d}\xi \\
&\quad +\int_{B_R} \int_0^t  \exp\left(\frac{1}{t-s}\int_s^t\log \left[\left(|\psi_{exp}(r,\xi)|^{(t-s) \Re[\beta(r)]}\right) \right] \mathrm{d}r  + 2\pi \int_0^t |\Im[\beta(r)]|\mathrm{d}r \right) \left|\cF[f(s,\cdot)](\xi) \right| \mathrm{d}s \mathrm{d}\xi  \\
&\leq \exp\left( 2\pi \int_0^t |\Im[\beta(r)]|\mathrm{d}r \right)\int_{B_R}\frac{1}{t}\int_0^t|\psi_{exp}(r,\xi)|^{t \Re[\beta(r)]}\mathrm{d}r  \cF[u_0](\xi) \mathrm{d}\xi \\
&\quad + \exp\left( 2\pi \int_0^t |\Im[\beta(r)]|\mathrm{d}r \right)
\int_0^t \int_{B_R}  \frac{1}{t-s}\int_s^t\left(|\psi_{exp}(r,\xi)|^{(t-s) \Re[\beta(r)]}\right) \mathrm{d}r 
   \left|\cF[f(s,\cdot)](\xi) \right| \mathrm{d}s \mathrm{d}\xi.
\end{align*}
Thus the above inequalities with \eqref{log con 0} and \eqref{log con 1} show that \eqref{20230624 10} holds.
Next, observe that 
\begin{align*}
|\psi(\rho,\xi)| 
\leq |\beta(\rho)|\left| \log \left(|\psi_{\exp}(\rho,\xi)|\right) \right|+ 2\pi|\beta(\rho)|
\end{align*}
if one considers the standard Branch in the complex logarithm.
Therefore for all $ t \in (0,T)$ and $R \in (0,\infty)$,
\begin{align}
							\notag
& \int_{B_R}\int_0^t |\psi(\rho,\xi)| \exp\left(\int_0^\rho \Re[\psi(r,\xi)]\mathrm{d}r \right) \cF[u_0](\xi)\mathrm{d}\rho \mathrm{d}\xi \\
							\notag
&+\int_{B_R}\int_0^t |\psi(\rho,\xi)| \int_0^\rho \exp\left(\int_s^\rho\Re[\psi(r,\xi)]\mathrm{d}r \right)  \left|\cF[f(s,\cdot)](\xi) \right| \mathrm{d}s \mathrm{d}\rho \mathrm{d}\xi \\
							\notag
&\lesssim   \int_{B_R} \left[\int_0^t |\beta(\rho)|\left(1+\left| \log \left(|\psi_{\exp}(\rho,\xi)|\right) \right|\right)  \left(\frac{1}{\rho}\int_0^\rho|\psi_{exp}(r,\xi)|^{\rho \Re[\beta(r)]}\mathrm{d}r \right) \mathrm{d}\rho \cF[u_0](\xi)  \right] \mathrm{d}\xi \\
							\notag
&\quad +\int_{B_R}\Bigg[ \int_0^t \left(\int_s^t |\beta(\rho)|\left(1+\left| \log \left(|\psi_{\exp}(\rho,\xi)|\right) \right|\right) \left(\frac{1}{\rho-s}\int_s^\rho\left(|\psi_{exp}(r,\xi)|^{(\rho-s) \Re[\beta(r)]}\right) \mathrm{d}r \right) 
 \mathrm{d}\rho  \right)  \\
							\label{20240204 01}
&\qquad\qquad\qquad\qquad\qquad\qquad\qquad\qquad\qquad\qquad\qquad\qquad\qquad\qquad\qquad\qquad  
\left|\cF[f(s,\cdot)](\xi) \right| \mathrm{d}s \Bigg] \mathrm{d}\xi,
\end{align}
where
\begin{align*}
\exp\left( 2\pi \int_0^t |\Im[\beta(r)]|\mathrm{d}r \right) 
\leq \exp\left( 2\pi \int_0^t |\beta(r)|\mathrm{d}r \right)  < \infty
\end{align*}
is used in the inequality above.
Finally, applying \eqref{log con 2} to \eqref{20240204 01}, we have \eqref{20230624 11}.
The theorem is proved.
\end{proof}							

Next, we consider the special case of $\psi_{\exp}(t,\xi)$.
More precisely, we consider the second-order case, \textit{i.e.}
\begin{align*}
\psi_{\exp}(t,\xi)
=\alpha^{ij}(t)\xi^i\xi^j.
\end{align*}
Recalling
\begin{align*}
\log \left( \psi_{exp}(t,-i \nabla)  \right)u(t,x)
:=\cF^{-1}\left[\log \left(\psi_{exp}(t,\xi)\right)\cF[u(t,\cdot)](\xi) \right](x),
\end{align*}
we use the special notation
\begin{align*}
\beta(t)\log \left( -\alpha^{ij}(t) D_{x^i} D_{x^j} \right)u(t,x)
\end{align*}
to denote
\begin{align*}
\cF^{-1}\left[\log \left(\alpha^{ij}(t)\xi^i \xi^j\right)\cF[u(t,\cdot)](\xi) \right](x).
\end{align*}
Combining all the information above, we finally state an evolution equation with the operator
$\beta(t)\log \left( -\alpha^{ij}(t) D_{x^i} D_{x^j} \right)$ as follows:
\begin{equation}
\begin{cases}
								\label{log second}
\partial_tu(t,x)=\beta(t)\log \left( -\alpha^{ij}(t) D_{x^i} D_{x^j} \right)u(t,x)+f(t,x),\quad &(t,x)\in(0,T)\times\mathbb{R}^d,\\
u(0,x)=u_0,\quad & x\in\mathbb{R}^d.
\end{cases}
\end{equation}
We only consider real-valued symmetric $\left( \alpha^{ij}(t) \right)$. 
Additionally, assume that the coefficients $ \alpha^{ij}(t)$ satisfy a uniform ellipticity, \textit{i.e.} for each $ t \in (0,T)$ there exists a positive constant $\nu$ such that
\begin{align}
								\label{20230810 02}
\nu|\xi|^2  \leq \alpha^{ij}(s)\xi^i\xi^j  \leq \frac{1}{\nu} |\xi|^2 \quad  \forall (s,\xi) \in (0,t) \times \bR^d,
\end{align}
which shows that $\psi_{\exp}(t,\xi)=\alpha^{ij}(t)\xi^i\xi^j$ satisfies $\eqref{exp make}$.
\begin{corollary}
								\label{cor log second}
Let $u_0 \in \cF^{-1}L_{1,\ell oc}(\bR^d)$ and $f \in \cF^{-1}L_{1,1,t\text{-}loc,x\text{-}\ell oc}\left((0,T) \times \bR^d \right)$.
Assume that \eqref{20230810 02} holds.
Additionally, suppose that 
\begin{align}
							\label{20230810}
0<\essinf_{ s \in (0,t)}|\beta(s)| \leq \esssup_{ s \in (0,t)}|\beta(s)| <\infty \quad \forall t \in (0,T).
\end{align}
Then there exists a unique Fourier-space weak solution $u$ to \eqref{log second}  in 
$$
\cF^{-1}L_{\infty,1,t\text{-}loc, x\text{-}\ell oc}\left( (0,T) \times \bR^d \right) \cap \cF^{-1}L_{1,1,t\text{-}loc, x\text{-}\ell oc}\left( (0,T) \times \bR^d, |\beta(t)| \mathrm{d}t \left|\log \left(\alpha^{ij}(t)\xi^i\xi^j\right)\right|\mathrm{d}\xi \right).
$$ 
\end{corollary}
\begin{proof}
We use Theorem \ref{main log} to obtain this corollary.
Let $t \in (0,T)$. 
Put
\begin{align*}
\psi_{exp}(r,\xi) = \alpha^{ij}(r) \xi^i \xi^j \quad \forall r \in (0,t),
\end{align*}
\begin{align*}
m=\essinf_{ s \in (0,t)} |\beta(s)| ,
\end{align*}
\begin{align*}
M=\esssup_{ s \in (0,t)} \left(|\beta(s)| + \sum_{i,j}|\alpha^{ij}(s)|\right).
\end{align*}
Due to \eqref{20230810 02} and \eqref{20230810}, it is obvious that $m,M \in (0,\infty)$.
It is sufficient to show that \eqref{log con 1} and \eqref{log con 2} hold.
Applying \eqref{20230810 02} and \eqref{20230810}, for all $t \in (0,T)$ and $R \in (0,\infty)$, we have
\begin{align*}
 \sup_{(s,\xi) \in [0,t) \times B_R} \left(\frac{1}{t-s}\int_s^t |\psi_{exp}(r,\xi)|^{(t-s) \Re[\beta(r)]} \mathrm{d}r\right) 
 \lesssim \left(1+R^{2tM}   \right)
\end{align*}
and
\begin{align*}
&\sup_{(s,\xi) \in [0,t) \times B_R} \left(\int_s^t |\beta(\rho)|\left(1+\left| \log \left(|\psi_{\exp}(\rho,\xi)|\right) \right|\right) \left(\frac{1}{\rho-s}\int_s^\rho\left(|\psi_{exp}(r,\xi)|^{(\rho-s) \Re[\beta(r)]}\right) \mathrm{d}r \right) 
 \mathrm{d}\rho  \right)  \\
&\lesssim \sup_{(s,\xi) \in [0,t) \times B_R} \left(\int_s^t \left(1+\left| \log \left(|\xi|\right) \right|\right) \left(|\xi|^{(\rho-s)m } +|\xi|^{(\rho-s) M}\right)  \mathrm{d}\rho  \right)   \\
&\lesssim \sup_{(s,\xi) \in [0,t) \times B_R}  \left(|\xi|^{(t-s)m } +|\xi|^{(t-s) M}\right)
\lesssim \left(1 + R^{tM}  \right) < \infty.
\end{align*}
The corollary is proved.
\end{proof}

\begin{rem}
The local boundedness of coefficients in \eqref{20230810 02} and \eqref{20230810} could be weakened by considering substitutes with complicated local integrabilities. 
It could be obtained by another application of Theorem \ref{main log}.
However, we do not give the details since conditions become extremely technical.
\end{rem}

\mysection{Second-order partial differential equations and proof of Theorem \ref{main second}}
									\label{second thm pf}

We study \eqref{second eqn} as a particular case of \eqref{ab eqn} in this section. 
We present a proof of Theorem \ref{main second} and give some applications of this theorem when data in weighted Bessel potential spaces. 
First, we prove Theorem \ref{main second}.

\begin{proof}[Proof of Theorem \ref{main second}]
Based on definitions, it is easy to check that \eqref{second eqn} becomes a special case of \eqref{ab eqn}
with the symbol
$$
\psi(t,\xi) = -a^{ij}(t)\xi^i \xi^j+ib^j(t)\xi^j + c(t).
$$
as mentioned in Remark \ref{equivalent solution}.
We use Corollary \ref{corollary weight pq} to prove this theorem.
Thus it is sufficient to show that
 \begin{align*}
  \left\|\frac{1+|a^{ij}(\rho)\xi^i\xi^j-ib^j(t)\xi^j-c(t)|}{W_k(\xi)} \exp\left(\int_{0}^\rho\left|\Re[a^{ij}(r)]\xi^i\xi^j + \Im[b^j(r)] \xi^j -\Re[c(r)]\right|\mathrm{d}r\right) \right\|_{L_{p,\infty}\left( (0,t) \times B_R , \mathrm{d}\rho \mathrm{d} \xi \right) } 
 \end{align*}
is finite for all $t \in (0,T)$, $R \in (0,\infty)$, and $k=0,1$.
It is an easy application of Minkowski's and H\"older's inequalities with \eqref{second as 1}, \eqref{second as 2} and \eqref{second as 2-2}.
Indeed, for each $k=0,1$, $t \in (0,T)$ and $R \in (0,\infty)$, applying \eqref{second as 1}, \eqref{second as 2}, and \eqref{second as 2-2}, we have
\begin{align*}
&  \left\|\frac{1+|a^{ij}(\rho)\xi^i\xi^j-ib^j(t)\xi^j-c(t)|}{W_k(\xi)} \exp\left(\int_{0}^\rho \left|\Re[a^{ij}(r)]\xi^i\xi^j + \Im[b^j(r)] \xi^j -\Re[c(r)] \right|\mathrm{d}r\right) \right\|_{L_{p,\infty}\left( (0,t) \times B_R , \mathrm{d}\rho \mathrm{d} \xi \right) } \\
&\lesssim
 \frac{1+\sum_{i,j}\|a^{ij}\|_{L_p\left( (0,t)\right)}R^2 +\sum_{j}\|b^{j}\|_{L_p\left( (0,t)\right)}R  + \|c\|_{L_p\left( (0,t)\right)}}{\kappa_k(R)} \\
 & \qquad  \times \exp\left(\int_{0}^t \left[R^2 \sum_{i,j}|a^{ij}(r)| +R\sum_{j}|b^j(r)|+|c(r)|\right]\mathrm{d}r \right) \\
&\lesssim
  \frac{1+\sum_{i,j}\|a^{ij}\|_{L_p\left( (0,t)\right)}R^2 + \sum_{j}\|b^{j}\|_{L_p\left( (0,t)\right)}R  + \|c\|_{L_p\left( (0,t)\right)}}{\kappa_k(R)}  \\
&\qquad \times \exp\left(C\left(\sum_{i,j}\|a^{ij}\|_{L_p\left( (0,t)\right)}R^2 + \sum_{j}\|b^{j}\|_{L_p\left( (0,t)\right)}R  + \|c\|_{L_p\left( (0,t)\right)}\right) \right)    < \infty,
\end{align*}
where $C$ is a positive constant depending only on $t$ and $p$. 
The theorem is proved. 
\end{proof}
We examine three distinct scenarios derived from Theorem \ref{main second}, focusing on cases where the data are presented within the (weighted) Bessel potential spaces.
We consider complex-valued coefficients so that $a^{ij}, b^j, c \in L_{p,loc}\left( (0,T) \right)$ for all $i,j \in \{1,\ldots,d\}$ in the following corollaries.
\begin{corollary}
Let $p \in [1,\infty]$, $q \in [1,2]$, $\gamma_1, \gamma_2 \in \bR$,  $u_0 \in H_{q}^{\gamma_1}(\bR^d)$, and $f \in \fH^{\gamma_2}_{p,q,t\text{-}loc}\left( (0,T) \times \bR^d\right)$.  Then there exists a unique weak solution $u$ $($tested by $\cF^{-1}\cD(\bR^d))$ to \eqref{second eqn}  in 
$$
 \cF^{-1}L_{p,q',t\text{-}loc, x\text{-}\ell oc}\left( (0,T) \times \bR^d, \mathrm{d}t \left(1+\left|a^{ij}(t)\xi^i\xi^j-ib^j(t)\xi^j -c(t)\right|\right)\mathrm{d}\xi \right).
$$
\end{corollary}

\begin{corollary}
Let $p \in [1,\infty]$, $q \in (1,2]$, $\gamma_1, \tilde \gamma_1 \in \left(-\frac{d(p-1)}{p}, \infty\right)$, $\gamma_2, \tilde \gamma_2 \in [0,\infty)$,  $u_0 \in H_{q, in}^{\gamma_1,\gamma_2}(\bR^d)$, and $f \in \fH^{\tilde \gamma_1, \tilde \gamma_2}_{p,q,in,t\text{-}loc}\left( (0,T) \times \bR^d\right)$.
Then there exists a unique weak solution $u$ $($tested by $\cF^{-1}\cD(\bR^d))$ to \eqref{second eqn}  in 
$$
 \cF^{-1}L_{p,q',t\text{-}loc, x\text{-}\ell oc}\left( (0,T) \times \bR^d, \mathrm{d}t \left(1+\left|a^{ij}(t)\xi^i\xi^j-ib^j(t)\xi^j -c(t)\right|\right)\mathrm{d}\xi \right).
$$
\end{corollary}

\begin{corollary}
Let $p \in [1,\infty]$, $q \in (2,\infty]$, $\gamma_1, \tilde \gamma_1 \in \bR$, $\gamma_2, \tilde \gamma_2 \in \left(\frac{d(q-2)}{2(q-1)},\infty \right)$,  $u_0 \in H_{q,out}^{\gamma_1,\gamma_2}(\bR^d)$, and $f \in \fH^{\tilde \gamma_1,\tilde \gamma_2}_{p,q,out,t\text{-}loc}\left( (0,T) \times \bR^d\right)$. 
 Then there exists a unique weak solution $u$ $($tested by $\cF^{-1}\cD(\bR^d))$ to \eqref{second eqn}  in 
$$
 \cF^{-1}L_{p,2,t\text{-}loc, x\text{-}\ell oc}\left( (0,T) \times \bR^d, \mathrm{d}t \left(1+\left|a^{ij}(t)\xi^i\xi^j-ib^j(t)\xi^j -c(t)\right|\right)\mathrm{d}\xi \right).
$$
\end{corollary}
The proofs of these corollaries could be easily obtained from Theorem \ref{main second} with some properties of the weighted Bessel potential spaces developed in Section \ref{fourier section}.
For the details, follow the proofs of Corollaries \ref{p small coro}, \ref{q small coro}, and \ref{p big coro}.

\vspace{3mm}
\textbf{Declarations of interest}
\vspace{3mm}

Declarations of interest: none

\vspace{3mm}
{\bf Data Availability}
\vspace{3mm}

Data sharing not applicable to this article as no datasets were generated or analysed during the current study.

\bibliographystyle{plain}

\begin{thebibliography}{99}

\bibitem{CJH KID 2023}
Choi, Jae-Hwan, and Ildoo Kim
``A weighted-regularity theory for parabolic partial differential equations with time-measurable pseudo-differential operators." 
Journal of Pseudo-Differential Operators and Applications  14, 55 (2023). 

\bibitem{CJH LJB KID 2023}
Choi, Jae-Hwan, Ildoo Kim, and Jin Bong Lee. 
``A regularity theory for an initial value problem with a time-measurable pseudo-differential operator in a weighted $ L_p $-space." arXiv preprint arXiv:2302.07507 (2023).

\bibitem{CJH KJH PDH 2024}
Choi, Jae-Hwan, Jaehoon Kang, and Daehan Park. 
``A Regularity Theory for Parabolic Equations with Anisotropic Nonlocal Operators in $L_q(L_p)$ Spaces." SIAM Journal on Mathematical Analysis 56.1 (2024): 1264-1299.

\bibitem{CJH KID 2023-1}
Choi, Jae-Hwan, and Ildoo Kim. 
``A maximal $L_p$-regularity theory to initial value problems with time measurable nonlocal operators generated by additive processes." Stochastics and Partial Differential Equations: Analysis and Computations 12.1 (2024): 352-415.

\bibitem{CJH 2024}
Choi, Jae-Hwan. 
``A regularity theory for evolution equations with time-measurable pseudo-differential operators in weighted mixed-norm Sobolev-Lipschitz spaces." Potential Analysis 63.2 (2025): 557–577.



\bibitem{HD YL 2023}
Dong, Hongjie, and Yanze Liu. 
``Sobolev estimates for fractional parabolic equations with space-time non-local operators." Calculus of Variations and Partial Differential Equations 62.3 (2023): 96.

\bibitem{Gelfand 1968}
Gel'fand, I.M., and G.E. Shilov. 
Generalized Functions, Volume 2: Spaces of fundamental and generalized functions. Academic press, 1968.

\bibitem{Grafakos 2014}
Grafakos, Loukas. Classical Fourier analysis. Third edition. Springer, 2014.

\bibitem{Grafakos 2014-2}
Grafakos, Loukas. Modern Fourier analysis. Third edition. Springer, 2014.



\bibitem{Hormander 1990}
Hörmander, Lars. The analysis of linear partial differential operators I: Distribution theory and Fourier analysis. Springer, 1990.

\bibitem{Hormander 1990-2}
Hörmander, Lars. The analysis of linear partial differential operators II: Differential operators with constant coefficients. Springer, 1990.

\bibitem{STC2024}
Janreung, Sutawas, Tatpon Siripraparat, and Chukiat Saksurakan. "On $L_ {p}$-Theory for Integro-Differential Operators with Spatially Dependent Coefficients." Potential Analysis 62.1 (2025): 61-100.

\bibitem{KJH PDH 2023}
Kang, Jaehoon, and Daehan Park. 
``An $L_q(L_p)$-theory for space-time non-local equations generated by Lévy processes with low intensity of small jumps." Stochastics and Partial Differential Equations: Analysis and Computations 12.3 (2024): 1439-1491.

\bibitem{KJH PDH 2023-1}
Kang, Jaehoon, and Daehan Park. 
``An $ L_ {q}(L_ {p}) $-regularity theory for parabolic equations with integro-differential operators having low intensity kernels." Journal of Differential Equations 415 (2025): 487-540.

\bibitem{KID SBL KHK 2015}
Kim, Ildoo, Kyeong-Hun Kim, and Sungbin Lim. 
``Parabolic BMO estimates for pseudo-differential operators of arbitrary order." 
Journal of Mathematical Analysis and Applications 427.2 (2015): 557-580.



\bibitem{KID SBL KHK 2016}
Kim, Ildoo, Sungbin Lim, and Kyeong-Hun Kim. 
``An $L_q(L_p)$-theory for parabolic pseudo-differential equations: Calderón-Zygmund approach." 
Potential Analysis 45 (2016): 463-483.


\bibitem{KID KHK 2016}
Kim, Ildoo, and Kyeong-Hun Kim. 
``An $L_p$-theory for stochastic partial differential equations driven by Lévy processes with pseudo-differential operators of arbitrary order." Stochastic Processes and their Applications 126.9 (2016): 2761-2786.

\bibitem{KID 2018}
Kim, Ildoo. 
``An $L_p$-Lipschitz theory for parabolic equations with time measurable pseudo-differential operators." 
Communications on Pure and Applied Analysis 17.6 (2018): 2751-2771.

\bibitem{KID KHK 2018}
Kim, Ildoo, and Kyeong-Hun Kim. 
``On the second order derivative estimates for degenerate parabolic equations." 
Journal of Differential Equations 265.11 (2018): 5959-5983.

\bibitem{KID KKH KPK 2019}
Kim, Ildoo, Kyeong-Hun Kim, and Panki Kim. 
``An $L_p$-theory for diffusion equations related to stochastic processes with non-stationary independent increment." Transactions of the American Mathematical Society 371.5 (2019): 3417-3450.


\bibitem{KID KHK 2023}
Kim, Ildoo, and Kyeong-Hun Kim. 
``A sharp $L_p$-regularity result for second-order stochastic partial differential equations with unbounded and fully degenerate leading coefficients." Journal of Differential Equations 371 (2023): 260-298.

\bibitem{KKH PDH RJH 2021}
Kim, Kyeong-Hun, Park, Daehan, and Ryu, Junhee. 
``An $L_q(L_p)$-theory for diffusion equations with space-time nonlocal operators." Journal of Differential Equations 287 (2021): 376-427.

\bibitem{Kartin 2009}
Schumacher, Katrin. 
``The stationary Navier-Stokes equations in weighted Bessel-potential spaces." 
Journal of the Mathematical Society of Japan 61.1 (2009): 1-38.


\bibitem{Komech 1994}
Komech, Aleksandr Il'ich. 
``Linear partial differential equations with constant coefficients." 
Partial Differential Equations II: Elements of the Modern Theory. Equations with Constant Coefficients (1994): 121-255.



\bibitem{Krylov 1996}
Krylov, Nikolaĭ Vladimirovich. 
Lectures on elliptic and parabolic equations in Holder spaces. 
No. 12. American Mathematical Soc., 1996.


\bibitem{Krylov 1999}
Krylov, Nikolaĭ Vladimirovich. 
``An analytic approach to SPDEs."
 Stochastic partial differential equations: six perspectives 64 (1999).




\bibitem{Krylov 2008}
Krylov, Nikolaĭ Vladimirovich. 
Lectures on elliptic and parabolic equations in Sobolev spaces. 
Vol. 96. American Mathematical Soc., 2008.



\bibitem{Kurtz 1980}
Kurtz, Douglas S. 
``Littlewood-Paley and multiplier theorems on weighted $L^p$ spaces." 
Transactions of the American Mathematical Society 259.1 (1980): 235-254.

\bibitem{RM CP 2017}
Mikulevičius, R., and C. Phonsom. 
``On $L^p$-theory for parabolic and elliptic integro-differential equations with scalable operators in the whole space." Stochastics and Partial Differential Equations: Analysis and Computations 5.4 (2017): 472-519.

\bibitem{RM CP 2019}
Mikulevičius, R., and C. Phonsom. 
``On the Cauchy problem for integro-differential equations in the scale of spaces of generalized smoothness." Potential Analysis 50 (2019): 467-519.

\bibitem{Stroock Varadhan 1997}
Stroock, Daniel W., and Varadhan, SR Srinivasa. 
Multidimensional diffusion processes. 
Vol. 233. Springer Science $\&$ Business Media, 1997.

\bibitem{XZ 2013}
Zhang, Xicheng. 
``\(L^{p}\)-maximal regularity of nonlocal parabolic equations and applications." Annales de l'Institut Henri Poincaré C 30.4 (2013): 573-614.

\end{thebibliography}

\end{document}